\documentclass[a4paper,11pt]{article}
\usepackage{amsthm}

\usepackage[utf8]{inputenc}
\usepackage[T1]{fontenc}
\usepackage[english]{babel}
\usepackage{a4wide}
\usepackage[]{hyperref}
\hypersetup{
    colorlinks=true,       
    linkcolor=red,          
    citecolor=blue,        
    filecolor=magenta,      
    urlcolor=cyan           
}

\usepackage[leqno]{amsmath}
\usepackage{amssymb}
\usepackage{mathrsfs}
\usepackage{amsthm}
\usepackage{amsxtra}
\usepackage{bm}
\usepackage[titletoc]{appendix}

\theoremstyle{plain}
\newtheorem{theo}{Theorem}[section]
\newtheorem{prop}[theo]{Proposition}
\newtheorem{lemm}[theo]{Lemma}
\newtheorem{coro}[theo]{Corollary}

\newtheorem{defi}[theo]{Definition}
\theoremstyle{definition}
\newtheorem{rema}[theo]{Remark}
\newtheorem{nota}[theo]{Notation}

\DeclareMathOperator{\dist}{dist}

\DeclareMathOperator{\supp}{supp}
\DeclareMathOperator{\di}{d}
\DeclareMathOperator{\R}{R}

\DeclareSymbolFont{pletters}{OT1}{cmr}{m}{sl}
\DeclareMathSymbol{s}{\mathalpha}{pletters}{`s}

\def\defn{\mathrel{:=}}

\def\eps{\varepsilon}

\def\la{\left\vert}
\def\lA{\left\Vert}
\def\le{\leq}
\def\les{\lesssim}
\def\leo{}

\def\eps{\varepsilon}
\def\mez{\frac{1}{2}}

\def\ra{\right\vert}
\def\rA{\right\Vert}

\def\tdm{\frac{3}{2}}

\def\xC{\mathbf{C}}
\def\xN{\mathbf{N}}
\def\xR{\mathbf{R}}

\def\xZ{\mathbf{Z}}
\def\cF{ \mathcal{F}}
\def\cN{ \mathcal{N}}

\def\RHS{\text{RHS}}

\def\k{\kappa}

\def\lp{\left(}

\def\rp{\right)}

\newcommand{\bq}{\begin{equation}}
\newcommand{\eq}{\end{equation}}
\newcommand{\bqa}{\begin{eqnarray*}}
\newcommand{\eqa}{\end{eqnarray*}}

\newcommand{\hk}{\hspace*{.15in}}
\numberwithin{equation}{section}

\title{A sharp Cauchy theory for the 2D gravity-capillary waves}
\author{
Quang Huy Nguyen
\footnote{UMR 8628 du CNRS, Laboratoire de Math\'ematiques d'Orsay, Universit\'e Paris-Sud, 91405 Orsay Cedex, France. Email: quang-huy.nguyen@math.u-psud.fr }
\footnote{The author was supported in part by Agence Nationale de la Recherche
  project  ANA\'E ANR-13-BS01-0010-03.}
}
\date{}
\begin{document}
\maketitle
\abstract{
This article is devoted to the Cauchy problem for the 2D gravity-capillary water waves in fluid domains with general bottoms. We prove that the Cauchy problem in Sobolev spaces is uniquely solvable for data $\frac{1}{4}$ derivatives less regular than the energy threshold (obtained by Alazard-Burq-Zuily \cite{ABZ1}), which corresponds to the gain of H\"older regularity of the semi-classical Strichartz estimate for the fully nonlinear system. To obtain this result, we  establish global, quantitative results for the paracomposition theory of Alinhac \cite{Alipara}.
}
\section{Introduction}
\subsection{The equations}
 We consider an incompressible, inviscid fluid with unit density moving in a time-dependent domain  
$$
\Omega = \{(t,x,y) \in[0,T] \times \xR \times \xR:(x, y)\in \Omega_t\} 
$$
where each $\Omega_t$ is a domain located underneath a free surface 
$$
\Sigma_t = \{(x,y)  \times \xR \times \xR: y=\eta(t, x)\} 
$$
and above a fixed bottom $\Gamma=\partial\Omega_t\setminus \Sigma_t$. We make the following assumption on the domain:\\
$\Omega_t$ is the intersection of the haft space 
\[
\Omega_{1,t}= \{(x,y)  \times \xR \times \xR: y=\eta(t, x)\} 
\]
and an open connected set $\Omega_2$ containing a fixed strip around $\Sigma_t$, i.e., there exists $h>0$ such that 
\[
 \{(x,y)  \in \xR \times \xR: \eta(x)-h\le y\le\eta(t, x)\} \subset \Omega_2.
\]
This important assumption prevents the bottom from emerging, or even from coming arbitrarily close to the free surface. The study of water waves without it is an open problem.\\
\hk Assume that the velocity field $v$ admits a potential $\phi:\Omega \to \xR$, i.e, $v=\nabla \phi$. Using the Zakharov formulation, we introduce the trace of $\phi$ on the free surface
$$\psi(t,x)= \phi(t,x,\eta(t,x)).$$ 
 Then $\phi(t, x, y)$ is the unique variational solution of 
\bq\label{phi}
\Delta\phi =0\text{~in}~\Omega_t,\quad \phi(t, x, \eta(t, x))=\psi(t, x).
\eq
The Dirichlet-Neumann operator is then defined by
\[
G(\eta) \psi = \sqrt{1 + \vert \partial_x \eta \vert ^2}
\Big( \frac{\partial \phi}{\partial n} \Big \arrowvert_{\Sigma}\Big)= (\partial_y \phi)(t,x,\eta(t,x)) - \partial_x \eta(t,x) (\partial_x \phi)(t,x,\eta(t,x)).
\]
The gravity water wave problem with surface tension consists in solving the following system of  $\eta,\psi$:
\begin{equation}\label{ww}
\left\{
\begin{aligned}
&\partial_t \eta = G(\eta) \psi,\\
&\partial_t \psi + g\eta-H(\eta)+\mez \vert \partial_x \psi \vert^2 - \mez \frac{(\partial_x \eta \partial_x \psi + G(\eta)\psi)^2}{1+ \vert \partial_x \eta \vert^2}=0
\end{aligned}
\right.
\end{equation}
where $H(\eta)$ is the mean curvature of the free surface:
\[
H(\eta)=\partial_x\left( \frac{\partial_x\eta}{\sqrt{1+|\partial_x\eta|^2}}\right).
\]
It is important to introduce the vertical and horizontal components of the velocity, which can be expressed in terms of $\eta$ and $\psi$:
\begin{equation}\label{BV}
B = (v_y)\arrowvert_\Sigma = \frac{ \partial_x \eta  \partial_x \psi + G(\eta)\psi} {1+ \vert \partial_x \eta \vert^2},\quad V= (v_x)\arrowvert_\Sigma  =\partial_x \psi - B \partial_x \eta.
 \end{equation}
\subsection{The problem}
 \hk Our purpose is to study the Cauchy problem for system \eqref{ww} with {\it sharp Sobolev regularity for initial data}.  For previous results on the Cauchy problem, we refer to the works of  Yosihara \cite{Yosihara},  Coutand- Shkoller \cite{CS},  Shatah-Zeng \cite{SZ1, SZ2, SZ3},  Ming-Zhang \cite{MiZh} for sufficiently smooth solutions; see also the works of Wu \cite{WuJAMS, WuInvent}, Lannes \cite{LannesJAMS} for gravity waves without surface tension. In term of regularity of initial data, the work of Alazard-Burq-Zuily \cite{ABZ1} reached an important threshold: local wellposedness as long as the velocity field is Lipschitz (in term of Sobolev embeddings) up to the free surface. More precisely, this corresponds to data (in view of the formula \eqref{BV})
\[
(\eta_0, \psi_0)\in H^{s+\mez}(\xR^d)\times H^s(\xR^d),\quad s>2+\frac{d}{2}.
\]
This is achieved by the energy method after reducing the system to a single quasilinear equation using a paradifferential calculus approach. However, observe that the linearized of \eqref{ww} around the rest state $(0, 0)$ reads
$$\partial_t\Phi+i\la D\ra^\tdm\Phi=0,\quad \Phi=\la D\ra^\mez\eta+i\psi$$
which is dispersive and enjoys the following Strichartz estimate with a gain of $\frac{3}{8}$ derivatives
\bq\label{Str:opt}
\lA \Phi\rA_{L^4_tW_x^{\sigma-\frac{1}{8},\infty}}\le C_\sigma \lA \Phi\arrowvert_{t=0}\rA_{H^\sigma_x},\quad \forall \sigma\in \xR.
\eq
Therefore, one may hope that the fully nonlinear system \eqref{ww} is also dispersive and enjoys similar Strichartz estimates. Indeed, this is true and was first proved by Alazard-Burq-Zuily \cite{ABZ2}:  any solution 
\bq\label{threshold:ABZ}
(\eta, \psi)\in C^0([0, T]; H^{s+\mez}(\xR)\times H^s(\xR)),\quad s>2+\frac{1}{2}
\eq
satisfies
\bq\label{Str:ABZ}
 (\eta, \psi)\in L^4([0, T]; W^{s+\frac{1}{4},\infty}(\xR)\times W^{s-\frac{1}{4},\infty}(\xR)).
\eq
Comparing to the classical (full) Strichartz estimate \eqref{Str:opt}, the estimate \eqref{Str:ABZ} exhibits a loss of $\frac{1}{8}$ derivatives and is called the {\it semi-classical Strichartz estimate}. This terminology comes from the work \cite{BGT1} for Schrodinger equations on manifolds. In fact, slightly earlier in \cite{CHS} the same Strichartz estimate was obtained for the 2D gravity-capillary water waves under another formulation. We also refer to \cite{Huy} for another proof of \eqref{Str:ABZ} and the semi-classical Strichartz estimate for 3D waves.\\
\hk It is known, for instance from the works of Bahouri-Chemin \cite{BaCh} and Tataru \cite{TataruNS}, that for dispersive PDEs, Strichartz estimates can be used to improve the Cauchy theory for data that are less regular than the one obtained merely via the energy method. We refer to \cite{BCD}, Chapter 9 for an expository presentation of quasilinear wave equations. Our aim is to proceed such a program for the gravity-capillary water waves system \eqref{ww}. For pure gravity water waves, this was  considered by Alazard-Burq-Zuily \cite{ABZ4}. Coming back to our system \eqref{ww}, from the semi-classical Strichartz estimate \eqref{Str:ABZ} for $s>2+\mez$  it is natural to ask\\
{\bf Q:}~{\it Does the Cauchy problem for \eqref{ww} have a unique solution for data }
\[
(\eta_0, \psi_0)\in H^{s+\mez}(\xR^d)\times H^s(\xR^d),\quad s>2+\frac{1}{2}-\frac{1}{4}=\frac{9}{4}?.
\]
In the previous joined work \cite{NgPo2}, we proved an "intermediate" result for  $s>2+1/2-3/20$ in 2D case (together with a result for 3D case), which asserts that water waves can still propagate starting from  {\it non-Lipschitz velocity} (up to the free surface) (see \cite{ABZ4} for the corresponding result for vanishing surface tension). Our contribution in this work is to prove an affirmative answer for question {\bf Q}.  \\
\hk Let us give an outline of the proof.  In \cite{NgPo1}, using a paradifferential approach we  reduced the system \eqref{ww} to a single dispersive equation as follows: assume that for some $s>r>2$
\bq\label{intro:reg}
(\eta, \psi)\in C^0([0, T]; H^{s+\mez}(\xR)\times H^s(\xR))\cap L^4([0, T]; W^{r+\mez,\infty}(\xR)\times W^{r,\infty}(\xR))
\eq
then after paralinearization and symmetrization, \eqref{ww} is reduced to the following equation of a complexed-valued unknown $\Phi$
\bq\label{intro:eqPhi}
\partial_t\Phi+T_V\partial_x\Phi+iT_{\gamma}\Phi=f
\eq
for some paradifferential symbol $\gamma\in \Sigma^{3/2}$ and $f(t)$ satisfies the {\it tame} estimate 
\[
		\lA f(t)\rA_{H^s}\leq\cF\lp\lA\eta(t)\rA_{H^{s+\mez}},\lA\psi(t)\rA_{H^s}\rp\lp1+\lA\eta(t)\rA_{W^{r+\mez,\infty}}+\lA\psi(t)\rA_{W^{r,\infty}}\rp.
\]
Such a reduction was first obtained in \cite{ABZ1} for solution at the energy threshold \eqref{threshold:ABZ}. Observe that the relation $s>r>2$ exhibits a gap of $\mez$ derivatives in view of the Sobolev embedding from $H^s$ to $C_*^{s-\mez}$ (see Definition \ref{spaces}). Having in hand the blow-up criterion and the contraction estimate in \cite{NgPo1} at the regularity \eqref{intro:reg}, the main difficulty in answering question {\bf Q} is to prove the semi-classical Strichartz estimate for solution $\Phi$ to \eqref{intro:eqPhi}. Comparing to the Strichartz estimates in \cite{NgPo2} we remark that the semi-classical gain in \cite{ABZ2} (when $s>2+\mez$) was achieved owing to the fact that when $d=1$ one can further reduce \eqref{intro:eqPhi} to an equation where the highest order term $T_\gamma$ becomes the Fourier multiplier $|D_x|^\tdm$:
\bq\label{intro:reduce}
\partial_t\widetilde\Phi+T_{\widetilde V}\partial_x\widetilde\Phi+i|D_x|^\tdm\widetilde\Phi=\widetilde f.
\eq
This reduction is proceeded by means of the {\it paracomposition} of Alinhac \cite{Alipara}. Here, we shall see that in our case we need a more precise paracomposition result for 2 purposes: (1) deal with rougher functions and (2) obtain quantitative estimates. This will be the content of section \ref{section:paracomposition}. After having \eqref{intro:reduce} we show in section \ref{section:semi} that the method in \cite{ABZ2} can be adapted to our lower regularity level to derive the semi-classical Strichartz estimate with an arbitrarily small $\eps$ loss.
\subsection{Main results}
  Let us introduce the Sobolev norm and the Strichartz norm for solution $(\eta, \psi)$ to the gravity-capillary system \eqref{ww}:
\begin{align*}
&M_\sigma(T)=\Vert (\eta, \psi)\Vert_{L^{\infty}([0, T]; H^{\sigma+\mez}(\xR)\times H^\sigma(\xR))}, \quad M_\sigma (0)=\Vert (\eta, \psi)\arrowvert_{t=0}\Vert_{H^{\sigma+\mez}(\xR)\times H^\sigma(\xR)},\\
&N_\sigma(T)=\Vert (\eta, \psi)\Vert_{L^4([0, T]; W^{\sigma+\mez, \infty}(\xR)\times W^{\sigma, \infty}(\xR))}.
\end{align*}
Our first result concerns the semi-classical Strichartz estimate for system \eqref{ww}.
\begin{theo}\label{theo:Strichartz}
Assume that $(\eta, \psi)$ is a solution to \eqref{ww} with 
\begin{equation}\label{regularity:a}
\left\{
\begin{aligned}
&(\eta, \psi)\in C^0([0, T]; H^{s+\mez}(\xR)\times H^s(\xR))\cap L^4([0, T]; W^{r+\mez,\infty}(\xR)\times W^{r,\infty}(\xR)),\\
&s>r>\tdm+\mez.
\end{aligned}
\right.
 \end{equation}
and 
\bq\label{distance}\inf_{t\in [0, T]}\dist(\eta(t), \Gamma) \ge h>0.\eq
Then, for any $\mu<\frac{1}{4}$ there exists a non-decreasing function $\cF$ independent of $(\eta, \psi)$ such that
\bq\label{strichartz:original}
N_{s-\mez+\mu}(T)\le \cF(M_s(T)+N_r(T)).
\eq
\end{theo}
As a consequence of Theorem \ref{theo:Strichartz} and the energy estimate in  \cite{NgPo1} we obtain a closed {\it a priori} estimate for the mixed norm 
$M_s(T)+N_r(T)$.
\begin{theo}\label{theo:apriori}
 Assume that $(\eta, \psi)$ is a solution to \eqref{ww} and satisfies conditions \eqref{regularity:a}, \eqref{distance} with 
\[
2<r<s-\mez+\mu,\quad \mu<\frac{1}{4},\quad h>0.
\]
Then there exists a non-decreasing function $\cF$ independent of $(\eta, \psi)$ such that
\[
M_s(T)+N_r(T)\le \cF\Big(\cF(M_s(0))+T\cF(M_s(T)+N_r(T))\Big).
\]
\end{theo}
Finally, we obtain a Cauchy theory for the gravity-capillary system \eqref{ww} with initial data $\frac{1}{4}$ derivatives less regular than the energy threshold  in \cite{ABZ1}.
\begin{theo}\label{theo:Cauchy}
Let $\mu<\frac{1}{4}$ and $2<r<s-\mez+\mu$. Then for any $(\eta_0, \psi_0)\in H^{s+\mez}(\xR)\times H^s(\xR)$ satisfying $\dist(\eta_0, \Gamma) \ge h>0$, there exists  $T>0$ such that the gravity-capillary waves system \eqref{ww} has a unique solution $(\eta, \psi)$ in
\[
 L^\infty([0, T]; H^{s+\mez}(\xR)\times H^s(\xR))\cap L^4([0, T]; W^{r+\mez,\infty}(\xR)\times W^{r,\infty}(\xR)).
\]
Moreover, we have 
\[
(\eta,\psi)\in C^0\lp[0, T]; H^{s_0+\mez}\times H^{s_0}\rp,\quad\forall s_0<s
\]
and 
\[
\inf_{t\in [0, T]}\dist(\eta(t), \Gamma)>\frac{h}{2}.
\]
\end{theo}
\begin{rema}
The proof of Theorem \ref{theo:Cauchy} shows that for each $\mu<\frac{1}{4}$ the existence time $T$ can be chosen uniformly for data $(\eta_0, \psi_0)$ lying  in a bounded set of $H^{s+\mez}(\xR)\times H^s(\xR)$ and the fluid depth $h$ lying in a bounded set of $(0, +\infty)$.
\end{rema}
\begin{rema}
We do not know yet if the semi-classical gain is optimal for solutions at the regularity \eqref{regularity:a}. However, some remarks can be made as follows. On the one hand, if one proves Strichartz estimate for \eqref{intro:eqPhi} then there is a nontrivial geometry of the symbol $\gamma$, for which trapping may occur. According to \cite{BGT1} (see Section $4$), at least in the case of spheres, the semi-classical Strichartz estimates are optimal. On the other hand, if one wishes to eliminate the geometry by making changes of variables, then as we shall see in Proposition \ref{parareduce} and Remark \ref{rema:mez}, there will appear a loss of $\mez$ derivatives in the source term, which turns out to be optimal for the semi-classical Strichartz estimate (see the proof of Theorem \ref{strich2}).
\end{rema}
\begin{rema}
The linearized system of \eqref{ww} in dimension 2 ($\eta, \psi:\xR^2\to \xR$) enjoys the semi-classical Strichartz estimate with a gain $\mez$ derivatives (see \cite{Huy}). It was proved in \cite{Huy} that the same estimate holds for the nonlinear system \eqref{ww} when 
\[
(\eta, \psi)\in C^0([0, T]; H^{s+\mez}(\xR^2)\times H^s(\xR^2)),~s>\frac{5}{2}+1.
\]
If the preceding regularity could be improved to ($\mez$ derivative
\[
\left\{
\begin{aligned}
&(\eta, \psi)\in C^0([0, T]; H^{s+\mez}\times H^s)\cap L^2([0, T]; W^{r+\mez,\infty}\times W^{r,\infty}),\\
&s-\mez>r>2,
\end{aligned}
\right.
\]
the results in \cite{NgPo1} would imply a Cauchy theory (see the proof of Theorem \ref{theo:Cauchy}) with initial surface 
\[
\eta_0\in H^{s+\mez}(\xR^2),~s>2+\mez,
\]
which has the lowest Sobolev regularity to ensure that the initial surface has bounded curvature (see the Introduction of \cite{NgPo2}).
\end{rema}
{\bf Acknowledgment}\\
This work was partially supported by the labex LMH through the grant no ANR-11-LABX-0056-LMH in the "Programme des Investissements d'Avenir". I would like to sincerely thank my advisor, Nicolas Burq for suggesting me the problem and sharing many fruitful discussions together with Claude Zuily, whom I warmly thank. Finally, I thank Thomas Alazard and Thibault de Poyferr\'e for direct and indirect discussions.
\section{Preliminaries on dyadic analysis}
\subsection{Dyadic  partitions} \label{subsection:dyadic}
 Our analysis below is sensitive with respect to the underlying dyadic partition of $\xR^d$. These partitions are constructed by using the cut-off functions given in the following lemma.
\begin{lemm}\label{lem1:phi}
For every $n\in \xN$, there exists $\phi_{(n)}\in C^\infty(\xR^d)$ satisfying 
\bq\label{phi:cd1}
\phi_{(n)}(\xi)=
\begin{cases}
1,\quad \text{if}~|\xi|\le 2^{-n},\\
0,\quad \text{if}~|\xi|> 2^{n+1},
\end{cases}
\eq 
\bq\label{phi:cd2}
\forall (\alpha, \beta)\in \xN^d\times \xN^d, \exists C_{\alpha, \beta}>0, \forall n\in \xN, \lA x^\beta\partial^\alpha\phi_{(n)}(x)\rA_{L^1(\xR^d)}\le  C_{\alpha, \beta}.
\eq
\end{lemm}
We leave the proofs of the results in this paragraph to Appendix 2. In fact, to guarantee condition \eqref{phi:cd2} we choose $\phi_{(n)}$ with support in a ball of size $2^{-n}+c$ for some $c>0$.\\
We shall skip the subscript $(n)$ and denote $\phi\equiv\phi_{(n)}$ for simplicity.
Setting
\[
\phi_k(\xi)=\phi(\frac{\xi}{2^k}),~ k\in \xZ, \quad \varphi_0=\phi=\phi_0,~\varphi=\chi-\chi_{-1},~ \varphi_k=\phi_k-\phi_{k-1}=\varphi(\frac{\cdot}{2^k}),~k\ge 1,
\]
 we see that 
\bq\label{supp:dyadic}
\begin{aligned}
&\supp\varphi_0\subset \mathcal{C}_0(n)\defn \{\xi\in \xR^d: |\xi|\le 2^{n+1}\}\\
&\supp\varphi\subset \mathcal{C}(n)\defn \{\xi\in \xR^d: 2^{-(n+1)}<|\xi|\le 2^{n+1}\}\\
&\supp\varphi_k\subset \mathcal{C}_k(n)\defn\{\xi\in \xR^d: 2^{k-(n+1)}<|\xi|\le 2^{k+(n+1)}\},~\forall k \ge 1.
\end{aligned}
\eq
Observing also that with 
\[
N_0\defn 2(n+1)
\]
we have
\[
\mathcal{C}_j(n)\cap \mathcal{C}_k(n)=\emptyset\quad \text{if}~ |j-k|\ge N_0.
\]
\begin{defi}\label{SD}
For every $\phi\equiv\phi_{(n)}$, defining the following Fourier multipliers
\[
\widehat{S_ku}(\xi)=\phi_k(\xi)\hat u(\xi), ~ k\in \xZ, \quad  \widehat{\Delta_ku}(\xi)=\varphi_k(\xi)\hat u(\xi),~k\ge 0.
\]
Denoting $u_k=\Delta_k u$ we obtain a dyadic partition 
\bq\label{partition}
u=\sum_{p=0}^\infty u_p,
\eq
where $n$ shall be called the size of this partition. Remark that with the notations above, there hold
\[
\Delta_0=S_0 ,\quad \sum_{p=0}^q\Delta_p=S_q,\quad S_{q+1}-S_q=\Delta_{q+1}.
\]
\end{defi}
Throughout this article, whenever $\xR^d$ is equipped with a fixed dyadic partition, we always define the Zygmund-norm (see Definition \ref{spaces}) of distributions on $\xR^d$ by means of this partition.\\
To  prove our paracomposition results we need to choose a particular size $n=n_0$, tailored to the diffeomorphism, in Proposition \ref{bigC} below, whose proof requires uniform bounds for the norms of the operators $S_j,~\Delta_j$ in Lebesgue spaces and H\"older spaces,  with respect to the size $n$. This fact in turn stems from property \eqref{phi:cd2} of $\phi_{(n)}$. 
\begin{lemm}\label{boundSj}
1. For every $\alpha\in \xN^d$, there exists $C_\alpha>0$ independent of $n$ such that
\[
 \forall j,~\forall 1\le p\le q\le \infty,~ \lA \partial^\alpha S_ju\rA_{L^q(\xR^d)}+ \lA \partial^\alpha \Delta_ju\rA_{L^q(\xR^d)}\le C_\alpha 2^{j(|\alpha|+\frac{d}{p}-\frac{d}{q})}\lA u\rA_{L^p(\xR^d)}. 
\]
\\
2. For every $\mu\in (0,\infty)$, there exists $M>0$ independent of $n$  such that 
\[
\forall j \in \xN,~\forall u\in W^{\mu,\infty}(\xR^d): ~\lA \Delta_ju\rA_{L^\infty(\xR^d)}\le M2^{-j\mu}\lA u\rA_{W^{\mu, \infty}(\xR^d)}.
\]
\end{lemm}
As a consequence of this lemma, one can examine the proof of Proposition $4.1.16$, \cite{MePise} to have
\begin{lemm}\label{HZ}
Let $\mu>0,~\mu\notin \xN$. Then there exists a constant $C_\mu$ independent of $n$, such that  for any $u\in W^{\mu,\infty}(\xR^d)$ we have
\[
\frac{1}{C_\mu}\lA u\rA_{W^{\mu,\infty}(\xR^d)}\le \lA u\rA_{C^\mu_*}\le C_\mu\lA u\rA_{W^{\mu,\infty}(\xR^d)}.
\]
Moreover, when $\mu\in \xN$ the second inequality still holds.
\end{lemm}
By virtue of Lemma \ref{HZ}, we shall identify $W^{\mu,\infty}(\xR^d)$ with $C_*^\mu (\xR^d)$ whenever $\mu >0,~\mu\notin \xN$, regardless of the size $n$. \\
For very $j\ge 1$, the reverse estimates for $\Delta_j$ in Lemma \ref{boundSj} 1. hold (see Lemma $2.1$, \cite{BCD})
\begin{lemm}\label{auxi:1}
Let $\alpha\in \xN^d$. Then there exists $C_\alpha(n)>0$ such that for every $1\le p\le \infty$ and every $j\ge 1$, we have
\[
\lA \Delta_j u\rA_{L^p(\xR^d)}\le C_\alpha(n)2^{-j|\alpha|}\lA \partial^\alpha\Delta_j u\rA_{L^p(\xR^d)}.
\]
\end{lemm}
Applying the previous lemmas yields
\begin{lemm}\label{auxi:2}
1. Let $\mu>0$. Then for every $\alpha\in \xN^d$ there exists $C_\alpha>0$ such that
\bq\label{auxi:21}
\forall v\in C_*^{\mu}(\xR^d),~\forall p\ge 0,~\lA \partial^{\alpha}(S_pv)\rA_{L^\infty}\le 
\begin{cases}
C_\alpha 2^{p(|\alpha|-\mu)}\lA \partial^\alpha v\rA_{C_*^{\mu-|\alpha|}},\quad &\text{if}~|\alpha|>\mu\\
C_\alpha\lA \partial^\alpha v\rA_{L^\infty}, \quad &\text{if}~|\alpha|<\mu\\
C_\alpha p\lA v\rA_{C_*^\mu}, &\text{if}~|\alpha|=\mu.
\end{cases}
\eq
2. Let  $\mu<0$. Then for every $\alpha\in \xN^d$ there exists $C_\alpha>0$ such that
\bq\label{auxi:21'}
 \forall v\in C_*^\mu(\xR^d),~~\forall p\ge 0,~\lA \partial^{\alpha}(S_pv)\rA_{L^\infty}\le C_\alpha 2^{p(|\alpha|-\mu)}\lA v\rA_{C_*^\mu}.
\eq
3. Let $\mu>0$. Then there exists $C(n)>0$  such that for any $v\in \mathcal{S}'$ with $\nabla v\in C_*^{\mu-1}(\xR^d)$ we have
 \bq\label{auxi:22}
\lA v-S_pv\rA_{L^\infty}\le C(n)2^{-p\mu}\lA \nabla v\rA_{C_*^{\mu-1}}.
\eq
\end{lemm}
\subsection{On the para-differential operators}
In this paragraph we clarify the choice of two cutoff functions $\chi$ and $\psi$ appearing in the definition of paradifferential operators \ref{defi:oper} in accordance with the dyadic partitions above. Given a dyadic system of size $n$ on $\xR^d$, define
\bq\label{choose:chi}
\chi(\eta, \xi)=\sum_{p= 0}^\infty \phi_{p-N}(\eta)\varphi_p(\xi)
\eq
with $N=N(n)\gg n$ large enough. It is easy to check that the so defined $\chi$ satisfies \eqref{cutoff-chi:1} and \eqref{cutoff-chi:2}. Plugging \eqref{choose:chi} into \eqref{eq.para} gives 
\begin{align*}
T_a u(x)&=\sum_{p=0}^\infty\int\int e^{i(\theta+\eta)x}\phi_{p-N}(\theta)\hat a(\theta, \eta)\varphi_p(\eta)\psi(\eta)\hat u(\eta)d\eta d\theta\\
&=\sum_{p=0}^\infty S_{p-N}(a)(x, D)(\psi\varphi_p)(D)u(x).
\end{align*}
Notice that for any $p\ge 1$ and  $\eta\in \supp \varphi_p$ we have $|\eta|\ge 2^{-n}$. Choosing  $\psi$ (depending on $n$) verifying
\[
\psi(\eta)= 1 \quad\text{if}~|\eta|\ge 2^{-n}, \quad \psi(\eta)= 0 \quad\text{if}~|\eta|\le 2^{-n-1}
\]
gives
\bq\label{T_au:expand}
T_au(x)=\sum_{p=1}^\infty S_{p-N}a(x, D)\Delta_pu(x)+S_{-N}(a)(\psi \varphi_0)(D)u(x).
\eq
Defining the "truncated paradifferential operator" by
\bq\label{Tdot}
\dot T_au=\sum_{p=1}^\infty S_{p-N}a\Delta_pu.
\eq
then the difference $T_a-\dot T_a$ is a smoothing operator in the following sense: if for some $\alpha\in \xN^d$, $\partial^\alpha u\in H^{-\infty}$ then $(T_a-\dot T_a)u\in H^\infty$ since $\psi\varphi_0$ is supported away from $0$. We thus can utilize the symbolic calculus Theorem \ref{theo:sc} for the truncated paradifferential operator $\dot T_au$ when working on distributions $u$ as above. The same remark applies to the paraproduct $TP_a$ defined in \eqref{def:paraproduct}. In general, smoothing remainders can be ignored in applications. However, the be precise in constructing the abstract theory we decide to distinguish these objects.
\begin{defi}
For $v,~w\in \mathcal{S}'$ define the truncated remainder
\[
\dot R(v, w)=\dot T_v w-\dot T_wv.
\]
Comparing to the Bony's remainder $R(v, w)$ defined in \eqref{R:Bony}, $\dot R(v, w)$ satisfies
\bq\label{RRdot}
\dot R(v, w)=R(w, w)+\sum_{k=1}^N\left (S_{k-N}v\Delta_kw+S_{k-N}w\Delta_kv\right).
\eq
\end{defi}
\begin{rema}
The relation \eqref{RRdot} shows that the estimates \eqref{Bony1}, \eqref{Bony2}, \eqref{Bony3} are valid for $\dot R$.
\end{rema}
\subsection{Choice of dyadic partitions}
 \hk Let $\k:\xR^d_1\to \xR^d_2$ be a diffeomorphism satisfying
\[
\begin{aligned}
 \exists \rho>0,~\partial_x \k\in C_*^{\rho}(\xR^d_1),\\
\exists m_0>0,~\forall x\in \xR_1^d,~\la \det \k'(x)\ra \ge m_0.
\end{aligned}
\]
We equip on $\xR_2^d$ a dyadic partition \eqref{partition} with $n=0$ and on $\xR^d_1$  the one with $n=n_0$ large enough as given the next proposition.
\begin{prop}\label{bigC}
 Let $p, q, j\ge 0$. For $\eps_0>0$ arbitrarily small, there exist $\mathcal{F}_1, \mathcal{F}_2$ nonnegative such that with 
\[
n_0=\mathcal{F}_1(m_0, \lA \k'\rA_{L^\infty})\in \xN, \quad p_0=\mathcal{F}_2(m_0, \lA \k'\rA_{C_*^{\eps_0}})\in \xN,
\]
and $N_0=2(n_0+1)$, we have
\[
\la S_p\k'(y)\eta-\xi\ra \ge 1, 
\]
\[
~\text{if~either } (\xi, \eta)\in \mathcal{C}_j(n)\times \mathcal{C}_q(1), \quad p\ge 0,~j\ge q+N_0+1
\]
\[
\text{or} ~ |\xi|\le 2^{j+(n+1)}, \eta\in \mathcal{C}_q(1), \quad p\ge p_0,~0\le j\le q-N_0-1.
\]
\end{prop}
\begin{proof}
 We consider 2 cases:\\
\hk $(i)$ $p\ge 0, j\ge q+N_0+1$. Using Lemma \ref{boundSj} we get for some constant $M_1=M_1(d)$
\begin{align*}
\la S_p\k'(y)\eta-\xi\ra &\ge |\xi|-|S_p\k'(y)\eta|\ge 2^{q+1}(2^{j-q-1-(n+1)}-M_1\lA \k'\rA_{L^\infty})\\
&\ge 2^{N_0-(n+1)}-M_1\lA \k'\rA_{L^\infty}\ge 2^{n+1}-M_1\lA \k'\rA_{L^\infty}.
\end{align*}
We choose $n\ge \left[ \log_2(M_1\lA \k'\rA_{L^\infty}+1)\right]$ to have $\la S_p\k'(y)\eta-\xi\ra \ge 1$.\\
\hk $(ii)$ $j\le q-N_0-1$. Note that for any $\eps_0>0$, owing to the estimate \eqref{auxi:22}, there is a constant $M_2=M_2(d, \eps_0)$ such that
\[
|\k'-S_p\k'|\le M_22^{-p\eps_0}\lA \k'\rA_{C_*^{\eps_0}}
\]
and consequently, for some increasing function $\cF$
\bq\label{S_pk'}
\la \det S_p\k'\ra\ge \la \det\k'\ra-M_22^{-p\eps_0}\cF(\lA \k'\rA_{C_*^{\eps_0}})\ge \frac{m_0}{2}
\eq
if we choose 
\bq\label{p0}
p\ge p_0\defn\frac{1}{\eps_0}\left[\ln \big(\frac{2M_2}{m_0}\cF(\lA \k'\rA_{C_*^{\eps_0}})\big)\right]+1.
\eq
We then use the inverse formula with adjugate matrix $(S_p\k')^{-1}=\frac{1}{\det S_p\k'}\text{adj}(S_p\k')$ when $d\ge 2$ to get for all $d\ge 1$,
\[
\la (S_p\k')^{-1}\ra\le  \frac{2}{m_0}\left(1+C(d)\lA \k'\rA_{L^\infty}^{d-1}\right):=K.
\]
It follows that\begin{align*}
\la S_p\k'(y)\eta-\xi\ra &\ge \frac{1}{K}|\eta|-|\xi|\ge 2^{j+n+1}(\frac{1}{K}\frac{2^{q-2-j-(n+1)}}{2}-1)\\
&\ge \frac{1}{K}2^{N_0-1-(n+1)}-1\ge \frac{1}{K}2^n-1.
\end{align*}
Choosing $n\ge [1+\ln K]+1$ lead to $\la S_p\k'(y)\eta-\xi\ra \ge 1$. The Proposition then follows with $p_0$ as in \eqref{p0} and 
\[
n_0= \left[ \log_2(M_1\lA \k'\rA_{L^\infty}+1)\right]+[1+\ln K]+1. \]
\end{proof}
\section{Quantitative and global paracomposition results}\label{section:paracomposition}
\subsection{Motivations}
\hk The semi-classical Strichartz estimate for solutions to \eqref{intro:eqPhi} proved in \cite{ABZ2} relies crucially on the fact that one can make a para-change of variable to convert the highest order term $T_\gamma$ to the simple Fourier multiplier $|D_x|^\tdm$.  This is achieved  by using the theory of {\it paracomposition} of  Alinhac  \cite{Alipara}. Let us recall here the main features of this theory:
\begin{theo}
\label{Alinhac}
Let $\Omega_1, \Omega_2$ be two open sets in $\xR^d$ and $\k:\Omega_1\to \Omega_2$ be a diffeomorphism of class $C^{\rho+1}, \rho>0$. Then, there exists a linear operator $\k^*_A:\mathcal{D}'(\Omega_2)\to \mathcal{D}'(\Omega_1)$ having the following properties:\\
1. $\k^*_A$ applies $H^s_{loc}(\Omega_2)$ to $H^s_{loc}(\Omega_1)$  for all $s\in \xR$.\\
2. Assume that $\k\in H^{r+1}_{loc}$ with $r>\frac{d}{2}$. Let $u\in H^s_{loc}(\Omega_2)$ with $s>1+\frac{d}{2}$. Then we have 
\bq\label{Ali-linearize}
\k^*_Au=u\circ\k-T_{u'\circ\k}\k+\R
\eq
with $R\in H^{r+1+\eps}_{loc}(\Omega_1),~\eps=\min (s-1-\frac{d}{2}, r+1-\frac{d}{2})$.\\
3. Let $h\in \Sigma^m_{\tau}$. There exists $h^*\in \Sigma^m_\eps$ with $\eps=\min (\tau, \rho)$ such that 
\bq\label{Ali-conjutate}
\k^*_AT_hu=T_{h^*}\k^*_Au+\R u
\eq
where $R$ applies $H^s_{loc}(\Omega_2)$ to  $H^{s-m+\eps}_{loc}(\Omega_2)$ for all $s\in \xR$. Moreover, the symbol $h^*$ can be computed explicitly as in the  classical pseudo-differential calculus (see Theorem \ref{theo:conjugation} below).
\end{theo}
Let $u\in \mathcal{E}'(\Omega_2)$, $\supp u= K$, $\psi\in C_0^\infty(\Omega_1)$, $\psi=1$ near $\k^{-1}(K)$.  The exact definition of $\k^*_A$ in \cite{Alipara} is given by 
\bq\label{Ali-k*}
\k^*_Au=\sum_{p=0}^{\infty}\sum_{j=p-N_0}^{p+ N_0}\widetilde\Delta_j(\psi\Delta_pu\circ \k)
\eq
for some $N_0\in \xN$ and some dyadic partition $1=\sum\widetilde\Delta_j$ depending on $\k, K$.\\
\hk This local theory was applied successfully by Alinhac in studying the existence and interaction of simple waves for nonlinear PDEs. The equation we have in hand is \eqref{intro:eqPhi}. More generally, let us consider the paradifferential equation
\bq\label{para-motivation}
\partial_tu+Nu+iT_hu=f,\quad (t,x)\in (0, T)\times\xR,
\eq
where $u$ is the unknown, $T_h$ is a paradifferential operator of order $m>0$ and $Nu$ is the lower order part. Assume further that $h(x,\xi)=a(x)|\xi|^m, a(x)>0$. We seek for a change of variables to convert $T_h$ to the Fourier multiplier $|D_x|^m$. Set 
\[
\chi(x)=\int_0^x a^{-\frac{1}{m}}(y)\di y
\]
and let $\k$ be the inverse map of $\chi$. Suppose that a global version of Theorem \ref{Alinhac} were constructed then part $3.$ would yield 
\[
\k^*_AT_hu=T_{h^*}\k^*_Au+\R u
\]
and the principle symbol of $h^*$ (as in the case of classical pseudo-differential calculus) would be indeed $|\xi|^m$. However, to be rigorous  we have to study the following points\\
{\bf Question 1:} A global version of Theorem 1, that is, in all statements  $H^s_{loc}(\xR)$ is replaced  by $H^s(\xR)$.\\
{\bf Question 2:} If the symbol $h$ is elliptic: $a(x)\ge c>0$  then the regularity condition $\k\in C^{\rho+1}(\xR)$ is violated for
\[
\k'(x)=\frac{1}{\chi'(\k(x))}=a^{\frac{1}{m}}(\k(x))\ge c^{\frac{1}{m}}.
\]
So, we need a result without any regularity assumption on $\k$ but only on its derivatives; in other words, only on the high frequency part of $\k$.\\
\hk Assume now that  equation \eqref{para-motivation} is quasilinear: $a(t, x)=F(u)(t, x)$.  We  then have to consider for each $t$, the diffeomorphism
\[
\chi_t( x)=\int_0^x F(u)^{-m}(t, y)\di y
\]
and this gives rise to the following problem\\
{\bf Question 3:} When one  conjugates \eqref{para-motivation} with $\k^*_A$ it is requisite to compute 
\bq\label{dtk*}
\partial_t (\k^*_Au)=\k^*_A(\partial_t u)+\R.
\eq
This would be complicated in view of the original definition \eqref{Ali-k*}. In \cite{ABZ2} the authors overcame this by using Theorem \ref{Alinhac} $2.$ as a new definition of the paracomposition:
\[
\k^*u=u\circ \k-T_{u'\circ \k}\k.
\]
For this purpose, we need to make use of part 2. of Theorem \ref{Alinhac} to estimate the remainder $k_A^*(T_hu)-k^*(T_h u)$. This in turn requires $T_h u\in H^s$ with $s>1+\frac{d}{2}$ or $u\in H^s$ with $s>m+1+\frac{d}{2}$, which is not the case if one wishes to study the optimal Cauchy theory for \eqref{para-motivation} since we are alway $1$-derivative above the "critical index" $\mu=m+\frac{d}{2}$.\\
{\bf Question 4:} If a linearization result as in part 2. of Theorem \ref{Alinhac} for $u\in H^s(\xR)$ with $s<1+\frac{d}{2}$ holds?.\\
Let's suppose that all the above questions can be answered properly. After conjugating \eqref{para-motivation} with $\k^*$ the equation satisfied by $u^*\defn\k^*u$ reads
\bq\label{eq:k*u}
\partial_tu^*+Mu^*+|D_x|^m u^*=\k^*f+g
\eq
where $g$ contains all the remainders in Theorem \ref{Alinhac} $2., 3.$ and  in \eqref{dtk*}.\\
To prove Strichartz estimates for \eqref{eq:k*u}, we need to control $g$ in $L^p_tL^q_x$ norms, which in turns requires {\it tame estimates} for $g$. It is then crucial to have quantitative estimates for the remainders appearing in $g$ and hence quantitative results for the paracomposition. 
\subsection{Statement of main results}
\hk Let $\k:\xR^d_1\to \xR^d_2$ be a diffeomorphism. We equip on $\xR_2^d$ and $\xR_1^d$ two dyadic partitions as in \eqref{partition} with $n=0$ and $n=n_0$, respectively, where $n_0$ is given in Proposition \ref{bigC}.
\begin{nota} 
1.  For a fixed integer $\widetilde N$ sufficiently large  (larger than $N$ given in \eqref{choose:chi} and $N_0=2(n_0+1)$ and to be chose appropriately in the proof of Theorem \ref{theo:conjugation}), we set for any $v\in \mathcal{S}'(\xR^d_1)$
\bq\label{def:recoup} the piece
[v]_p=\sum_{|j-p|\le \widetilde N}\Delta_j v.
\eq
2. For any positive real numer $\mu$ we set  $\mu_{-}=\mu$ if $\mu\notin \xN$ and $\mu_{-}=\mu-\eps$ if $s\in \xN$ with $\eps>0$ arbitrarily small so that $\mu-\eps\notin \xN$.
\end{nota}
Henceforth, we always assume the following assumptions on $\k$:\\
{\bf Assumption I} 
\bq\label{assumption1}\exists \rho>0, ~\partial_x\k\in C_*^\rho(\xR^d_1), \quad \exists \alpha\in \xN^d,~ r>-1,~\partial_x^{\alpha_0}\k\in  H^{r+1-|\alpha_0|}(\xR^d_1).
\eq
{\bf Assumption II} 
\bq
\exists m_0>0,\forall x\in \xR^d_1,~\la \det\k'(x)\ra \ge m_0.
\eq
\begin{defi}\label{k:global}(Global paracomposition)
 For any $u\in \mathcal{S'}(\xR^d_2)$ we define formally 
\[
\k_g^*u=\sum_{p= 0}^\infty[u_p\circ \k]_p.
\]
\end{defi}
We state now our precise results concerning the paracomposition operator $\k_g^*$.
\begin{theo}\label{theo:operation}(Operation)
For every $s\in \xR$ there exists $\cF$ independent of $\k$ such that 
\begin{align*}
&\forall u\in C^s_*(\xR^d_2),\quad \lA \k^*_gu\rA_{C^s_*}\le  \cF(m_0,\lA \k'\rA_{L^\infty})\lA u\rA_{C^s_*},\\
&\forall u\in H^s(\xR^d_2),\quad\lA \k^*_gu\rA_{H^s}\le \cF(m_0,\lA \k'\rA_{L^\infty})\lA u\rA_{H^s}.
\end{align*}
\end{theo}
\begin{theo}\label{theo:linearization}(Linearization)
Let $s\in \xR$. For all $u\in \mathcal{S}'(\xR^d_2)$ we define
\bq\label{def:Rline}
\R_{line}u=u\circ \k-\left(\k^*_gu+\dot T_{u'\circ \k}\k\right).
\eq
$(i)$ If $0<\sigma< 1$, $\rho+\sigma>1$ and $ r+\sigma>0$ then there exists $\cF$ independent of $\k,~u$  such that 
\[
\lA \R_{line}u\rA_{H^{\tilde s}}\le \cF(m_0,\lA \k'\rA_{C_*^\rho})\lA \partial_x^{\alpha_0}\k\rA_{H^{r+1-|\alpha_0|}}\big(1+\lA u'\rA_{H^{s-1}}+\lA u\rA_{C_*^{\sigma}}\big)
\]
where $\tilde s=\min( s+\rho, r+\sigma)$.\\
$(ii)$ If $\sigma >1$, set $\eps=\min(\sigma-1, \rho+1)_{-}$  then there exists $\cF$ independent of $\k,~u$ such that 
\[
\lA \R_{line}u\rA_{H^{\tilde s}}\le \cF(m_0, \lA \k\rA_{C_*^\rho})\lA \partial_x^{\alpha_0}\k\rA_{H^{r+1-|\alpha_0|}}\big(1+\lA u'\rA_{H^{s-1}}+\lA u\rA_{C_*^\sigma}\big)
\]
where $\tilde s=\min(s+\rho, r+1+\eps)$.
\end{theo}
\begin{theo}\label{theo:conjugation}(Conjugation)
 Let  $m,~s\in \xR$ and $\tau>0$. Set $\eps=\min(\tau, \rho)$. Then for every $h(x, \xi)\in \Gamma^m_{\tau}$, homogeneous in $\xi$ there exist
\begin{itemize}
\item  $h^*\in \Sigma^m_\eps$,
\item $\cF$ nonnegative,  independent of $\k,~h$,
\item $k_0=k_0(d, \tau)\in \xN$ 
\end{itemize}
such that we have for all $u\in H^s(\xR^d_2)$,
\begin{gather}\label{formule:conju}
\k^*_gT_hu=T_{h^*}\k^*_gu+\R_{conj}u,\\
\label{main:remainder}
\lA \R_{conj}u\rA_{H^{s-m+\eps}}\le \cF(m_0,\lA\k'\rA_{C_*^\rho})M^m_\tau(h; k_0)\left(1+\lA \partial^{\alpha_0}\k\rA_{H^{+1-\alpha_0}}\right)\lA u\rA_{H^s}.
\end{gather}
(the semi-norm $M^m_\tau(h; k_0)$ is defined in \eqref{defi:norms}). Moreover,  $h^*$ is computed by the formula
\bq\label{h*}
h^*(x, \xi)=\sum_{j=0}^{[\rho]}h^*_j:=\sum_{j=0}^{[\rho]}\frac{1}{j!}\partial_{\xi}^j D_y^j\Big(h\left(\k(x), \R(x,y)^{-1}\xi\right)\frac{\la \det \partial_y\k(y)\ra}{\la\det \R(x,y)\ra} \Big)\arrowvert_{y=x},
\eq
\[
 \R(x, y)={}^t\int_0^1 \partial_x\k(tx+(1-t)y)\di t.
\]
\begin{rema}
\begin{itemize}
\item The definition \eqref{def:Rline} of $\R_{line}$ involves $\dot T_{u\circ \k'}\k$ which does not require the regularity on the low frequency part of the diffeomorphism $\k$.
\item Part $(i)$ of Theorem \ref{theo:conjugation} gives an estimate for the remainder of the linearization of $\k_g^*u$ where $u$ is allowed to be non $C^1$.
\item In part $(ii)$ of Theorem \ref{theo:conjugation}, the possible loss of arbitrarily small regularity in $\eps=\min(\sigma-1, \rho+1)_{-}$ is imposed to avoid the technical issue in the composition of two functions in Zygmund spaces (see the proof of Lemma \ref{lemm:appr}). On the other hand, there is no loss in part $(i)$ when $\sigma \in (0, 1)$.
\item In the estimate \eqref{main:remainder}, $u$ is assumed to have Sobolev regularity. Therefore, in the conjugation formula \eqref{formule:conju} the paradifferential operators $T_h,~T_{h^*}$ can be replaced by their truncated operators $\dot T_h,~\dot T_{h^*}$, modulo a remainder bounded by the right-hand side of \eqref{main:remainder}.
\end{itemize}
\end{rema}
\end{theo}
\subsection{Proof of the main results}
\begin{nota}
 To simplify notations, we denote throughout this section $C^\mu=C^\mu_*(\xR^d)$.
\end{nota}
\subsubsection{Technical lemmas}
First, for every $u\in \mathcal{S}'(\xR^d_2)$ we define formally
\bq\label{remainder:global}
\R_gu=\k^*_gu-\sum_{p\ge 0}[u_p\circ S_p\k]_p.
\eq
The remainder $\R_g$ is  $\rho$-regularized as to be shown in the following lemma.
\begin{lemm}\label{defialte:R}
For every $\mu\in \xR$  there exists $\mathcal{F}$ independent of $\k$ such that:
 \[
\forall v\in H^\mu(\xR^d_2),~\lA \R_gv\rA_{H^{\mu+\rho}}\le\cF(m_0, \lA \k'\rA_{C^{\rho}}) \lA v'\rA_{H^{\mu-1}}\left(1+\lA \partial^{\alpha_0}\k\rA_{H^{r+1-\alpha_0}}\right).
\]
\end{lemm}
\begin{proof}
By definition, we have
\[
\R_gv=-\sum_{p\ge 0}[v_p\circ S_p\k-v_p\circ \k]_p=-\sum_{p\ge 0}[A_p]_p.
\]
Each term $A_p$ can be written using Taylor's formula:
\[
A_p(x)=\int_0^1v'_p(\k(x)+t(S_p\k(x)-\k(x)))\di t(S_p\k(x)-\k(x)).
\]
\hk 1. Case 1: $p\ge p_0$. Setting  $y(x)=\k(x)+t(S_p\k(x)-\k(x))$, one has as in \eqref{S_pk'} $
\la \det (y')\ra\ge \frac{m_0}{2}$, hence
\bq\label{est:A_p}
\lA \int_0^1v'_p(\k(x)+t(S_p\k(x)-\k(x)))\di t\rA_{L^2}\le \cF(m_0, \lA \k'\rA_{C^{\rho}})\lA v_p'\rA_{L^2}.
\eq
Then by virtue of the estimate \eqref{auxi:22}  we obtain 
\bq\label{def:appr1}
\forall p\ge p_0,\quad \lA A_p\rA_{L^2} \le 2^{-p(\rho+1)}2^{-p(\mu-1)}\cF(m_0, \lA \k'\rA_{C^{\rho}})e_p=2^{-p(\rho+\mu)}\cF(m_0, \lA \k'\rA_{C^{\rho}})e_p
\eq
with $$\sum_{p= p_0}^\infty e_p^2\le \lA v'\rA^2_{H^{\mu-1}}.$$
\hk 2. Case 2: $0\le p<p_0$. We have by the Sobolev embedding $H^{d/2+1}\hookrightarrow L^\infty$
\[
\lA \int_0^1v'_p(\k(x)+t(S_p\k(x)-\k(x)))\di t\rA_{L^\infty}\le \lA v_p'\rA_{L^\infty}\le 2^{p(\frac{d}{2}-s+2)}\lA v'\rA_{H^{s-1}}.
\]
Applying Lemma \ref{auxi:1} we may estimate with $\sum_{p\ge 0}f_p^2\le \cF(m_0, \lA \k'\rA_{C^{\rho}})\lA \partial^{\alpha_0}\k\rA_{H^{r+1-|\alpha_0|}}^2$
\begin{align*}
\lA \k-S_p\k\rA_{L^2}&\le  \sum_{j=p+1}^\infty \lA \Delta_j \k\rA_{L^2}\le  \sum_{j=p+1}^\infty 2^{-j|\alpha_0|}\lA \Delta_j \partial^{\alpha_0}\k\rA_{L^2}\\
& \le  \sum_{j=p+1}^\infty 2^{-j|\alpha_0|}2^{-j(r+1-|\alpha_0|)}f_p\le \cF(m_0, \lA \k'\rA_{C^{\rho}})\lA \partial^{\alpha_0}\k\rA_{H^{r+1-\alpha_0}},
 \end{align*}
where we have used the assumption that $r+1>0$. Therefore,
\bq\label{def:appr2}
\forall p<p_0,\quad \lA A_p\rA_{L^2} \le \cF(m_0, \lA \k'\rA_{C^{\rho}})\lA v'\rA_{H^{\mu-1}}\lA \partial^{\alpha_0}\k\rA_{H^{r+1-\alpha_0}}.
\eq
\hk 3. Finally, noticing that the spectrum of $[A_p]_p$ is contained in an annulus $\{M^{-1}2^p \le |\xi|\le 2^p M\}$ with $M$ depending on $n_0$, the lemma then follows from \eqref{def:appr1}, \eqref{def:appr2}.
\end{proof}
\begin{lemm}\label{lemm:appr}
Let  $\mu>0$ and $\eps=\min(\mu, \rho+1)_{-}$. For every $v\in C^{\mu}(\xR^d_2)$, set
\[
r_p\defn S_p(v\circ \k)-(S_pv)\circ (S_p\k).
\]
Then for every $\alpha\in \xN$ there exists a non-decreasing function $\cF_\alpha$ independent of $\k$ and $v$ such that
\[ \lA  \partial_x^\alpha r_p\rA_{L^\infty}\le 2^{p(|\alpha|-\eps)}\cF_\alpha(\lA \k'\rA_{C^{\rho}})(1+\lA v\rA_{C^\mu}).
\]
\end{lemm}
\begin{proof} We first remark that by interpolation, it suffices to prove the estimate for $\alpha=0$ and all $|\alpha|$ large enough. By definition of $\eps$ we have $v\circ \k\in C^\eps$ with norm bounded by $\cF(\lA \k'\rA_{C^{\rho}})(1+\lA v\rA_{C^\mu})$.\\ 
\hk 1. $\alpha=0$. One writes 
\[
r_p=(S_p(v\circ \k)-v\circ \k) +(v-S_pv)\circ\k+(S_pv\circ \k-S_p v\circ S_p\k)
\]
and use \eqref{auxi:22} to estimate the first two terms. For  the last term, by Taylor's formula and  \eqref{auxi:22} (consider $\mu>1, =1$ or $<1$) we have
\[
\lA S_pv\circ \k-S_p v\circ S_p\k\rA_{L^\infty}\le \lA S_pv'\rA_{L^\infty}\lA \k-S_p\k\rA_{L^\infty}\le C2^{-p\eps}\lA v\rA_{C^\mu}\lA \k'\rA_{C^\rho}.
\]
Therefore,\[
\lA r_p\rA_{L^\infty}\le C2^{-p\eps}\left(\lA v\circ\k\rA_{C^\eps}+\lA v\rA_{C^\eps}+\lA v\rA_{C^\mu}\lA \k'\rA_{C^\rho}\right).
\]
\hk 2. $|\alpha|>\rho+1$. The estimate  \eqref{auxi:21} implies
\[
\lA S_p(v\circ \k)^{(\alpha)}\rA_{L^\infty}\le C_\alpha 2^{p(|\alpha|-\eps)}\lA v\circ \k\rA_{C^\eps}.
\]
On the other hand, part 2. of the proof of Lemma $2.1.1$, \cite{Alipara} gives
\[
\lA (S_pv\circ S_p\k)^{(\alpha)}\rA_{L^\infty}\le C_\alpha 2^{p(|\alpha|-\eps)}(1+\lA \k'\rA_{C^\rho})^{|\alpha|}\lA v\rA_{C^\mu}.
\]
Consequently, we get the desired estimate for all  $|\alpha|>1+\rho$, which completes the proof.
\end{proof}
\begin{lemm}\label{lem:main}
Let $v\in C^\infty(\xR_2^d), ~\supp \hat v\in \mathcal{C}_q(0)$. Recall that $N_0=2(n_0+1)$ with $n_0$ given by Proposition \ref{bigC}.\\
$(i)$ For $p\ge 0$, $j\ge q+N_0+1$ and $k\in \xN$ there exists $\cF_k$ independent of $\k,~v$ such that
\[
\lA \left(v\circ S_p\k\right)_j \rA_{L^2(\xR^d)}  \le 2^{-jk}2^{p(k-\rho)_+}\lA v\rA_{L^2}\cF_k(m_0, \lA \k'\rA_{C^\rho}).
\]
$(ii)$ For $p\ge p_0$, $0\le \ell\le \ell'\le q-N_0-1$ and $k\in \xN$ there exists $\cF_k$  independent of $\k,~v$ such that
\[
\lA \sum_{j=\ell}^{\ell'}\left(v\circ S_p\k\right)_j \rA_{L^2(\xR^d)} \le 2^{-qk}2^{p(k-\rho)_+}\lA v\rA_{L^2}\cF_k(m_0,\lA \k'\rA_{C^\rho}).
\]
$(iii)$ Set $R_pu=[u_p\circ S_p\k]_p-(u_p\circ S_p\k)$. For any $p\ge p_0$, there exists $\cF_k$  independent of $\k,~u$ such that
\[
\lA R_p u\rA_{L^2}\le 2^{-p\rho}\lA u_p\rA_{L^2}\cF_k(m_0, \lA \k'\rA_{C^\rho}).
\]
\end{lemm}
\begin{proof}
First, it is clear that $(iii)$ is a  consequence of $(i)$ and $(ii)$ both applied with $k>\rho$. The proof of $(i)$ and $(ii)$ follows {\it mutadis mutandis} that of Lemma $2.1.2$, \cite{Alipara}, using the technique of integration by parts with non-stationary phase. We only explain how to obtain the non-stationariness here. Let $\widetilde \varphi=1$ on $\mathcal{C}(0)$ and $\supp\widetilde\varphi\subset \mathcal{C}(1)$. The phase of the integral (with respect to $y$) appearing in the expression of $(v\circ S_p\k)_j $ and $\sum_{j=\ell}^{\ell'}(v\circ S_p\k)_j$ is 
\[
S_p\k(y)\eta-y\xi
\]
where,
\begin{itemize}
\item in case $(i)$, $(\eta, \xi)\in \supp (\widetilde\varphi(2^{-q}\cdot)\times \supp\varphi(2^{-j}\cdot)$,\\
 \item in case $(ii)$, $(\eta, \xi)\in \supp (\widetilde\varphi(2^{-q}\cdot)\times \supp \phi(2^{-\iota}\cdot)$ with $\iota=\ell$ or $\ell'+1$, which comes from the fact that 
\[
\sum_{j=\ell'}^\ell \varphi(2^{-j}\xi)=\phi(2^{-\ell}\xi)-\phi(2^{-\ell'-1}).
\]
\end{itemize}  
In both cases, 
\[
\la \partial_y(S_p\k(y)\eta-y\xi)\ra=\la S_p\k'(y)\eta-\xi\ra \ge 1
\]
by virtue of Proposition \ref{bigC}.
\end{proof}
\subsubsection{Proof of Theorem \ref{theo:operation}}
By definition \ref{k:global} of the global paracomposition $
\k_g^*u=\sum_{p= 0}^\infty[u_p\circ \k]_p.$ Since each $[u_p\circ \k]_p$ is spectrally localized in a dyadic cell depending on $n_0=\cF(m, \|\k'\|_{L^\infty})$, the theorem follows from Lemma \ref{boundSj} after making the change of variables $y=\k (x)$.
\subsubsection{Proof of Theorem \ref{theo:linearization}}
Using the dyadic partition $1=\sum_{p\ge 0}u_p$  and the fact that $S_p\to Id$ in $\mathcal{S}'$ we have in $\mathcal{D}'(\xR_1^d)$
\[
u\circ \k=\sum_{p\ge 0}u_p\circ\k=\sum_{p\ge 0}\sum_{q\ge 0}\left(u_p\circ S_{q+1}\k-u_p\circ S_q\k\right)+\sum_{p\ge 0}u_p\circ S_0\k.
\]
Denoting by $S$ the first right-hand side term, one has by Fubini,
\[
S=\sum_{q\ge 0}\sum_{0\le p\le q}(u_p\circ S_{q+1}\chi-u_p\circ S_q\k)+\sum_{q\ge 0}\sum_{p\ge q+1}(u_p\circ S_{q+1}\k-u_p\circ S_q\k)=:(I)+(II).
\]
For $(I)$ we take the sum in $p$ first and notice that $S_0=\Delta_0$ to get
\[
(I)=\sum_{q\ge 0}(S_qu\circ S_{q+1}\k-S_qu\circ S_q\k).
\]
For $(II)$ we write
\[
(II)=\sum_{p\ge 1}\sum_{0\le q\le p-1}(u_p\circ S_{q+1}\k-u_p\circ S_q\k)=\sum_{p\ge 1}(u_p\circ S_p\k-u_p\circ S_0\k).
\]
Summing up, we derive
\bq\label{ucircchi}
u\circ \k=\sum_{p\ge 0}u_p\circ S_p\k+\sum_{q\ge 0}(S_qu\circ S_{q+1}\k-S_qu\circ S_q\k)=:A+B.
\eq
Thanks to lemma \ref{defialte:R}, there hold
\bq\label{writeA}
A=\sum_{p\ge 0}u_p\circ S_p\k=\k^*_gu+\R_gu,\quad\text{with}
\eq
\bq\label{linear:R1}
\lA \R_gu\rA_{H^{s+\rho}}\le \cF(m_0, \lA \k'\rA_{C^{\rho}})\lA u'\rA_{H^{s-1}}\left(1+\lA \partial^{\alpha_0}\k\rA_{H^{r+1-\alpha_0}}\right).
\eq
On the other hand, $B=\sum_{q\ge 0}B_q$ with 
\[
B_q\defn S_{q}u\circ S_{q+1}\k-S_{q}u\circ S_q\k=r_{q+1}\k_{q+1}+S_{q-N+1}(u'\circ \k)\k_{q+1}
\]
where 
\[
r_{q+1}=\int_0^1(S_{q}u')(tS_{q+1}\k+(1-t)S_q\k)\di t-S_{q-N+1}(u'\circ \k).
\]
By definition of truncated paradifferential operators
\bq\label{Tdot:1}
\sum_{q\ge 0}S_{q-N+1}(u'\circ \k)\k_{q+1}=\sum_{p\ge 1}S_{p-N}(u'\circ \k)\k_{p}=\dot T_{u'\circ \k}\k.
\eq
Thus, it remains to estimate
\[
\sum_{q\ge 0}r_{q+1}\k_{q+1}=\sum_{q\ge 1}r_q\k_q.
\]
\hk $(i)$ {\it Case 1}: $0< \sigma <1,~\sigma+\rho>1$\\
In this case, we see that $u\circ \k\in C^{\sigma}$, hence $(u\circ \k)'\in C^{\sigma-1}$ with norm bounded by $\cF(m_0, \lA \k'\rA_{C^\rho})\lA u\rA_{C^\sigma}$. Then, using \eqref{tame:H<0} with $\alpha=1-\sigma,~\beta=\rho_{-}$ yields
\[
\lA u'\circ \k\rA_{C^{\sigma-1}}=\lA (\k')^{-1}(u\circ \k)'\rA_{C^{\sigma-1}}\le  \cF(m_0, \lA \k'\rA_{C^\rho})\lA (\k')^{-1}\rA_{C^{\rho_{-}}}\lA (u\circ \k)'\rA_{C^{\sigma-1}}.
\]
By writing $(\k')^{-1}=\frac{1}{\det (\k')}\text{adj}(\k')$ we get easily that 
$
 \lA (\k')^{-1}\rA_{C^{\rho_{-}}}\le \cF(m_0, \lA \k'\rA_{C^\rho})
$
and hence 
\[
\lA u'\circ \k\rA_{C^{\sigma-1}}\le \cF(m_0, \lA \k'\rA_{C^\rho})\lA u\rA_{C^\sigma}.\]
Now,  we claim that
\bq\label{notC1}
\forall q\ge 1,~ \forall \alpha\in \xN^d,~ \lA \partial_x^\alpha r_p\rA_{L^\infty}\le 2^{q(|\alpha|+1-\sigma)}\cF_\alpha(m_0, \lA \k'\rA_{C^\rho})\lA u\rA_{C^\sigma}.
\eq
 Since $\sigma-1<0$ it follows from \eqref{auxi:21'} that
\[
\lA \partial_x^\alpha S_{q-N}(u'\circ \k)\rA_{L^\infty}\le C_\alpha 2^{q(|\alpha|+1-\sigma)}\lA u'\circ \k\rA_{C^{\sigma-1}}\le 2^{q(|\alpha|+1-\sigma)}\cF_\alpha(m_0, \lA \k'\rA_{C^\rho})\lA u\rA_{C^\sigma}.
\]
Thus, to obtain \eqref{notC1} it remains to prove 
\bq\label{notC1:1}
\forall q\ge 1,~ \forall \alpha\in \xN^d,~ \lA \partial_x^\alpha (S_qu'(S_q\k))\rA_{L^\infty}\le 2^{q(|\alpha|+1-\sigma)}\cF_\alpha(m_0, \lA \k'\rA_{C^\rho})\lA u\rA_{C^\sigma}.
\eq
By interpolation, this will follow from the corresponding estimates for $\alpha=0$ and $|\alpha|>1+\rho$. Again, since $\sigma-1<0$ we have \eqref{notC1:1} for $\alpha=0$.\\
 Now, consider $|\alpha|>1+\rho$. By the Faa-di-Bruno formula $((S_qu')\circ (S_q\k))^{(\alpha)}$ is a finite sum of  terms of the following form 
\[
A=(S_qu')^{(m)}\prod_{j=1}^t [(S_q\k)^{(\gamma_j)}]^{s_j},
\]
where $1\le |m|\le |\alpha|,~ |\gamma_j|\ge 1,~|s_j|\ge1,~\sum_{j=1}^t|s_j|\gamma_j=\alpha,~\sum_{j=1}^ts_j=m$.\\
By virtue of \eqref{auxi:21},  one gets
\begin{align*}
\lA (S_q\k)^{(\gamma_j)}\rA_{L^\infty}=\lA (S_q\k')^{(\gamma_j-1)}\rA_{L^\infty}&\le
\begin{cases}
C2^{q(|\gamma_j|-1-\rho)}\lA  \k'\rA_{C^{\rho}},\quad &\text{if}~ |\gamma_j|-1>\rho\\
C\lA \k'\rA_{C^\rho}, \quad &\text{if}~|\gamma_j|-1<\rho\\
C_\eta2^{q\eta}\lA \k'\rA_{C^{\rho}},\forall \eta>0 &\text{if}~|\gamma_j|-1=\rho.
\end{cases}
\\
& \le C_\alpha 2^{q(|\gamma_j|-1)(1-\frac{\rho}{|\alpha|-1})}\lA \k'\rA_{C^\rho}.
\end{align*}
Consequently,
\bq\label{est:Pik}
\lA \prod_{j=1}^t [(S_q\k)^{(\gamma_j)}]^{s_j}\rA_{L^\infty}\le  C_\alpha 2^{q(|\alpha|-|m|)(1-\frac{\rho}{|\alpha|-1})}\lA \k'\rA^{|m|}_{C^\rho}.
\eq
Combining \eqref{est:Pik} with the estimate (applying \eqref{auxi:21} since $m+1>\sigma$)
\[
\lA (S_qu')^{(m)}\rA_{L^\infty}\le C_m 2^{q(m+1-\sigma)}\lA u\rA_{C^{\sigma}}
\]
yields
\[
\lA \partial_x^\alpha (S_qu'(S_q\k))\rA_{L^\infty}\le 2^{qM}\cF_\alpha(\lA \k'\rA_{C^\rho})\lA u\rA_{C^{\sigma}}
\]
with 
\[
M=(m+1-\sigma)+(|\alpha|-m)(1-\frac{\rho}{|\alpha|-1}) \le |\alpha|+1-\sigma,
\]
which concludes the proof the claim  \eqref{notC1:1}.\\
On the other hand, according to Lemma \ref{auxi:1} $(ii)$ for any $q\ge 1, \alpha\in \xN$ there holds
\bq
\begin{aligned}\label{noteC1:2}
\lA \partial_x^\alpha \k_q\rA_{L^2}&\le C_\alpha 2^{q|\alpha|}\lA  \k_q\rA_{L^2}\le  C_\alpha 2^{q(|\alpha|-|\alpha_0|)}\lA  \partial^{\alpha_0}\k_q\rA_{L^2}\\
&\le C_\alpha  2^{q(|\alpha|-|\alpha_0|)}2^{-q(r+1-|\alpha_0|)}a_p=C_\alpha  2^{q(|\alpha|-r-1)}a_p,
\end{aligned}
\eq
with 
\[
\sum_{q\ge 1} a_q^2\le \cF(m_0, \lA\k'\rA_{C^\rho})\lA \partial_x^{\alpha_0}\k \rA_{H^{r+1-\alpha_0}}^2.
\]
We deduce from \eqref{notC1:1} and \eqref{noteC1:2} that
\[
\forall \alpha\in \xN^d,~\forall q\ge 1,~ 
\lA \partial_x^\alpha (r_q\k_q)\rA_{L^2}\le 2^{q(|\alpha|-r-\sigma)}\cF_\alpha(\lA \k\rA_{C^\rho})\lA u\rA_{C^{\sigma}}a_q.
\]
By the assumption $ r+\sigma>0$  we conclude 
\bq\label{linear:R2}
\|\sum_{q\ge 1}r_q\k_q\|_{H^{r+\sigma}}\le \lA u\rA_{C^{\sigma_{-}}}\cF(\lA \k\rA_{C^\rho})\lA \partial_x^{\alpha_0}\k \rA_{H^{r+1-\alpha_0}}.
\eq
Combining \eqref{writeA}, \eqref{linear:R1}, \eqref{Tdot:1}, \eqref{linear:R2} we obtain the assertion $(i)$ of Theorem \ref{theo:linearization}.\\
$(ii)$ {\it Case 2}: $\sigma>1$. This case was studied in \cite{Alipara}. One writes
\[
S_{q-N}(u'\circ \k)=t_q+S_{q-1}u'\circ S_{q-1}\k+s_q.
\]
with
\[
t_q=S_{q-N}(u'\circ \k)-S_{q-1}(u'\circ \k),~s_q=S_{q-1}(u'\circ \k)-S_{q-1}u'\circ S_{q-1}\k.
\]
Plugging this into $r_q$ gives 
\[
r_q=z_q-t_q-s_q
\]
with
\begin{align*}
z_q&=\int_0^1(S_{q-1}u')(tS_{q}\k+(1-t)S_{q-1}\k)\di t-S_{q-1}u'\circ S_{q-1}\k\\
&=\k_q\int_0^1t\int_0^1((S_{q-1}u'))'(S_{q-1}\k+st\k_q)\di s\di t.
\end{align*}
Now we estimate the $L^\infty$-norm of derivatives of $r_q$. Since
\[t_q=-\sum_{j=q-N+1}^{q-1}(u'\circ \k)_j
\]
we get with $\eps=\min(\sigma-1, \rho+1)_{-}$ 
\bq\label{stz}
\forall q\ge 1, \forall \alpha\in \xN,~ \lA  \partial_x^\alpha t_q\rA_{L^\infty}\le 2^{p(|\alpha|-\eps)}\cF_\alpha(\lA \k'\rA_{C^{\rho}})(1+\lA u\rA_{C^\sigma}).
\eq
Applying Lemma \ref{lemm:appr} we have the same estimate as \eqref{stz} for $s_q$. Finally, following exactly part d) of the proof of Lemma $3.1$, \cite{Alipara} with the use of Lemma \ref{auxi:2} one obtains the same bounds for $z_q$.\\
We conclude by using \eqref{noteC1:2}  that
\[
\| \sum_{q\ge 1}r_q\k_q\|_{H^{r+1+\eps}}\le (1+\lA u\rA_{C^\sigma})\lA \partial_x^{\alpha_0}\k \rA_{H^{r+1-\alpha_0}}\cF(\lA \k'\rA_{C^{\rho}}),
\]
which combines with \eqref{linear:R1} gives the  assertion $(ii)$ of Theorem \ref{theo:linearization}.
\subsubsection{Proof of Theorem \ref{theo:conjugation}}
We recall first the following lemma in \cite{Alipara}.
\begin{lemm}[\protect{\cite[Lemme d), page 111]{Alipara}}]\label{lem:pseudo}
Let $K\subset \xR^d$ be a compact set. Let $a(x, y, \eta)$ be a bounded function;  $C^\infty$ in $\eta$ and its support w.r.t $\eta$ is contained in $K$; its derivatives w.r.t $\eta$ are bounded. For every $p\in \xN$, define the associated pseudo-differential operator
\[
A_pv(x)=\int e^{i(x-y)\xi}a(x, y, 2^{-p\xi})v(y)\di y\di\xi.
\]
 Then, there exist  a constant $C>0$ independent of $a, p$  and an integer $k_1=k_1(d)$ such that with 
\[
M=\sup_{|\alpha|\le k_1}\lA \partial_\eta^\alpha a\rA_{L^\infty(\xR^d\times\xR^d\times \xR^d)}
\]
we have
\[
\forall v\in L^2(\xR^d),~\lA A_pv\rA_{L^2}\le CM\lA v\rA_{L^2}.
\]
\end{lemm}
Now we quantify the proof of Lemma $3.3$  in \cite{Alipara}.  Let  $m,~s\in \xR$,  $\tau>0$, $\eps=\min(\tau, \rho)$ and $h(x, \xi)\in \Gamma^m_{\tau}$, homogeneous in $\xi$. We say that a quantity $Q$ is controllable if  $\lA Q\rA_{H^{s-m+\eps}}$ is bounded by the right-hand side of  \eqref{main:remainder} and therefore can be neglected. Also, by $A\sim B$ we mean that $A-B$ is controllable.\\
{\bf Step 1.} First, by Lemma \ref{defialte:R} we have
\bq\label{pre}
\k^*_g(T_hu)\sim \sum_{p\ge 0}\left [\Delta_pT_hu\circ S_p\k\right]_p.
\eq
Then with $v^q= (S_{q-N}h)(x, D)u_q$, it holds that
\[
\k^*_g(T_hu)\sim\sum_{p\ge 0}\sum_{q\ge 1}\left [\Delta_p v^q\circ S_p\k\right]_p.
\]
One can see easily that if $N$ is chosen larger than $n$ enough then the spectrum of $v^q$ is contained in the annulus 
\[
\left\{\xi\in \xR^d: 2^{p-M_1}\le |\xi|\le 2^{p+M_1}\right\}
\]
with $M_1=M_1(N, n_0)$. This implies 
\[
\Delta_pv^q=0~\text{if}~|p-q|>M:=M_1+n_0+1=M(N, n_0)
\]
and thus, 
\[
\k^*_g(T_hu)=\sum_{|p-q|\le M}\left [\Delta_p v^q\circ S_p\k\right]_p.
\]
Set
\begin{align*}
S_1&=\sum_{|p-q|\le M}\Big( \left [\Delta_p v^q\circ S_p\k\right]_p-\left [\Delta_p v^q\circ S_p\k\right]_q\Big), \\
S_2&=\sum_{|p-q|\le M}\Big( \left [\Delta_p v^q\circ S_p\k\right]_q-\left [\Delta_p v^q\circ \k\right]_q\Big).
\end{align*}
We shall prove that $S_1,~S_2$ are controllable so that 
\bq\label{def:kT}
\k^*_g(T_hu)\sim\sum_{p, q\ge 0}\left [\Delta_p v^q\circ \k\right]_q=\sum_{q\ge 0}\left [ v^q\circ \k\right]_q=\sum_{p\ge 0} \left [(S_{p-N}h)\Delta_p u\circ \k\right]_p.
\eq
The estimate for $S_2$ is proved along the same lines as in the proof of Lemma \ref{defialte:R}. We now consider $S_1$. If we choose $\widetilde N\gg M+N_0$ in the definition of $[\cdot]_p$ then 
\[S_1=\sum_{\substack{p, q\\|p-q|\le M}}\sum_{\substack{j\\N_0<|j-p|\le \widetilde N}}\Delta_j(\Delta_pv^q\circ S_p\k)-\sum_{\substack{p, q\\|p-q|\le M}}\sum_{\substack{j\\N_0<|j-q|\le \widetilde N}}\Delta_j(\Delta_pv^q\circ S_p\k)=S_{1,1}-S_{1,2}.
\]
Each of $S_{1,1}$ and $S_{1,2}$ is treated in the same way. Let us consider $S_{1,1}=\sum_ja_j^1+\sum_ja_j^2,$
\[
a_j^1=\sum_{\substack{p\\ p<j-N_0\\|p-j|\le \widetilde N}}\sum_{\substack{q\\ |q-p|\le M}}\Delta_j(\Delta_pv^q\circ S_p\k),\quad a_j^2=\sum_{\substack{p\\ p>j+N_0\\|p-j|\le \widetilde N}}\sum_{\substack{q\\ |q-p|\le M}}\Delta_j(\Delta_pv^q\circ S_p\k).
\]
Since for all $q$ 
\[
\lA v^q\rA_{L^2}\le M^m_0(h)\lA \Delta_qu\rA_{H^m},
\]
by  virtue of Lemma \ref{lem:main} $(i)$ (applied with $k>\rho$) one has
\begin{align*}
\lA a_j^1\rA_{L^2}&\le \sum_{p< j-N_0, |p-j|\le \widetilde N, |q-p|\le M}C_k2^{-jk}2^{p(k-\rho)}M^m_0(h)\lA \Delta_qu\rA_{H^m}\cF_k(m_0, \lA \k'\rA_{C^\rho})\\
&\le \sum_{|q-j|\le M+\widetilde N}C_k2^{-j\rho+mq}M^m_0(h)\lA \Delta_qu\rA_{L^2}\cF_k(m_0, \lA \k'\rA_{C^\rho})\\
&\le C_k2^{-j(\rho-m+s)}M^m_0(h)\cF_k(m_0, \lA \k'\rA_{C^\rho})\sum_{|q-j|\le M+\widetilde N}b_q.
\end{align*}
with $\lA b\rA_{\ell^2}\le C\lA u\rA_{H^s}$.
Then, thanks to the spectral localization of $a_j^1$ we conclude that
\[
\| \sum_j a_j^1\|_{H^{s-m+\rho}}\le M^m_0(h)\lA u\rA_{H^s}\cF_k(m_0, \lA \k'\rA_{C^\rho}).
\]
For the second sum $\sum a_j^2$ we apply Lemma \ref{lem:main} $(ii)$.\\
{\bf Step 2.} Recall from \eqref{def:kT} that 
\[
\k_g^* T_hu=\sum_{p\ge 0}[A_p]_p,\quad A_p=\big( (S_{p-N}h)(x, D)u_p\big)\circ \k.
\]
One writes
\[
A_p(y)=\int e^{i(\k(y)-y')\xi}\widetilde\varphi(2^{-p}\xi)(S_{p-N}h)(\k(y), \xi)u_p(y')\di y'\di\xi 
\]
where $\widetilde \varphi$ is a cutt-off function analogous to $\varphi$ and equal to $1$ on the support of $\varphi$.\\
In the expression of $A_p$ we make two changes of variables
\[
y'=\k(z),\quad \xi={}^t\R^{-1}\eta, \quad 
\R=\R(y, z):=\int_0^1\k'(ty+(1-t)z)\di t
\]
to derive
\[
A_p(y)=\int e^{i(y-z)\eta}\widetilde\varphi(2^{-p}. {}^t\R^{-1})(S_{p-N}h)(\k(y), {}^t\R^{-1}\eta)u_p(\k(z))\frac{\la \k'(z) \ra}{\la \det \R\ra}\di z\di\eta.
\]
The rest of the proof follows the same method as in the classical case classical pseudodifferential calculus except that we shall regularize first the symbol $a_p(y, z, \eta)$ of $A_p$: set
\[
b_p(y, z, \eta)=\widetilde\varphi(2^{-p}. {}^t\R_p^{-1})(S_{p-N}h)(S_p\k(y), {}^t\R_p^{-1}\eta)\frac{\la S_p\k'(z) \ra}{\la \det \R_p\ra}
\]
with 
\[
\R_p=\R_p(y, z)=\int_0^1S_p\k'(ty+(1-t)z)\di t.
\]
Thanks to the homogeneity of $S_{p-N}h$ we write $a_p(y, z, \eta)=2^{pm}\widetilde a_p(y, z, 2^{-p}\eta)$ and similarly for $\widetilde b_p$. Then due to the presence of the cut-off function $\widetilde \varphi$ one can prove without any difficulty that
\[
\forall k\in \xN,~\sup_{|\alpha|\le k}\la \partial_\eta^\alpha(\widetilde a_p-\widetilde b_p)(y, z, \eta)\ra\le C_k2^{-p\rho}\cF_k(m_0, \lA \k'\rA_{C^\rho})M^m_0(h; k+1).
\]
Therefore, in view of Lemma \ref{lem:pseudo} we see that in $\k^*_gT_hu$ the replacement of $a_p$ by $b_p$ gives rise to a controllable remainder.\\
{\bf Step 3.} 
 Next, we expand $b_p(y, z, \eta)$ by Taylor's formula w.r.t $z$ up to order $\ell=[\rho]$, at $z=y$ to have
\[
b_p(y, z, \eta)=b^0_p(y, \eta)+b^1_p(y,\eta)(z-y)+...+b^{\ell}_p(y, \eta)(z-y)^{\ell}+r_p^{\ell+1}(y, z, \eta)(z-y)^{\ell+1}
\]
where $b^j$ is the $j^{th}$-derivative of $b_p$ with respect to $z$, taken at $z=y$ and 
\[
r_p^{\ell+1}(y, z, \eta)=C\int_0^1b_p^{\ell+1}(y, y+t(z-y), \eta)\di t(z-y)^{\ell+1}.
\]
In the pseudo-differential operator $\R_p^{\ell+1}$ with symbol $r_p^{\ell+1}$ we integrate by parts w.r.t $\eta$ $\ell+1$ times to obtain a sum of symbols of the form
$2^{p(m-\ell-1)}\widetilde r_p(y ,z, 2^{-p}\eta)$,
\[
\widetilde r_p(y ,z, \eta)=C\int_0^1\partial_z^\alpha\partial_\eta^\beta\widetilde b_p(y, y+t(z-y), \eta)\di t,~|\alpha|=|\beta|=\ell+1.
\]
For $|\alpha|=\ell+1,~|\gamma|=\ell+1+k,~k\in \xN$ it holds
\[
\la \partial_z^\alpha\partial_\eta^\gamma\widetilde b_p(y, z, \eta)\ra\le C_k2^{p(\ell+1-\rho)}M^m_0(h, [\rho]+1+k)\cF_k(m_0, \lA \k'\rA_{C^\rho}).
\]
Lemma \ref{lem:pseudo} then gives for some $k_1=k_1(d)\in \xN$,
\[
\lA \R_p^{\ell+1}\rA_{L^2}\le 2^{p(\ell+1-\rho)}2^{p(m-\ell-1)}\lA u_p\rA_{L^2}M^m_0(h, [\rho]+1+k_1)\cF_k(m_0, \lA \k'\rA_{C^\rho}).
\]
Therefore, the remainder $\sum_p[R_p^{\ell+1}]_p$ is controllable.\\
{\bf Step 3.}  We write
\[
B_p^ju(y)=\int e^{i(y-z)\eta}b_p^j(y, \eta)(z-y)^ju_p(\k(z))\di z\di\eta
\]
and integrate by parts $j$ times w.r.t $\eta$ to get
\[
B_p^ju=\int e^{i(y-z)\eta}c_p^j(y, \eta) u_p(\k(z))\di z\di\eta.
\]
  The key point here is: in the expression above we shall replace $u_p\circ \k$, $p\ge 0$ by its "recoupe" $[u_p\circ \k]_p$ which will enter  $T_{h^*}\k_g^*u$. Therefore, one has to estimate the $L^2$-norm of the difference  
\[
W_p\defn  u_p\circ \k-[u_p\circ \k]_p
\]
as 
\[
\lA W_p\rA_{L^2}\le 2^{-p\rho}\cF(m_0, \lA \k'\rA_{C^\rho})\lA u_p\rA_{L^2}.
\]
For $0\le p<p_0$. We treat separately each term in $W_p$ by making the change of variables $x\mapsto \k(x)$ to have
\[
\lA W_p\rA_{L^2}\le \cF(m_0, \lA \k'\rA_{C^\rho})\lA u_p\rA_{L^2}.
\]
For $p\ge p_0$ we write
\[
W_p=(u_p\circ \k-u_p\circ S_p\k)+(u_p\circ S_p\k-[u_p\circ S_p\k]_p)+([u_p\circ S_p\k]_p-[u_p\circ \k]_p).
\]
 The second term is estimated using directly Lemma \ref{lem:main} $(iii)$. The first and the last term are treated exactly as in the first part (case 1.) of the proof of Lemma \ref{defialte:R} (see \eqref{def:appr1}). \\
Again, by virtue of Lemmma \ref{lem:pseudo} we conclude that: in $\k_gT_hu$ the replacement of $u_p\circ\k$ by $[u_p\circ \k]_p$ is $(\rho+j-m)-$regularized and controllable. \\
{\bf Step 4.} Set 
\[
C_p^ju(y)=\int e^{i(y-z)\eta}c_p^j(y, \eta) [u_p\circ \k]_p(z)\di z\di\eta.
\]
We observe that if the cut-off function $\widetilde\varphi$ is chosen appropriately then all the terms in $c_p^j$ relating to $\partial^\alpha\widetilde\varphi$ is $1$ if $\alpha=0$ and is $0$ if $\alpha\ne 0$, on the spectrum of $[u_p\circ \k]_p$. Therefore, comparing to the classical calculus \eqref{h*} for $S_{p-N}h$ we can prove that 
\[
\sup_{\substack{|\alpha|\le k\\ 0<c_1\le |\eta|\le c_2}}\la\partial^\alpha_\eta\left(c_p^j-S_{p-N}h^*_{j}\right)(y, \eta) \ra\le C_k2^{-p\eps_j}M^m_\tau(h; k+j+1)\cF_k(m_0, \lA \k'\rA_{C^\rho})
\]
with $\eps_j=\min(\tau, \rho-j)$. \\
Then,  Lemma \ref{lem:pseudo} implies that in $\k^*_gT_hu$ our replacement of $C_p^ju$ by 
\[
D_ju(y)=\int e^{i(y-z)\eta}(S_{p-N}h^*_{j})(y, \eta) [u_p\circ \k]_p(z)\di z\di\eta
\]
leaves a controllable remainder of order $m-j-\eps_j\le m-\eps$. \\
{\bf Step 5.} We have proved in step 4 that 
\[
\k^*_gT_hu\sim\sum_{j=0}^{[\rho]}\sum_p\left[D_p^ju \right]_p=\sum_{j=0}^{[\rho]}\sum_p\left[(S_{p-N}h^*_j)(x, D) [u_p\circ \k]_p\right]_p.
\]Now, notice that if in the definition of $\k_g^*T_hu$ in \eqref{pre} we had chosen instead of $[\cdot]_p$ a larger piece $[\cdot]'_p$ corresponding to $\overline N\gg \widetilde N$ (remark that such a replacement is controllable  according to Lemma \ref{defialte:R} and Lemma \ref{lem:main}) we would have obtained
\[
\k^*_gT_hu=\sum_{j=0}^{[\rho]}\sum_p\left[(S_{p-N}h^*_{j})(x, D) [u_p\circ \k]_p\right]'_p=\sum_{j=0}^{[\rho]}\sum_p\sum_{|k-p|\le \overline{N}}\Delta_k\left(S_{p-N}h^*_{j})(x, D) [u_p\circ \k]_p\right).
\]
Remark that the spectrum of $(S_{p-N}h^*_{j})(x, D) [u_p\circ \k]_p$ is contained in the annulus
\[
\left\{\xi\in \xR^d: 2^{p-M}\le |\xi|\le 2^{p+M} \right\}
\]
for some $M=M(\widetilde{N}, N)>0$. Therefore, if we choose $\overline{N}\gg M(\widetilde{N}, N)$ then 
\[
\Delta_k \left(S_{p-N}h^*_{j})(x, D) [u_p\circ \k]_p\right)=0~\text{if}~|k-p|>\overline{N}
\]
and hence
\bq\label{kT:sum}
\k^*_gT_hu=\sum_{j=0}^{[\rho]}\sum_p(S_{p-N}h^*_{j})(x, D) [u_p\circ \k]_p.
\eq
Finally, we write for $0\le j\le [\rho]$ 
\begin{align*}
T_{h^*_{j}}\k^*u&=\sum_p(S_{p-N}h^*_{j})(x, D) \Delta_p\sum_q[u_q\circ \k]_q\\
&=\sum_p(S_{p-N}h^*_{j})(x, D) \Delta_p\sum_{q}\sum_{\substack{k\\|k-q|\le \widetilde N}}\Delta_k(u_q\circ \k)\\
&=\sum_{\substack{p\\k, q:|k-p|\le N_0, |k-q|\le \widetilde N}}(S_{p-N}h^*_{j})(x, D) \Delta_p\Delta_k(u_q\circ \k).
\end{align*}
In the sum above, the replacement of $(S_{p-N}h^*_j)(x, D)$ by $(S_{q-N}h^*_j)(x, D)$ leaves a controllable remainder, so
\begin{align*}
T_{h^*_j}\k^*u
&=\sum_{\substack{k\\|p-k|\le N_0\\|q-k|\le \widetilde N}}(S_{q-N}h^*_j)(x, D) \Delta_p\Delta_k(u_q\circ \k)=\sum_{\substack{k, q\\|q-k|\le \widetilde N}}(S_{q-N}h^*_j)(x, D) \Delta_k(u_q\circ \k)\\
&=\sum_q(S_{q-N}h^*_j)(x, D)[u_q\circ \k]_q.
\end{align*}
Therefore, we conclude in view of \eqref{kT:sum} that
$
\k^*_gT_hu\sim\sum_{j=0}^{[\rho]}T_{h^*_j}\k^*u.
$
\section{The semi-classical Strichartz estimate}\label{section:semi}
\subsection{Para-change of variable}
First of all, let us recall the symmetrization of \eqref{ww} to a paradifferential equation proved in \cite{NgPo1} for rough solutions. This symmetrization requires the introduction of the following symbols:
\begin{itemize}
\item $\gamma=\left(1+(\partial_x\eta)^2 \right)^{-\frac{3}{4}}|\xi|^\tdm$,\\
\item $\omega=-\frac{i}{2}\partial_x\partial_\xi\gamma$,\\
\item $q=\left(1+(\partial_x\eta)^2 \right)^{-\frac{1}{2}}$,\\
\item $p=\left(1+(\partial_x\eta)^2 \right)^{-\frac{5}{4}}|\xi|^\mez+p^{(-\mez)} $, where $p^{(-\mez)}=F(\partial_x\eta, \xi)\partial_x^2\eta$, 
$F\in C^\infty(\xR\times \xR\setminus\{0\}; \xC)$ is homogeneous of order $-1/2$ in $\xi$.
\end{itemize}
\begin{theo}[\protect{\cite[Proposition~4.1]{NgPo1}}]\label{symmetrization}
Assume that $(\eta, \psi)$ is a solution to \eqref{ww} and satisfies
\begin{equation}\label{assume:eta, psi}
\left\{
\begin{aligned}
&(\eta, \psi)\in C^0([0, T]; H^{s+\mez}(\xR)\times H^s(\xR))\cap L^4([0, T]; W^{r+\mez,\infty}(\xR)\times W^{r,\infty}(\xR)),\\
&s>r>\tdm+\mez.
\end{aligned}
\right.
 \end{equation}
Define 
\[
U:=\psi-T_B\eta,\quad \Phi=T_p\eta+T_qU,
\]
then $\Phi$ solves the problem
\bq\label{eq:para}
\partial_t\Phi+T_V\partial_x\Phi+iT_{\gamma}\Phi=f
\eq
and  there exists a function $\cF:\xR^+\times \xR^+\to \xR^+$, non-decreasing in each argument, independent of $(\eta, \psi)$ such that for a.e. $t\in [0, T]$,
\begin{equation}\label{est:RHS}
		\lA f(t)\rA_{H^s}\leq\cF\lp\lA\eta(t)\rA_{H^{s+\mez}},\lA\psi(t)\rA_{H^s}\rp\lp1+\lA\eta(t)\rA_{W^{r+\mez,\infty}}+\lA\psi(t)\rA_{W^{r,\infty}}\rp.
	\end{equation}
\end{theo}
We assume throughout this section that $(\eta, \psi)$ is a solution to \eqref{ww} with regularity \eqref{assume:eta, psi}. We shall apply our results on the paracomposition in the preceding section to reduce further equation \eqref{eq:para} by adapting the method in \cite{ABZ2}. Define for every $(t, x)\in [0, T]\times \xR$
\[
\chi(t, x)=\int_0^x\sqrt{1+\lp\partial_y\eta(t, y) \rp^2}\di y
\]
then for each $t\in I:=[0, T]$, the mapping $x\mapsto \chi(t, x)$ is a diffeomorphism from $\xR$ to itself. Introduce then for each $t\in I$ the inverse $\k(t)$ of $\chi(t)$.\\
 Concerning the underlying dyadic partitions, we shall write 
\[
\eta(t),~\psi(t): \xR_2\to \xR_1, \quad \k(t): \xR_1\to \xR_2,
\]
where, $\xR_2$ is equipped with the dyadic partition \eqref{partition} of size $n=0$ and $\xR_1$ is equipped with the one of size $n=n_0$ determined in Proposition \ref{bigC}: $n_0=\mathcal{F}_1(m_0, \lA \k'\rA_{L^\infty})$. Since
\[
\k'(x)=\frac{1}{(\partial_x\chi)\circ \k}=\frac{1}{\sqrt{1+(\partial_x\eta)\circ \k(x))^2}},
\]
we get
\bq\label{est:m0}
m_0:=\left(1+\lA \partial_x\eta \rA_{L^\infty_tL^\infty_x}^2\right)^{-1/2}\le \k'(x)\le 1,\quad\forall x\in \xR.
\eq
Therefore, up to a constant of the form $\cF(\lA  \partial_x\eta \rA_{L^\infty_tL^\infty_x})$  we will not distinguish between $\xR_1$ and $\xR_2$ in the rest of this article.\\
\hk As mentioned in the introduction of our paracomposition results, we shall consider the linearized part of $\k_g^*$ as a new definition for paracomposition. More precisely,
we set
\bq\label{defi:u}
u=\k^*\Phi:=\Phi\circ \k-\dot T_{(\partial_x\Phi)\circ \k}\k,
\eq
where, for any function $g: I\times \xR_2\to \xC$ we have denoted
\[
(g\circ \k)(t, x)=g(t, \k(t, x)),\quad \forall (t, x)\in I\times \xR_1.
\]
Let us first gather various estimates that will be used frequently in the sequel. To be concise, we denote 
\[
\mathcal{N}=\cF(\lA \eta\rA_{L^\infty_tH^{s+\mez}_x}, \lA \psi\rA_{L^\infty_tH^s_x})
\]
where $\cF$ is non-decreasing in each argument, independent of $\eta,~\psi$ and $\cF$ may change from line to line. 
\begin{lemm}\label{est:various}
The following estimates hold
\begin{enumerate}
\item $\lA\Phi\rA_{L^\infty_tH^s_x}\le \cN$,
\item $\lA \Phi(t)\rA_{C^r_{*,x}}\le \cN\left(1+\lA  \eta(t)\rA_{W^{r+\mez,\infty}}+\lA  \psi(t)\rA_{W^{r,\infty}}\right)$,
\item $\lA \partial_x\chi-1\rA_{L^\infty_tH^{s-\mez}_x}\le \cN$,
\item $\lA \partial_t\chi\rA_{L^\infty_{t,x}}\le \cN$,
\item $\lA \partial_t\k\rA_{L^\infty_{t,x}}\le \cN$,
\item $\lA \partial_x\k\rA_{L^\infty_tW^{(s-1)_{-},\infty}_x}\le \cN$,
\item $\lA \partial_x\k-1\rA_{L^\infty_tH^{s-\mez}_x}\le \cN$,
\item  $\lA \partial_t\partial_x\chi(t)\rA_{L^\infty_x}\le \cN\left(1+\lA  \psi(t)\rA_{C^r_*}\right)$
\end{enumerate}
\end{lemm}
\begin{proof}
The estimates 1., 2., 3.  can be deduced  straightforwardly from the definition of $\Phi$ and the regularity of $(\eta, \psi)$ given in \eqref{assume:eta, psi}.\\
\hk 4. By definition of $\chi$, 
\bq\label{for:dtchi}
\partial_t\chi(t, x)=\int_0^x\partial_t\partial_y\eta(t, y)\partial_y\eta(t, y)\left(1+\lp\partial_y\eta(t, y) \rp^2\right)^{-\mez}\di y
\eq
so by H\"older's inequality we get
\[
\la \partial_t\chi(t, x)\ra\le \lA \partial_x\partial_t\eta(t )\rA_{L^2_x}\lA F_0\lp\partial_y\eta(t) \rp\rA_{L^2_x}
\]
where $F_0(z)=\frac{z}{\sqrt{z^2+1}}$.  Using \eqref{F(u):H} and Sobolev's embedding give
$
 \lA F_0\lp\partial_y\eta(t) \rp\rA_{L^2_x}\le \cN.
$
On the other hand, using the first equation in \eqref{ww} and the fact that $s>2$ we get
\[
 \lA \partial_x\partial_t\eta(t )\rA_{L^2_x}\le  \lA G(\eta)\psi(t)\rA_{H^1_x}\le \lA G(\eta)\psi(t)\rA_{H^{s-1}_x}\le \cN.
\]
\hk 5. This follows from 4. by using the formula $\partial_t\k=-\frac{\partial_t\chi}{\partial_x\chi}\circ \k$ and noticing that ${\partial_x\chi}\ge 1$.\\
\hk 6. With $F(z)=\frac{1}{\sqrt{1+z^2}}-1$ and $G:=F\circ (\partial_x\eta)$ we have
\bq\label{recurrence:dk}
\partial_x\k=\frac{1}{(\partial_x\chi)\circ \k}=1+F\circ(\partial_x\eta)\circ \k=1+G\circ \k
\eq
 From 3. and Sobolev's embedding, $\partial_x\eta\in L^\infty_tC^{s-1}_* \subset L^\infty_tW^{(s-1)_{-}
,\infty}_x.$
This together with the fact that $F\in C^\infty_b(\xR)$ implies $G\in L^\infty_tW^{(s-1)_{-},\infty}_x$ and 
\bq\label{G:Holder}
\lA G\rA_{L^\infty_tW^{(s-1)_{-}
,\infty}_x}\le \cN.
\eq
Then, bootstrap the recurrence relation \eqref{recurrence:dk} we deduce that $\partial_x\k\in L^\infty_tW^{[(s-1)_{-}],\infty}_x$ and
\bq\label{dk:Holder0}
\lA \partial_x\k\rA_{L^\infty_tW^{[(s-1)_{-}]
,\infty}_x}\le \cN.
\eq
Now, set $\mu =(s-1)_{-}-[(s-1)_{-}]\in (0, 1)$. Again, by \eqref{recurrence:dk}
\bq\label{dGk}
\partial_x^{[(s-1)_{-}]}(\partial_x\k)=\partial^{[(s-1)_{-}]}_x(G\circ \k)
\eq
is a finite combination of terms of the form 
\bq\label{dGK:A}
A=[(\partial^qG)\circ \k]\prod_{j=1}^m\partial_x^{\gamma_j}\k,\quad 1\le q\le [(s-1)_{-}],~\gamma_j\ge 1,~\sum_{j=1}^m\gamma_j=[(s-1)_{-}].
\eq
Using \eqref{dk:Holder0} and \eqref{G:Holder} it follows easily that A belongs to $W^{\mu, \infty}(\xR^d)$ with norm bounded by $\cN$ and thus 6. is proved.\\
\hk 7. First, the nonlinear estimate \eqref{F(u):H} implies that $G=F\circ \partial_x\eta$ defined in the proof of 6. satisfies
\bq\label{G:Sobolev}
\lA G\rA_{L^\infty_tH^{s-\mez}_x}\le \cN.
\eq
Then changing the variable $x\mapsto \chi(x)$ in \eqref{recurrence:dk}  gives
\[
\lA \partial_x\k-1\rA_{L^\infty_tL^2_x}\le \lA G\rA_{L^\infty_tL^2_x}\lA \chi'\rA^\mez_{L^\infty_{t,x}}\le \cN.
\]
Now using \eqref{recurrence:dk}, \eqref{dk:Holder0} and induction we get
\bq\label{dk:Sobolev0}
\lA \partial_x\k-1\rA_{L^\infty_tH^{[(s-1)_{-}]}_x}\le \cN.
\eq
Next, set $\mu=(s-\mez)-[(s-1)_{-}]\in [\mez, \mez+\eps]$, $\eps$ arbitrarily small (so that $\mu\in [\mez, 1)$). To obtain 7. we are left with the estimate for $\partial_x^{[(s-1)_{-}]}\partial_x\k$ in $H^\mu$-norm. This amounts to estimating
\bq\label{doubleint}
\iint_{\xR^2} \frac{|\partial_x^{[(s-1)_{-}]}(G\circ \k)(x)-\partial_x^{[(s-1)_{-}]}(G\circ \k)(y)|^2}{|x-y|^{1+2\mu}}\di x\di y
\eq
where $\partial_x^{[(s-1)_{-}]}(G\circ \k)$ is a finite linear combination of terms of the form A in \eqref{dGK:A}. Inserting $A$ into  \eqref{doubleint} one estimates successively the difference of each factor in $A$ under the double integral while the others are estimated in $L^\infty$-norm. This is done using \eqref{G:Sobolev}, \eqref{dk:Sobolev0} for Sobolev-norm estimates and \eqref{G:Holder}, \eqref{dk:Holder0} for H\"older-norm estimates. \\
\hk 8. By definition of $\chi$, it holds with $F_0(z)=\frac{z}{\sqrt{1+z^2}}$
\[
\partial_t\partial_x\chi(t,x)=F_0(\partial_x\eta)\partial_x\partial_t\eta=F_0(\partial_x\eta)\partial_xG(\eta)\psi.
\]
Then, applying the Holder estimate for the Dirichlet-Neumann operator in Proposition $2.10$, \cite{NgPo1} we get
\[
\lA \partial_xG(\eta)\psi\rA_{L^\infty}\le \lA \partial_xG(\eta)\psi\rA_{C_*^{r-2}}\le \cN\left(1+\lA  \psi(t)\rA_{C_*^r}\right)
\]
and hence the result.
\end{proof}
The main task here is to apply Theorem \ref{theo:linearization} and Theorem \ref{theo:conjugation} to convert the highest order paradifferential operator $T_\gamma$  to the Fourier multiplier $|D_x|^\tdm$.
\begin{prop}\label{parareduce}
The function $u$ defined by \eqref{defi:u} satisfies the equation
\bq\label{eq:u}
\left(\partial_t+T_W\partial_x+i|D_x|^{\tdm}\right)u=f
\eq
where 
\bq\label{defi:W}
W=(V\circ \kappa)(\partial_x\chi\circ\kappa)+\partial_t\chi\circ \kappa
\eq
and for a.e. $t\in [0, T]$,
\bq\label{est:RHS}
\lA f(t)\rA_{ H^{s-\mez}}\le \cF(\lA \eta\rA_{L^\infty_tH^{s+\mez}_x}, \lA \psi\rA_{L^\infty_tH^s_x})\left(1+\lA  \eta(t)\rA_{W^{r+\mez,\infty}}+\lA  \psi(t)\rA_{W^{r,\infty}}\right).
\eq
\end{prop}
\begin{proof} We proceed in 4 steps. We shall say that $A$ is controllable if for a.e. $t\in [0, T]$, $\|A(t)\|_{H^{s-\mez}}$ is bounded by  the right-hand side of \eqref{est:RHS} denoted by $\RHS$.\\
{\bf Step 1.} Let us first prove that for some controllable remainder $R_1$, 
\bq\label{dtPhi}
\k^*(\partial_t \Phi)=\lp\partial_t+T_{(\partial_t\chi)\circ \k}\partial_x\rp u+R_1.
\eq
By definition of $\k^*$ we have
\[
\k^*(\partial_t \Phi)=\partial_t\Phi\circ \k-\dot T_{(\partial_x\partial_t\Phi)\circ \k}\k=\partial_t(\Phi\circ \k)-(\partial_x\Phi\circ k)\partial_t\k-\dot T_{(\partial_x\partial_t\Phi)\circ \k}\k.
\]
Therefore,
\bq\label{dtPhi1}
\k^*(\partial_t \Phi)=\partial_t (\k^*\Phi)+A_1+A_2
\eq
\[
A_1=\dot T_{(\partial_x^2\Phi\circ\k)\partial_t\k}\k,\quad A_2=\dot T_{(\partial_x\Phi)\circ \k}\partial_t\k-(\partial_x\Phi\circ \k)\partial_t\k.
\]
\hk 1. Since the truncated paradifferential operator $\dot T_{(\partial_x^2\Phi\circ\k)\partial_t\k}\k$  involves only the high frequency part of $\k$ we have
\bq\label{dtPhi2}
\lA  A_1\rA_{H^{s+\mez}_x}\le \mathcal{N}\lA (\partial_x^2\Phi\circ\k)\partial_t\k\rA_{L^\infty_x}\lA \partial_x^2\k\rA_{H^{s-\tdm}}.
\eq
From Lemma \ref{est:various} $2., 5.$ there holds
\[
\lA (\partial_x^2\Phi\circ\k)\partial_t\k\rA_{L^\infty_x}\le \cN\left(1+\lA  \eta(t)\rA_{W^{r+\mez,\infty}}+\lA  \psi(t)\rA_{W^{r,\infty}}\right).
\]
On the other hand, Lemma \ref{est:various} $7.$ gives $\lA\partial_x^2\k\rA_{H^{s-\tdm}}\le \cN$, hence $A_1$ is controllable.\\
\hk 2. To study $A_2$, one uses $\partial_t\k=-ab$ with $a=(\partial_t\chi)\circ \k,~b=\partial_x\k$. Set $c=(\partial_x\Phi)\circ \k$ then 
\[
\partial_x(\k^*\Phi)=bc-\dot T_cb-\dot T_{\partial_xc}\k,
\]
hence
\begin{align*}
A_2&=-\dot T_c(ab)+abc=\dot T_{ab}c+\dot R(c, ab)=\dot T_a\dot T_b c+ R_2+\dot R(c, ab)\\
&=\dot T_a(bc-\dot T_cb)-\dot T_a\dot R(b, c)+ R_2+\dot R(c, ab)\\
&=\dot T_a(\partial_x(\k^*\Phi))+\dot T_a\dot T_{\partial_xc}\k-\dot T_a\dot R(b, c)+ R_2+\dot R(c, ab)
\end{align*}
where $R_2=\dot T_{ab}c-\dot T_a\dot T_b c$.\\
$(i)$  The symbolic calculus Theorem \ref{theo:sc}  implies for a.e. $t\in [0, T]$
\[
\lA R_2(t)\rA_{H^s}\le K \big(\lA a(t)\rA_{W^{1,\infty}}\lA b(t)\rA_{L^\infty}+\lA a(t)\rA_{L^\infty}\lA b(t)\rA_{W^{1,\infty}}\big) \lA c(t)\rA_{H^{s-1}}.
\]
Now, from Lemma \ref{est:various} 6. and the fact that $s-1>1$ one gets $\| b\|_{L^\infty_tW^{1,\infty}_x}\le\cN$. On the other hand, Lemma \ref{est:various} 4., 8. give, respectively
\[
\lA a(t)\rA_{L^\infty}\le\cN,\quad \lA a(t)\rA_{W^{1,\infty}}\le \RHS.
\]
Applying Lemma $3.2$ in \cite{ABZ1} and Lemma \ref{est:various} 1., 6. yield
\bq\label{c:H}
\lA c(t)\rA_{L^\infty_tH^{s-1}_x}\le \cN.
\eq
Therefore, $\lA R_2(t)\rA_{H^s}$ is controllable.\\
$(ii)$ In views of Lemma \ref{est:various} 2., 4., 7. the term $\dot T_a\dot T_{\partial_xc}\k$ can be estimated by 
\[
\lA \big(\dot T_a \dot T_{\partial_xc}\k\big)(t)\rA_{H^s} \le  \mathcal{N}\lA a(t)\rA_{L^\infty}\lA \partial_xc(t)\rA_{L^\infty}\lA \partial_x^2\k(t)\rA_{H^{s-2}}\le \RHS.
\]
$(iii)$ The estimate 7. in Lemma \ref{est:various} and Sobolev's embedding imply that $\lA b\rA_{L^\infty_tC^{s-1}_*}\le \cN$. Then according to \eqref{Bony2} and the fact that $s>2$ we obtain
\[
\lA \dot T_a\dot R(b, c)(t)\rA_{H^s}\le  \mathcal{N}\lA a(t)\rA_{L^\infty}\lA b(t)\rA_{C^{s-1}_*}\lA c(t)\rA_{H^{s-1}}\les \cN.
\] 
By the same argument, to estimate $\| \dot R(ab, c)(t)\|_{H^s}$ it remains to bound $\| (ab)(t)\|_{C_*^1}$ which is in turn bounded by $\| (ab)(t)\|_{W^{1,\infty}}$. From Lemma \ref{est:various} 1. and 4. we have 
\[
\lA a(t)\rA_{L^\infty}+\lA b(t)\rA_{L^\infty}\le \cN.
\]
On the other hand, the estimate 6. (or 7.) of that lemma gives 
$
\|\partial_xb\|_{L^\infty}\le \cN.
$ Finally, we write $\partial_xa=[(\partial_t\partial_x\chi)\circ \k]\partial_x\k$ and use Lemma \ref{est:various} 8. to get $\|\partial_xa\|_{L^\infty}\le \RHS$.\\
We have proved that modulo a controllable remainder, $A_2=\dot T_{\partial_t\chi\circ\k}u$. Consequently, modulo a controllable remainder, $A_2=T_{\partial_t\chi\circ \k}u.$
Then putting together this and \eqref{dtPhi1}, \eqref{dtPhi2} we end up with the claim \eqref{dtPhi}.\\
{\bf Step 2.} With the definitions of $\R_{line}$ and $\R_{conj}$ in Theorem \ref{theo:linearization} and Theorem \ref{theo:conjugation} we write for any $h\in \Gamma^m_\tau$
\bq\label{kkgu}
\k^*T_h\Phi=T_{h^*}\k^*\Phi-\R_{line}(T_h\Phi)+T_{h^*}\R_{line}\Phi+\R_{conj}\Phi.
\eq
 It follows from Lemma \ref{est:various} 7. that 
\[
\lA \partial_x\k-1\rA_{L^\infty_tC^{s-1}_*}\le \lA \partial_x\k-1\rA_{L^\infty_tH^{s-\mez}_x}\le \cN.
\]
Therefore, $\k$ satisfies condition \eqref{assumption1} with 
\bq\label{justify:k}
\rho=1,~r_1=s-\mez,~\alpha_0=2
\eq
where we have changed the notation in \eqref{assumption1}: $\partial_x^{\alpha_0}\k\in H^{r_1+1-|\alpha_0|}$ to avoid the $r$ used in \eqref{assume:eta, psi} for the H\"older regularity of $\psi$. On the other hand, we have seen from \eqref{est:m0} that $\k'\ge m_0$ and thus the Assumptions I, II on $\k$ are fulfilled.\\
For the transport term, the symbol is $h(x, \xi)=i\xi V(x)$.\\
$(i)$ Now one can apply Theorem \ref{theo:conjugation} with $\tau=\rho=1$ (hence $\eps=\min(\tau, \rho)=1$ ) to have 
\[
h^*(x, \xi)=iV\circ \k(x)\frac{\xi}{\k'(x)}=i(V\circ \k)(\partial_x\chi\circ \k)\xi
\]
and  at a.e. $t\in [0, T]$
\[
\lA \R_{conj}\Phi\rA_{H^s}\le \cF(m_0,\lA\k'\rA_{C_*^\rho})M^1_1(h; k_0) \left(1+\lA \partial^{2}\k\rA_{H^{s-\tdm}}\right)\lA \Phi\rA_{H^s}.
\]
On the right-hand side, we estimate 
\[
\lA\k'\rA_{C_*^\rho}+\lA u\rA_{H^s}+\lA \partial^{2}\k\rA_{H^{s-\tdm}}\le \cN, \quad
M^1_1(h; k_0)\le \RHS
\]
hence,
\[
\lA \R_{conj}\Phi(t)\rA_{H^s}\le \RHS.
\]
$(ii)$ The term $T_{h^*}\R_{line}\Phi$ is bounded as
\[
\lA T_{h^*}\R_{line}\Phi(t)\rA_{H^s}\le M^1_0(h^*)\lA \R_{line}\Phi(t)\rA_{H^{s+1}}
\]
where $M^1_0(h^*)\le \cN$. Applying Theorem \ref{theo:linearization} $(ii)$ with $\Phi(t)\in C^2_*,~\sigma=r,~\eps=\min(\sigma-1, 1+\rho)_{-}\ge 1$ we have
\[
\tilde s=\min( s+\rho, r_1+1+\eps)=\min(s+1, s-\mez+1+\eps)=s+1,
\]
\[
\lA \R_{line}\Phi(t)\rA_{H^{s+1}}\le \cF(m_0,\lA \k'\rA_{C_*^\rho})\lA \partial_x^2\k\rA_{H^{s-\tdm}}\big(1+\lA \Phi'(t)\rA_{H^{s-1}}+\lA \Phi(t)\rA_{C_*^\sigma}\big)\le \RHS.
\]
(In the last inequality, we have used Lemma \ref{est:various} 1., 2.)\\
Therefore
\bq\label{TRu}
\lA T_{h^*}\R_{line}\Phi(t)\rA_{H^s}\le \RHS.
\eq
In \eqref{kkgu} we are left with the estimate for $\R_{line}(T_h\Phi)$. Notice that since $M^1_0(h)\le \cN$, with $v=T_h\Phi$ one has 
\[
\lA v(t)\rA_{H^{s-1}}\le \cN, \quad \lA v(t)\rA_{C_*^{r-1}}\le \RHS.
\]
Then, by virtue of Theorem \ref{theo:linearization} $(ii)$ applied to $v$ and $\sigma=r-1,~\eps=\min(r-2, 2)_{-}$ we have  
\[
\tilde s=\min(s+1, s-\mez+1+\eps)>s+\mez,
\]
\[
\lA \R_{line}v\rA_{H^{s+\mez}}\le \RHS.
\]
Summing up, we conclude from \eqref{kkgu} that 
\[
\k^*T_h\Phi=T_{h^*}\k^*\Phi+R_2,\quad 
\lA R_2(t)\rA_{H^s}\le \RHS.
\]
{\bf Step 3.} We now conjugate the highest order term $T_\gamma \Phi$ with $\k^*$. This is the point where we really need Theorem \ref{theo:linearization} $(i)$ for non-$C^1$ functions. Recall the formula \eqref{kkgu} and the verifications of Assumptions I, II given by \eqref{justify:k} and \eqref{est:m0}. With $c_0=(1+(\partial_x\eta))^{-1/2}$, we have that $\gamma=c_0|\xi|^{3/2}$ satisfies $M^\tdm_1(\gamma)\le \cN$. Theorem \ref{theo:conjugation} applied with $m=3/2,~\tau=1$ then yields
\[
h^*(x, \xi)=h(\k(x), \frac{\xi}{\k'(x)})=(c_0\circ \k)(x)\frac{|\xi|^\tdm}{\k'(x)}=|\xi|^\tdm
\]
for $1/\k'(x)=(\chi'\circ \k)(x)=(c_0\circ \k)(x)$; and (at a.e. $t\in [0, T]$)
\[
\lA \R_{conj}\Phi\rA_{H^{s-\tdm+1}}\le \cF(m_0,\lA\k'\rA_{C_*^\rho})M^\tdm_1(h; k_0)\left(1+\lA \partial^{2}\k\rA_{H^{s-\tdm}}\right)\lA \Phi\rA_{H^s}\le \cN.
\]
The term $T_{h^*}\R_{line}\Phi(t)$ is estimated exactly as in \eqref{TRu} noticing that $h^*$ now is of order $3/2$ we get 
\[
\| T_{h^*}\R_{line}\Phi(t)\|_{H^{s-\mez}}\le \RHS.
\]
 Consider the remaining term $\R_{line}T_{h}\Phi(t)$. Since $T_h\Phi(t)$ belongs to $C_*^{r-\tdm}$ and $r-\tdm$ can be smaller than $1$, we have to use in this case Theorem \ref{theo:linearization} $(i)$:
\[
\sigma=\mez,~ \rho+\sigma=\tdm>1,~\tilde s=\min((s-\tdm)+1, (s-\mez)+\mez)=s-\mez,
\]
\[
\lA \R_{line}T_h\Phi(t)\rA_{H^{s-\mez}}\le \cF(m_0,\lA \k'\rA_{C_*^\rho})\lA \partial_x^{2}\k\rA_{H^{s-\tdm}}\big(1+\lA T_h\Phi(t)\rA_{H^{s-\tdm}}+\lA T_h\Phi(t)\rA_{C_*^\sigma}\big).
\]
We conclude in this step that 
\[
\k^*T_\gamma \Phi=|D_x|^\tdm \k^*\Phi +R_3,\quad \lA R_3(t)\rA_{H^{s-\mez}}\le \RHS.
\]
{\bf Step 4.} Since $\omega\in \Gamma^\mez_0$ with the semi-norms bounded by $\cN$, one gets by virtue of Theorem \ref{theo:operation} and Theorem \ref{theo:linearization} $(ii)$
\[
\lA \k^*T_\omega \Phi(t)\rA_{H^{s-\mez}}\le \cN.
\]
Similarly, $f(t)\in H^s\hookrightarrow C_*^{s-\mez}$ with $s-\mez>\tdm$ we also have
\[
\lA \k^*f(t)\rA_{H^{s-\mez}}\le \RHS.
\]
Putting together the results in the previous steps, we conclude the proof of Proposition \ref{parareduce}.
\end{proof}
\begin{rema}\label{rema:mez}
In fact, in the above proof, we have proved that 
\[
\k^*(\partial_t +T_V\partial_x)\Phi(t)=(\partial_t +T_W\partial_x)\k^*\Phi(t)+f_1(t)
\]
with 
\[
\lA f_1(t)\rA_{H^s}\le \cN\left(1+\lA  \eta(t)\rA_{W^{r+\mez,\infty}}+\lA  \psi(t)\rA_{W^{r,\infty}}\right).
\]
We loose $\mez$ derivative only in Step 3 and Step 4 when conjugating $\k^*$ with $T_\gamma \Phi$ and also $T_{\omega}$, where in Step 3 we applied Theorem \ref{theo:conjugation} with $\rho=1, \tau=\frac{3}{2}$ and thus $\eps=1$. The  reason is that we want to keep the right-hand side  of \eqref{est:RHS} to be tame. On the other hand, if we apply the mentioned theorem with $\rho=\frac{3}{2}$ then it follows that 
\[
\k^*T_\gamma \Phi=|D_x|^\tdm \k^*\Phi +R_3
\]
with 
\[
\lA R_3(t)\rA_{ H^s}\le \cF(\lA \eta\rA_{L^\infty_tH^{s+\mez}_x}, \lA \psi\rA_{L^\infty_tH^s_x})\cF_1\left(1+\lA  \eta(t)\rA_{W^{r+\mez,\infty}}+\lA  \psi(t)\rA_{W^{r,\infty}}\right).
\]
If we assume more regularity: $s>2+\mez$ then by Sobolev's embedding $\lA R_3(t)\rA_{ H^s}\le \cN$ and we see again the result proved in \cite{ABZ1} (cf. Proposition $3.3$) (after performing in addition another change of variable to suppress the $\mez$ order terms).
\end{rema}
In the next paragraphs, we shall prove Strichartz estimates for $u$ solution to \eqref{eq:u}. To have an independent result, let  us restate the problem as follows. Let $I=[0,T],~s_0\in \xR$ and
\bq\label{regu:Wf}
\begin{aligned}
&W\in L^\infty([0, T]; L^{\infty}(\xR))\cap L^4([0, T]; W^{1, \infty}(\xR)),\\
 &f\in L^4(I; H^{s_0-\mez}(\xR)).
\end{aligned}
\eq
 If  $u\in L^\infty(I, H^{s_0}(\xR))$ is a solution  to the problem 
\bq\label{eq:u'}
\left(\partial_t+T_W\partial_x+i|D_x|^{\tdm}\right)u=f
\eq
we shall derive the semi-classical Strichartz estimate for $u$ (with a gain of $\frac{1}{4}-\eps$ derivatives). Remark that the same problem was considered in \cite{ABZ2} at the following regularity level
\[
W\in L^\infty([0, T];H^{s-1}(\xR)),\quad f\in L^\infty(I; H^s(\xR)),\quad s>2+\mez.
\]
We shall in fact examine the proof in \cite{ABZ2} to show that our regularity \eqref{regu:Wf} is sufficient. It turns out that for the semi-classical Strichartz estimate, {\it the loss of $\mez$ derivatives in the source term $f$ is optimal}.\\
Remark also that $u$ is defined on $\xR$ equipped with a dyadic partition of size $n_0$. Then as remarked before, up to a constant of the form  $\cF(\lA  \partial_x\eta \rA_{L^\infty_tL^\infty_x})$, which will appear in our final Strichartz estimate, we shall work as if $n_0=0$.
\subsection{Frequency localization}\label{frequency-localization}
\hk To prove Strichartz estimates for equation \eqref{eq:u'}, we will adapt the proof of Theorem $1.1$ in \cite{ABZ2}: microlocalize the solution using Littlewood-Paley theory and establish dispersive estimates for those dyadic pieces.\\
The first step consists in conjugating \eqref{eq:u'} with the dyadic operator $\Delta_j$ to get the equation satisfied by $\Delta_j u$:
\bq\label{LDu}
 \Big(\partial_t+\mez(T_W\partial_x+\partial_xT_W)+i|D_x|^{\tdm}\Big)\Delta_j u = \Delta_j f + \mez \Delta_j (T_{\partial_x W}u) +\mez \big( [T_W, \Delta_j]\partial_xu +\partial_x[T_W, \Delta_j] u\big).
\eq
After localizing $u$ at frequency $2^j$ one can replace the paradifferential operator $T_W$ by the paraproduct with $S_{j-N}(W)$ as follows
\begin{lemm}[\protect{\cite[Lemma~4.9]{ABZ4}}]\label{TS}
For all $j\ge 1$ and for some integer $N$, we have
\begin{align*}
 T_{W}\partial_x \Delta_j u&= S_{j-N}(W)\partial_x\Delta_ju + R_ju\\
\partial_xT_{W} \Delta_j u= &=\partial_xS_{j-N}(W) \Delta_j u + R'_j u
 \end{align*}
where $R_ju, R'_ju$ have spectrum contained in an annulus $\{c_12^{j}\le |\xi|\le  c_22^{j}\}$ and satisfies the following estimate for all $s_0\in \xR$:
$$
\Vert R_ju \Vert_{H^{s_0}(\xR)} + \Vert R'_ju \Vert_{H^{s_0}(\xR^d)} \leq   C(s_0)\Vert W \Vert_{W^{1,\infty}(\xR^d)} \Vert u \Vert_{H^{s_0}(\xR^d)}.
$$
\end{lemm} 
From now on, we always consider the high frequency part of $u$, that is $\Delta_ju$ with $j\ge 1$. Combining \eqref{LDu} and Lemma \ref{TS} leads to
\begin{multline}\label{replaceT:S}
 \Big(\partial_t +\mez( S_{j-N}(W) \partial_x + \partial_x   S_{j-N}(W))+i|D_x|^{\tdm}\Big)\Delta_j u =\\
   \Delta_j f+\mez \Delta_j (T_{\partial_x W}u)+\mez \big( [T_W, \Delta_j] \partial_x u +\partial_x  [T_W, \Delta_j] \big)u +R_j u + R'_ju.
 \end{multline}
Next, as in \cite{BaCh}, \cite{TataruNS}, \cite{ABZ4} we smooth out the symbols (see for instance Lemma $4.4$, \cite{ABZ4})
\begin{defi}\label{Sdelta}
Let $\delta >0$ and $U\in \mathcal{S}'(\xR)$. For any $j\in \xZ,~j\ge -1$ we define 
\[
S_{\delta j}(U)=\psi(2^{-\delta j}D_x)U.
\]
\end{defi}
Let $\chi_0\in C_0^{\infty}(\xR)$, $\supp \chi\subset \{\frac{1}{4}\le |\xi|\le 4\}$, $\xi=1$ in $\{\frac{1}{2}\le |\xi|\le 2\}$. Define 
\begin{equation}\label{Ldelta}
\begin{cases}
a(\xi)=\chi_0(\xi)|\xi|^{\frac{3}{2}},~h=2^{-j},\\
\mathcal{L}_\delta = \partial_t +\mez( S_{(j-N)\delta}(W)\cdot \partial_x +\partial_x \cdot  S_{\delta (j-N)}(W)) +i\chi_0(h\xi)|D_x|^{\frac{3}{2}}.
\end{cases}
\end{equation}
Using \eqref{replaceT:S}, we have
\begin{equation}\label{Ldj}
\mathcal{L}_\delta  \Delta_ju = F_j,\quad\text{where}
\end{equation}
  \begin{multline}\label{Fj}
F_j = \Delta_jf+\mez \Delta_j (T_{\partial_x W}u)+ \mez \big( [T_W, \Delta_j] \partial_x u +\partial_x[T_W, \Delta_j] u\big) +R_j u + R'_ju +\\ \mez \Big\{ \big(S_{(j-N)\delta}(W)  - S_{(j-N)}(W)\big)\partial_x \Delta_ju +\partial_x  \big(S_{(j-N)\delta}(W)  - S_{(j-N)}(W)\big) \Delta_ju\Big\}.
\end{multline}
\subsection{Semi-classical parametrix and dispersive estimate}\label{parametrix}
 Recall that $\varphi$ is the cut-off function employed to defined the dyadic partition of size $n=0$ in paragraph \ref{subsection:dyadic}. To simplify the presentation, let us rescale the existence time to $T=1$ and set $
h=2^{-j},~j\ge 1$,
\[
E_0=L^\infty([0, T]; L^{\infty}(\xR)), \quad E_1=L^4([0, T]; W^{1,\infty}(\xR).
\]
The main result of this paragraph is the following semi-classical dispersive estimate for the operator $\mathcal{L}_\delta$.
 \begin{theo}\label{dispersive}
 Let $\delta<\frac{1}{2}$ and $t_0 \in \xR$. For any  $u_0 \in L^1(\xR^d)$ set $u_{0,h} = \varphi(hD_x)u_0$. 
Denote by $S(t,t_0)u_{0,h}$ solution of the problem
\begin{equation*}
\mathcal{L}_{\delta}u_h(t,x) =0,  \quad u_{h}(t_0,x) = u_{0,h}(x).
\end{equation*}
Then there exists  $\mathcal{F}: \xR^+ \to \xR^+$  such that  
\bq\label{dispersive:est}
\Vert S(t, t_0)u_{0,h } \Vert_{L^\infty(\xR^d)} 
\leq  \mathcal{F}(\Vert W\Vert_{E_0} )\,   h^{-\frac{1}{4}}\vert t-t_0 \vert^{-\frac{1}{2}} \Vert u_{0,h } \Vert_{L^1(\xR^d)}
\eq
for all $0<\vert t-t_0 \vert\leq h^{\mez}$ and $0<h\le 1$.
\end{theo}
We make the change of temporal variables $t=h^\mez\sigma$ and set 
\bq\label{change:t}
\quad W_h(\sigma, x)=S_{(j-N)\delta}(W)(\sigma h^\mez, x),
\eq
and denote the obtained semiclassical operator by
\bq
\label{L}
L_{\delta}=h\partial_\sigma +h^\mez W_h(h\partial_x) +\mez h\partial_xW_h +ia(hD_x).
\eq
For this new differential operator, we shall prove the the corresponding (classical) dispersive estimate:
 \begin{theo}\label{dispersive'}
 Let $\delta<\frac{1}{2}$ and $\sigma_0 \in [0, 1]$. For any  $u_0 \in L^1(\xR^d)$ and  $u_{0,h} = \varphi(hD_x)u_0$. 
Denote by $\widetilde S(\sigma, \sigma_0)u_{0,h}$ solution of the problem
\begin{equation*}
L_{\delta}U_h(\sigma,x) =0,  \quad U_{h}(\sigma_0,x) = u_{0,h}(x).
\end{equation*}
Then there exists  $\mathcal{F}: \xR^+ \to \xR^+$  such that  
\bq\label{dispersive:estimate}
\Vert \widetilde S(\sigma, \sigma_0)u_{0,h } \Vert_{L^\infty(\xR^d)} 
\leq  \mathcal{F}(\Vert W\Vert_{E_0} )\,   h^{-\frac{1}{2}}\vert \sigma-\sigma_0 \vert^{-\frac{1}{2}} \Vert u_{0,h } \Vert_{L^1(\xR^d)}
\eq
for all $\sigma\in [0, 1]$.
\end{theo}
Theorem \ref{dispersive'} will imply Theorem \ref{dispersive}. Indeed,  the relation
\[
 L_\delta u_h(\sigma, x)=h^\tdm\mathcal{L}_{\delta}u_{h}(\sigma h^\mez, x),
\]
yields
\[
 \widetilde S(\sigma, \sigma_0)u_{0, h}(x)= S(h^\mez\sigma, h^\mez \sigma_0)u_{0, h} ( x).
\]
If Theorem \ref{dispersive'} were proved then via the relation $t=\sigma h^\mez$,
\begin{align*}
\Vert S(t, t_0)u_{0, h}\Vert_{L^\infty_x}&=\Vert \widetilde S(\sigma, \sigma_0)u_{0, h}\Vert_{L^\infty_x}\\
&\le \mathcal{F}(\Vert W\Vert_{E_0} )\,   h^{-\frac{1}{2}}\vert \sigma-\sigma_0 \vert^{-\frac{1}{2}} \Vert u_{0,h } \Vert_{L^1(\xR^d)}\\
&\le\mathcal{F}(\Vert W\Vert_{E_0} )\,   h^{-\frac{1}{4}}\vert t-t_0 \vert^{-\frac{1}{2}} \Vert u_{0,h } \Vert_{L^1(\xR^d)}
\end{align*}
which proves Theorem \ref{dispersive}.\\
To prove Theorem \ref{dispersive'}, we use the WKB method to construct a parametrix of the following integral form
\bq
\label{parametrix-form}
\widetilde U_h(\sigma,x)=\frac{1}{2\pi h}\iint e^{\frac{i}{h}(\varphi(\sigma, x, \xi, h)-z\xi)}\widetilde b(\sigma, x, z, \xi, h)u_{0, h}(z)dzd\xi
\eq 
where \\
$(i)$ the phase $\varphi$ satisfies $\varphi(\sigma=0)=x\xi$,\\
$(ii)$ the amplitude $\widetilde b$ has the form
\bq
\label{b-form}
\widetilde b(\sigma, x, \xi, h)=b(\sigma, x, \xi, h)\zeta(x-z-\sigma a'(\zeta))
\eq
with $\zeta\in C_0^{\infty}(\xR),~\zeta(s)=1$ if $|s|\le 1$ and $\zeta(s)=0$ if $|s|\ge 2$.\\
We shall work with the following class of symbols.
\begin{defi}
For small $h_0$ to be fixed, we set
\[
\mathcal{O}=\left\{(\sigma, x, \xi, h)\in \xR^4: h\in (0, h_0), |\sigma|<1, 1<|\xi|<3\right\}.
\]
If $m\in \xR$ and $\rho\in \xR^+$, we denote by $S_{\rho}^m(\mathcal{O})$  the set of all functions $f$ defined on $\mathcal{O}$ which are $C^\infty$ with respect to $(\sigma, x, \xi)$ and satisfy
\[
\vert \partial_x^\alpha f(\sigma, x, \xi, h)\vert \le C_\alpha h^{m-\alpha\rho}
,\quad \forall \alpha\in \xN,~\forall (\sigma, x, \xi, h)\in \mathcal{O}.
\]
\end{defi}
\begin{rema}
Recall that 
\[
W_h(\sigma, x)=S_{(j-N)\delta}(W)(\sigma h^\mez, x)\equiv\phi(2^{-(j-N)\delta} D_x)W(\sigma h^\mez, x).
\]
Hence, for any $\alpha\in \xN$, there hold
\bq
\label{Vh}
\begin{aligned}
&\vert \partial_x^\alpha W_h(\sigma, x)\vert\le C_{\alpha}h^{-\delta\alpha}\Vert W(\sigma h^\mez, \cdot)\Vert_{L^{\infty}}, \\
&\vert \partial_x^{\alpha+1} W_h(\sigma, x)\vert\le C_{\alpha}h^{-\delta\alpha}\Vert W(\sigma h^\mez, \cdot)\Vert_{W^{1,\infty}}.
\end{aligned}
\eq
\end{rema}
The following result for transport problems is elementary.
\begin{lemm}\label{est:transport}
If $v$ is a solution of the problem 
\[
\left(\partial_\sigma+m(\xi)\partial_x+if\right)v(\sigma, x, \xi)=g(\sigma, x, \xi),\quad u\arrowvert_{\sigma=0}=z\in \xC,
\]
where $f$ be real-valued, then $v$ satisfies
\[
\la v(\sigma, x, \xi)\ra\le |z|+\int_0^\sigma \la g(\sigma', x+(\sigma'-\sigma)a'(\xi), \xi)\ra\di \sigma'.
\]
\end{lemm}
The existence of the parametrix is given in the following Proposition.
\begin{prop}\label{parametrix:prop}
There exists a phase $\varphi$ of the form 
\[
\varphi(\sigma, x, \xi, h)=x\xi-\sigma a(\xi)+h^{\mez}\psi(\sigma, x, ,\xi, h)
\]
with $\partial_x\psi\in S^0_{\delta}(\mathcal{O})$ and there exists a symbol $b\in S^0_{\delta}(\mathcal{O})$ such that with the amplitude $\widetilde b$ defined by \eqref{b-form}, we have
\bq\label{L:parametrix}
L_{\delta}\left(e^{\frac{i}{h}\phi}\widetilde b\right)=e^{\frac{i}{h}\phi}r_h,
\eq
where for any $N\in \xN$ there holds
\begin{multline}
\label{est:r}
\sup_{\sigma\in [0,1]}\left \|\iint e^{\frac{i}{h}(\varphi(\sigma, x, \xi, h)-z\xi)}r(\sigma, x, z, \xi, h)u_{0, h}(z)dzd\xi\right\|_{H^1{(\xR_x)}}\cr 
\le h^N\cF_{N}\left(\Vert W\Vert_{E_0}+\Vert W\Vert_{E_1}\right)\Vert u_{0,h}\Vert_{L^1(\xR)}.
\end{multline}
\end{prop}
\begin{proof} We proceed in several steps.\\
{\bf Step 1.} Construction of the phase $\varphi$.\\
We find $\varphi$ under the form 
\bq\label{form:varphi}
\varphi(t, x, \xi, h)=x\xi-\sigma a(\xi)+h^{\mez}\psi(\sigma, x, ,\xi, h)
\eq
where $\psi$ solves the following transport problem
\bq\label{eq:psi}
\begin{cases}
\partial_{\sigma}\psi+a'(\xi)\partial_x\psi=-\xi W_h,\\
\psi\arrowvert_{\sigma=0}=0.
\end{cases}
\eq
Differentiating \eqref{eq:psi} with respect to $x$ and $\xi$ then using Lemma \ref{est:transport}  together with \eqref{Vh} and H\"older's inequality we derive 
\bq\label{est:psi}
\vert \partial_{\xi}^k\partial_x^\alpha \psi(\sigma, x, \xi, h)\vert\le C_{k \alpha}|\sigma|^{\frac{3}{4}}h^{-\delta(\alpha+k-1)^+}\Vert W\Vert_{L^4([0, T], W^{1,\infty}_x)}, 
\eq
for every $(\alpha, k)\in \xN^2$, for every $(\sigma, x, \xi, h)\in \mathcal{O}$; where $m^+=\max\{m ,0\}$. \\
Remark that in \cite{ABZ2} where $W\in L^{\infty}([0, T], W^{1,\infty}(\xR))$, one has the better estimate 
\bq\label{est:psi'}
\vert \partial_{\xi}^k\partial_x^\alpha \psi(\sigma, x, \xi, h)\vert\le C_{k \alpha}|\sigma|h^{-\delta(\alpha+k-1)^+}\Vert W\Vert_{L^\infty([0, T], W^{1,\infty}_x)}.
\eq
However, \eqref{est:psi} is enough to get $\partial_x\psi\in S^0_{\delta}(\mathcal{O})$. Consequently, the estimates from $(4.17)$ to $(4.30)$ in \cite{ABZ2} still hold  and thus we have by $(4.29)$, \cite{ABZ2}
\bq\label{remainder:r}
r=h\left(\partial_\sigma b+a'(\xi)\partial_x b+if b+h^{\mu_0}\sum_{l=0}^{M_1}e_l(h^\delta\partial_x)^lb \right)\zeta+i\sum_{j=1}^4r_j
\eq
with $e_l\in S^0_{\delta}(\mathcal{O})$,
\bq\label{parametrix:f}
\mu_0=\mez(\mez-\delta)>0, \quad f=W_h\partial_x\psi+a''(\xi)(\partial_x\psi)^2~(\text{real valued});
\eq 
and with
\[
\rho(x, y)=\int_0^1\partial_x\varphi(\sigma, \lambda x+(1-\lambda)y, \xi, h)d\lambda,
\]
the remainders $r_i's$ are then given by 
\bq
r_1=ch^{M-1}\iiint_0^1e^{\frac{i}{h}(x-y)\eta}\kappa_0(\eta)(1-\lambda)^{M-1}\partial_{y}^M\left\{a^{(M)}(\lambda\eta+(\rho(x, y))\widetilde b(y)\right\}d\lambda dy d\eta,
\eq
\bq
r_2=\sum_{k=0}^{M-1}c_{k, M}h^{M+k}\iint_0^1z^M\hat\kappa_0(z)(1-\lambda)^{M-1}\partial_{y}^{M+k}\left\{a^{(k)}((\rho(x, y))\widetilde b(y)\right\}_{y=x-\lambda hz}d\lambda dz.
\eq
\bq
r_3=\sum_{k=0}^{M-1}\sum_{j=1}^kc'_{j,k}h^k\partial_y^{k-j}\left\{(\partial_\xi^ka)(\rho(x, y))b(y) \right\}\arrowvert_{y=x}\zeta^{(j)}.
\eq
\bq
r_4=\frac{1}{i}h\left\{-a'(\xi)+h^\mez W_h \right\}b\zeta'
\eq
where $c, c_{k, M}, c'_{jk}$ are constants and $\kappa_0\in C^\infty_0(\xR),~\kappa=1$ in a neighborhood of the origin.\\
Now, combining \eqref{est:psi} with the fact that $W_h\in S_{\delta}^0(\mathcal{O})$ (by \eqref{Vh}) we obtain the following estimate for $f$
\bq
\label{est:f}
\vert \partial_{\xi}^k\partial_x^\alpha f(\sigma, x, \xi, h)\vert\le |\sigma|^{\frac{3}{4}}h^{-\delta(\alpha+k)}\cF_{k \alpha}\left(\Vert W\Vert_{L^4([0, T], W^{1,\infty}_x)}\right)\Vert W\Vert_{L^\infty([0, T], L^{\infty}_x)},
\eq
 $\forall (\alpha, k)\in \xN^2, ~\forall (\sigma, x, \xi, h)\in \mathcal{O}$.\\
{\bf Step 2.} Construction of the amplitude $b$. According to the WKB method, ones find  $b$ under the form 
\bq\label{b:form}
b=\sum_{j=0}^{M-1}h^{j\mu_0}b_j
\eq
where $b_0$ solves
\[
\begin{cases}
\partial_\sigma b_0+a'(\xi)\partial_xb_0+ifb_0=0,\\
b_0\arrowvert_{\sigma=0}=\chi_1(\xi)
\end{cases}
\]
and $b_j's,~ j\ge 1$ solves
\[
\begin{cases}
\partial_\sigma b_j+a'(\xi)\partial_xb_j+ifb_j=-\sum_{l=0}^{M_1}e_l(h^\delta\partial_x)^lb_{j-1},\\
b_j\arrowvert_{\sigma=0}=0.
\end{cases}
\]
Owing to Lemma \ref{est:transport} and the estimate \eqref{est:f}, one can use induction for the preceding transport problems (see Lemma $4.7$, \cite{ABZ2}) to have
\bq
b_j(\sigma, x, \xi, h)=\chi_1(\xi)c_j(\sigma, x, \xi, h),\quad\forall 0\le j\le J-1
\eq
and the $c_j$ satisfies $\forall (\alpha, k)\in \xN^2,~\forall (\sigma, x, \xi, h)\in \mathcal{O}$,
\bq
\label{est:cj}
\vert \partial_{\xi}^k\partial_x^\alpha c_j(\sigma, x, \xi, h)\vert\le h^{-\delta(\alpha+k)}\cF_{jk \alpha  }\left(\Vert W\Vert_{E_0}+\Vert W\Vert_{E_1}\right).
\eq
{\bf Step 3.} Estimate for the remainder $r$. \\
Plugging \eqref{b:form} into \eqref{remainder:r} we obtain 
$r=\sum_{j=0}^5r_j$ with $r_5=h^{M\mu_0}b_{M-1}\zeta.$
We want to prove \eqref{est:r}, i.e, for a.e. $t\in [0, T]$ and for all $j=1,..5$,
\begin{multline}
\label{est:r'}
\left \|\iint e^{\frac{i}{h}(\varphi(\sigma, x, \xi, h)-z\xi)}r(\sigma, x, z, \xi, h)u_{0, h}(z)dzd\xi\right\|_{H^1{(\xR_x)}}\cr 
\le h^N\cF_N\left(\Vert W\Vert_{E_0}+\Vert W\Vert_{E_1}\right)\Vert u_{0,h}\Vert_{L^1(\xR)}.
\end{multline}
Let us denote the function inside the norm on the left-hand side by $F_h^j$. 
The proofs for $\Vert F_h^j\Vert_{H^1_x}$, $j=1,2,3, 5$ remain unchanged compared to section those in $4.1.1$, \cite{ABZ2} (using integration by parts). The only point that we need to take care is the estimate for $\|F_4\|_{H^1}$ since $r_4$ contains $W_h$ which is  less regular than it was in \cite{ABZ2}. Recall that
\[r_4=\frac{1}{i}h\left\{-a'(\xi)+h^\mez W_h \right\}b\zeta'.
\]
On the support of all derivatives of $\zeta$ one has $\vert x-z-\sigma a'(\xi)\vert\ge 1$. Now, by \eqref{est:psi} 
\[
h^\mez\partial_x\psi\le Ch^\mez|\sigma|^{\frac{3}{4}}\le ch^\mez
\]
 hence using \eqref{form:varphi} we deduce that 
\[
\vert \partial_\xi(\varphi(\sigma, x, \xi, h)-z\xi)\vert=\vert x-z-\sigma a'(\xi)-h^\mez\partial_\xi\psi\vert\ge \mez
\]
 for $h$ small enough. Therefore, we can integrate by parts $N$ times in the integral defining $F_4$  using the vector filed
\[
L=\frac{h}{i\partial_\xi(\varphi(\sigma, x, \xi, h)-z\xi)}\partial_\xi.
\]
Taking into account the fact that for all $\alpha\in \xN$, on the support of $\zeta$,  $\langle x-z-\sigma a'(\xi)\rangle\le C$ and  (due to \eqref{Vh}, \eqref{est:cj} and \eqref{est:psi})
\begin{align*}
&\vert \partial^\alpha_\xi r_4(\sigma, x, \xi, h)\vert \le C(1+\|W_h(\sigma)\|_{L^{\infty}_x})h^{1-\alpha \delta}\cF_{\alpha  }\left(\Vert W\Vert_{E_0}+\Vert W\Vert_{E_1}\right),\\
&\vert \partial^{\alpha+1}_\xi (\varphi(\sigma, x, \xi, h)-z\xi)\vert \le C(1+\Vert W\Vert_{E_1})h^{-\alpha \delta}
\end{align*}
 we obtain 
\begin{align*}
\Vert F_h^4(\sigma)\Vert_{L^2_x}&\le h^{1+N(1-\delta)}(1+\|W_h(\sigma)\|_{L^{\infty}_x})\cF_{\alpha  }\left(\Vert W\Vert_{E_0}+\Vert W\Vert_{E_1}\right)\times \\
&\quad \times\int|u_{0, h}(z)|dz\int |\chi_1(\xi)|d\xi\\
&\le h^{1+N(1-\delta)}\cF_N\left(\Vert W\Vert_{E_0}+\Vert W\Vert_{E_1}\right)\Vert u_{0, h}(z)\Vert_{L^1_x}.
\end{align*}
Similarly, one gets
\begin{align*}
\Vert \partial_xF_h^4(\sigma)\Vert_{L^2_x}&\le h^{1+N(1-\delta)}(1+\|\partial_xW_h(\sigma)\|_{L^{\infty}_x})\cF_N\left(\Vert W\Vert_{E_0}+\Vert W\Vert_{E_1}\right)\Vert u_{0, h}(z)\Vert_{L^1_x}\\
&\le h^{1+N(1-\delta)}(1+h^{-\delta}\|W(\sigma h^\mez)\|_{L^{\infty}_x})\cF_N\left(\Vert W\Vert_{E_0}+\Vert W\Vert_{E_1}\right)\Vert u_{0, h}(z)\Vert_{L^1_x}\\
&\le h^{(N+1)(1-\delta)}\cF_N\left(\Vert W\Vert_{E_0}+\Vert W\Vert_{E_1}\right)\Vert u_{0, h}(z)\Vert_{L^1_x}.
\end{align*}
Therefore, we end up with 
\[
\sup_{\sigma\in [0, 1]}\Vert F_h^4(\sigma)\Vert_{H^1(\xR)}\le  h^{N(1-\delta)}\cF_N(\|W\|_{E_0}+\|W\|_{E_1})\Vert u_{0, h}(z)\Vert_{L^1_x},
\]
which concludes the proof.
\end{proof}
Having established the previous Proposition, we turn to prove Theorem \ref{dispersive'}.\\
{\bf Proof of Theorem \ref{dispersive'}}
Without loss of generality, we take $\sigma_0=0$. By a scaling argument, it suffices to prove the dispersive estimate \eqref{dispersive:estimate} for the operator $\widetilde S$ for $\sigma=1$. Indeed, let $\sigma_1\in (0, 1]$, making the following changes of variables
\[
\tau=\frac{\sigma}{\sigma_1},~\bar x=\frac{x}{\sigma_1},~\bar h=\frac{h}{\sigma_1}
\]
we see that the operator $L_\delta$ becomes 
\[
\bar L_\delta=\bar h\partial_\tau+\bar h^\mez \bar W_h(\bar h\partial_{\bar x})+\mez \bar h^\tdm(\partial_{\bar x} \bar W_h)+i\vert\bar hD_{\bar x}\vert^\tdm
\]
where 
\[
\bar W_h(\tau, \bar x)=\sigma_1^\mez W_h(\sigma_1 \tau, \sigma_1\bar x).
\]
Observe that there exists $C>0$ independent of $\sigma_1\in (0, 1]$ for which there holds
\[
\lA \bar W_h\rA_{E_0}+\lA \bar W_h\rA_{E_1}\le C.
\]
Suppose that the dispersive estimate \eqref{dispersive:estimate} for $L_\delta$ were proved for $\sigma=1$, it then would imply the same estimate for $\bar L_\delta$ for $\tau= 1$. Calling $\bar S$ the propagator of $\bar L_\delta$, we have for all $\sigma \in [0 ,1]$
\[
\widetilde S(\sigma, 0)u_0(x)=(\bar S(\frac{\sigma}{\sigma_1})\bar u)(\frac{x}{\sigma_1}),\quad  \bar u(\frac{x}{\sigma_1})=u_0(x).
\]
Taking $\sigma=\sigma_1$ then it would follow that 
\[
\lA \widetilde S(\sigma_1)u_0\rA_{L^\infty(\xR)}=\lA \bar S(\frac{\sigma_1}{\sigma_1})\bar u\rA_{L^\infty(\xR)}\le \frac{C}{\bar h^\mez}
\lA \bar u_0\rA_{L^1(\xR)}\le \frac{C \sigma_1^\mez}{h^\mez \sigma_1}
\lA  u_0\rA_{L^1(\xR)} \le \frac{C}{| h \sigma_1|^\mez}
\lA u_0\rA_{L^1(\xR)},
\]
which is the estimate \eqref{dispersive:estimate} for $L_\delta$  for $\sigma=\sigma_1$. Therefore, it suffices to prove \eqref{dispersive:estimate} for $\sigma=1$.\\
 Now, combining \eqref{parametrix-form} and Proposition \ref{parametrix:prop} yields 
\bq\label{dispersive:1}
L_\delta \widetilde U_h(\sigma, x)=F_h(\sigma, x)
\eq
with
\bq
\label{est:Fh}
\sup_{\sigma\in [0, 1]}\Vert F_h(\sigma)\Vert_{H^1_x(\xR))}\le C_Nh^N\cF_{N}\left(\Vert W\Vert_{E_0}+\Vert W\Vert_{E_1}\right)\Vert u_{0,h}\Vert_{L^1(\xR)}.
\eq
Using integration by parts we can show that $\widetilde U_h$ is a good parametrix at the initial time (see $(4.53)$, \cite{ABZ2}) in the following sense
\bq\label{decomposeU0}
\widetilde U_h(0, \cdot)=u_{0,h}+v_{0,h},\quad \Vert v_{0,h}\Vert_{H^1(\xR)}\le C_Nh^N\Vert u_{0,h}\Vert_{H^1(\xR)}.
\eq
Combining \eqref{dispersive:1}, \eqref{decomposeU0} and the Duhamel formula gives
\[
S(\sigma, 0)u_{0,h}=R_1+R_2+R_3
\]
where
\[
\begin{cases}
R_1=\widetilde U_h(\sigma, x),\\
R_2=-S(\sigma, 0)v_{0,h},\\
R_3=-\int_0^{\sigma} S(\sigma, r)[F_h(r, x)]dr.
\end{cases}
\]
We shall successively estimate $R_i$. First, by Sobolev's inequalities and \eqref{decomposeU0},
\[
\Vert R_2(\sigma)\Vert_{L^\infty_x}\le C \Vert S(\sigma, 0)v_{0,h}\Vert_{H^1_x}=C\Vert v_{0,h}\Vert_{H^1_x}\le C_Nh^N\Vert u_{0,h}\Vert_{L^1}.
\]
Next, for $R_3$ we estimate
\[
\Vert R_3(\sigma)\Vert_{L^\infty_x}\le \int_0^{\sigma} \Vert S(\sigma, r)[F_h(r, x)]\Vert_{H^1_x}dr \le  \int_0^{\sigma}\Vert F_h(r, x)\Vert_{H^1_x}dr.
\]
Then, by virtue of the estimate \eqref{est:Fh} we deduce that
\[
\Vert R_3(\sigma)\Vert_{L^\infty_x}\le h^N\cF_N\left(\Vert W\Vert_{E_0}+\Vert W\Vert_{E_1}\right)\Vert u_{0,h}\Vert_{L^1(\xR)}.
\]
Finally, from \eqref{parametrix-form} we have
\[
\widetilde U_h(\sigma, x)=\int K(\sigma, x, z, h)u_{0,h}(z)dz
\]
with
\[
K(\sigma, x, z, h)=\frac{1}{2\pi h}\int e^{\frac{i}{h}(\varphi(\sigma, x, \xi, h)-z\xi)}\widetilde b(\sigma, x, z, \xi, h)d\xi.
\]
Because $\sigma=1$ is fixed, the proof of Proposition $4.8$, \cite{ABZ2} still works and we obtain for some $\cF:\xR^+\to \xR^+$ independent of all parameters
\[
\vert K(1, x, z, h)\vert\le \frac{1}{h^\mez}\cF\left(\Vert W\Vert_{E_0}+\Vert W\Vert_{E_1}\right).
\]
This gives
\[
\Vert R_1(1)\Vert_{L^\infty_x}=\Vert \widetilde U_h(1)\Vert_{L^\infty_x}\le h^{-\mez}\cF\left(\Vert W\Vert_{E_0}+\Vert W\Vert_{E_1}\right)\Vert u_{0,h}\Vert_{L^1}.
\]
The proof is complete.
\subsection{The semi-classical Strichartz estimate}
Combining the dispersive estimate \eqref{dispersive:est}  with the usual $TT^*$ argument and Duhamel's formula, we derive the Strichartz estimate on small time interval $[0, h^{\mez}]$.
\begin{coro}\label{semicl}
 Let $I_h=[0, h^{\mez}]$ and $u$ be a solution to the problem
\[
\mathcal{L}u(t, x)=f(t, x), \quad u(0, x)=0
\] 
with $\supp \hat{f}\subset \{c_1h^{-1}\le |\xi|\le c_2h^{-1}\}$. Then there exists  $\mathcal{F}: \xR^+ \to \xR^+$ (independent of $u,~f,~W,~h$) such that  
\[
\Vert u \Vert_{L^{4}(I_h, L^{\infty}(\xR ))} 
\leq h^{-\frac{1}{8}}\cF\left(\Vert W\Vert_{E_0}+\Vert W\Vert_{E_1}\right)\Vert f \Vert_{L^1(I_h, L^2 (\xR))} .
\]
\end{coro}
Finally, we glue these estimates together both in frequency and in time to obtain the semi-classical  Strichartz estimate for $u$ on $[0, T]$. 
\begin{theo}\label{strich2}
Let $I=[0,T]$ and $s_0\in \xR$. Let $W\in E_0\cap E_1$ and $f\in L^4(I; H^{s_0-\mez}(\xR))$. If $u\in L^\infty(I, H^{s_0}(\xR))$ is a solution  to the problem 
\[
\left(\partial_t+T_W\partial_x+i|D_x|^{\tdm}\right)u=f,
\]
 then for every $\eps>0$, there exists $\cF_{\eps}$ (independent of $u,~f,~W$) such that 
\bq\label{Strichartz}
 \Vert  u\Vert_{L^4(I; C_*^{s_0- \frac{1}{4}-\eps}(\xR))} \leq \cF_\eps\left(\Xi\right) \Big(\Vert f\Vert_{L^4(I; H^{s_0-\frac{1}{2}-\eps}(\xR))}
+ \Vert u\Vert_{L^\infty(I; H^{s_0}(\xR))}\Big),
\eq
where 
\[
\Xi=\Vert W\Vert_{E_0}+\Vert W\Vert_{E_1}+\lA  \partial_x\eta \rA_{L^\infty_tL^\infty_x}.
\]
 \end{theo}
\begin{proof} Throughout this proof, we denote 
$
\cF=\cF(\Vert W\Vert_{E_0}+\Vert W\Vert_{E_1})
$ and $\RHS$ the right-hand side of \eqref{Strichartz}. Remark first that by \eqref{Ldj} we have
$
\mathcal{L}_{\delta}u_h=F_h,$
where $F_h$ is given by \eqref{Fj}.\\
{\bf Step 1.} Let $\chi \in C_0^\infty(0,2)$ equal to one on $[ \mez, \frac{3}{2}].$ For $0 \leq k \leq [Th^{-1}]-2$ define
$$
I_{h,k} = [kh^{\mez}, (k+2)h^{\mez}], \quad \chi_{h,k}(t) = \chi\Big( \frac{t-kh^{\mez}}{h^{\mez}}\Big), \quad  u_{h,k} = \chi_{h,k}(t)u_h.
$$
Then
\begin{equation*} 
\mathcal{L}_{\delta}u_{h,k} = \chi_{h,k}F_h   +h^{-\mez}\chi'\Big( \frac{t-kh^{\mez}}{h^{\mez}}\Big)u_h, \quad u_{h,k}(kh,\cdot)=0.
\end{equation*}
Applying Corollary \ref{semicl} to each $u_{h,k}$ on the interval $I_{h,k}$   we obtain, since $\chi_{h,k}(t) = 1$ for $(k+\mez)h\leq t \leq (k+\frac{3}{2})h$, 
\begin{equation*} 
\begin{aligned}
\Vert &u_h\Vert_{L^4(  (k+\mez)h^{\mez},(k+\frac{3}{2})h^{\mez});L^{\infty}(\xR))}  \\
& \leq h^{-\frac{1}{8}}\mathcal{F}.\Big( \Vert   F_h \Vert_{L^1( I_{h,k}; L^2(\xR))} 
+ h^{-\mez} \Vert \chi'\Big( \frac{t-kh^{\mez}}{h^{\mez}}\Big)u_h \Vert_{L^1(I_{h,k}; L^2(\xR))}\Big)\\
&\leq h^{-\frac{1}{8}}\mathcal{F}. \Big(h^{\frac{3}{8}} 
\Vert   F_h \Vert_{L^4( I_{h,k}; L^2(\xR))} 
+ \Vert u_h\Vert_{L^\infty(I; L^2(\xR))}\Big).
\end{aligned}
\end{equation*}
Raising to the power $4$ both sides of the preceding estimate, summing in $k$ from $0$ to  $[Th^{-\mez}]-2$ and then taking the power $1/4$ we get
\begin{equation}\label{global3}
\Vert u_h\Vert_{L^4(I;L^{\infty}(\xR))}  
  \leq \mathcal{F}. \Big(
 h^{\frac{1}{4}}\Vert F_h \Vert_{L^4(I; L^2(\xR))} 
+ h^{-\frac{1}{4}}\Vert u_h\Vert_{L^\infty(I; L^2(\xR))}\Big).
\end{equation}
Set $\nu=\mez-\delta$. Multiplying both sides of the above inequality by $h^{-s_0+\frac{1}{4}+\nu}$ and taking into account the fact that $u_h$ and $F_h$ are spectrally supported in annulus of size $h^{-1}$, it follows that
\begin{equation}\label{global4}
\Vert u_h\Vert_{L^4(I; L^\infty(\xR))} h^{-s_0+\frac{1}{4}+\nu}
  \leq \mathcal{F}.\Big(
\Vert   F_h \Vert_{L^4(I; H^{s_0-1+\delta}(\xR))} 
+ \Vert u_h\Vert_{L^\infty(I; H^{s_0-\nu}(\xR))}\Big).
\end{equation}
{\it Step 2.} We now estimate $\Vert   F_h \Vert_{L^4(I; H^{s_0-1+\delta}(\xR))}$, where recall from \eqref{Fj} that
  \begin{multline}
F_h=\Delta_jf+\mez \Delta_j (T_{\partial_x W}u)+ \mez \big( [T_W, \Delta_j] \partial_x u +\partial_x[T_W, \Delta_j] u\big) +R_j u + R'_ju +\\ \mez \big\{ \big(S_{(j-N)\delta}(W)  - S_{(j-N)}(W)\big)\partial_x \Delta_ju +\partial_x  \big(S_{(j-N)\delta}(W)  - S_{(j-N)}(W)\big) \Delta_ju\big\}.
\end{multline}
Since $W\in L^4(I, W^{1,\infty}(\xR))$, we can apply the symbolic calculus Theorem \ref{theo:sc} $(i)$, $(ii)$ to have
\bq\label{TV-D}
\begin{aligned}
&\Vert \Delta_j (T_{\partial_xW}u)\Vert_{L^4(I; H^{s_0-1+\delta}(\xR))} +\Vert [ T_{W} \partial_x+ \partial_xT_{W}, \Delta_j]u\Vert_{L^4(I; H^{s_0-1+\delta}(\xR))} \\
&\quad \leq C\Vert W \Vert_{L^4(I; W^{1, \infty}(\xR))} \Vert u \Vert_{L^\infty(I; H^{s_0-1+\delta}(\xR))}.
\end{aligned}
\eq
Next, remark that the spectrum of $\Lambda_j\defn\big(S_{(j-N)\delta}(W)  - S_{j-N}(W)\big)\partial_x \Delta_ju $ is contained in a ball of radius $C\, 2^j$ we can write for fixed $t$
\begin{equation*}
\begin{aligned}
\Vert \Lambda_j(t, \cdot) \Vert_{H^{s_0-1+\delta}(\xR)} &\le C\, 2^{j(s_0-1+\delta)} 
\Vert  (S_j(W) - S_{j \delta}(W)) \partial_x\Delta_ju_h)(t,\cdot) \Vert_{L^2(\xR)}\\
& \le C\, 2^{j(s_0-1+\delta)}\Vert  (S_j(W) - S_{j \delta}(W)(t,\cdot) \Vert_{L^\infty(\xR)} 
2^{j(1-s_0)}  \Vert u_h(t,\cdot) \Vert_{H^{s_0}(\xR)}. 
\end{aligned}
\end{equation*}
 According to the convolution formula,
$$
(S_j(W) - S_{j \delta}(W))(t,x) 
= \int_{\xR^d} {\check\phi}(z) \big(W(t,x-2^{-j }z) - W(t,x-2^{-j\delta} z) \big) \,dz.
$$
where $\check\phi$ is the inverse Fourier transform of the Littlewood-Paley function $\phi$. It follows that
$$
\Vert (S_j(W) - S_{j \delta}(W))(t,\cdot) \Vert_{L^\infty(\xR)} 
\leq C\, 2^{-j \delta} \Vert W(t,\cdot) \Vert_{W^{1,\infty}(\xR)}.
$$
Therefore, we obtain 
\bq\label{smoothS:1}
\Vert \big(S_{j\delta}(W)  - S_j(W)\big)\partial_x \Delta_ju \Vert_{L^4(I; H^{s_0-1+\delta}(\xR))} \leq C\,  
\Vert W\Vert_{E_1}  \Vert u_h \Vert_{L^\infty(I; H^{s_0}(\xR))}.
\eq
Similarly, it also holds that
\bq\label{smoothS:2}
\Vert \partial_x  \big(S_{j\delta}(W)  - S_j(W)\big) \Delta_ju \Vert_{L^4(I; H^{s_0-1+\delta}(\xR))} \leq C\,  
\Vert W\Vert_{E_1}  \Vert u_h \Vert_{L^\infty(I; H^{s_0}(\xR))}.
\eq
Now, combining \eqref{TV-D}, \eqref{smoothS:1}, \eqref{smoothS:2} and Lemma \ref{TS} and the fact that $0<\delta<\mez$ we conclude 
\bq\label{est:Fj}
\Vert F_h\Vert_{L^2(I; H^{s_0-1+\delta}(\xR))} \leq C\,
\Vert f_h\Vert_{L^4(I; H^{s_0-1+\delta}(\xR))}+\Vert W\Vert_{E_1}\Vert u_h \Vert_{L^\infty(I; H^{s_0}(\xR))}.
\eq
Now, combining this estimate with \eqref{global4} we derive 
\begin{equation}\label{global5}
\Vert u_h\Vert_{L^4(I; L^\infty(\xR))} h^{-s_0+\frac{1}{4}+\nu} 
  \leq \cF. \Big(
\Vert   f_h \Vert_{L^4(I; H^{s_0-1+\delta}(\xR))} 
+ \Vert u_h\Vert_{L^\infty(I; H^{s_0}(\xR))}\Big).
\end{equation}
Finally, for every given $\eps$ we choose $\delta=\mez-\eps=\mez-\nu$ to get
\[
\lA u\rA_{L^4(I; C_*^{s_0-\frac{1}{4}-\eps}(\xR))}=\sup_h\Vert u_h\Vert_{L^4(I; L^\infty(\xR))} h^{-s_0+\frac{1}{4}+\eps}\le \RHS. 
\]
\end{proof}
\section{Proof of Theorems \ref{theo:Strichartz}, \ref{theo:apriori}, \ref{theo:Cauchy}}
Throughout this section, we assume that $(\eta, \psi)$ is a solution to the gravity-capillary water waves system \eqref{ww} having the regularity given by \eqref{assume:eta, psi}. For any real number $\sigma$, let us define the Sobolev-norm and the Strichartz-norm of the solution:
\begin{align}\label{MN}
&M_\sigma(T)=\Vert (\eta, \psi)\Vert_{L^{\infty}([0, T]; H^{\sigma+\mez}\times H^\sigma)}, \quad M_\sigma (0)=\Vert (\eta, \psi)\arrowvert_{t=0}\Vert_{H^{\sigma+\mez}\times H^\sigma},\\
&N_\sigma(T)=\Vert (\eta, \psi)\Vert_{L^4([0, T]; W^{\sigma+\mez, \infty}\times W^{\sigma, \infty})}.
\end{align}
From the Strichartz estimate \eqref{Strichartz} we have for any $\eps>0$
\bq\label{Strichartz'}
 \Vert  u\Vert_{L^4(I; W^{s - \frac{1}{4}-\eps, \infty})} \leq \cF_\eps\left(\Vert W\Vert_{E_0}+\Vert W\Vert_{E_1}+\lA  \partial_x\eta \rA_{L^\infty_tL^\infty_x}\right) \Big(\Vert f\Vert_{L^4(I; H^{s-\mez})}
+ \Vert u\Vert_{L^\infty(I; H^s)}\Big).
\eq
We shall estimate the norms of $W$ and $u$ appearing on the right-hand side of  \eqref{Strichartz'} in terms of of $M_s$ and $N_s$.
\begin{lemm}\label{est:u-final}
We have 
\[
\lA u\rA_{L^\infty([0, T]; H^s)}\le \cF(M_s(T)).
\]
\end{lemm}
\begin{proof}
By definition \eqref{defi:u}, $u$ is given by
\[
u=\Phi\circ \k-\dot T_{(\partial_x\Phi)\circ \k}\k=\k^*_g \Phi-\R_{line}\Phi.
\]
Lemma \ref{est:u-final} then follows from Theorem \ref{theo:operation} and Theorem \ref{theo:linearization} $(ii)$.
\end{proof}
\begin{lemm}\label{est:W}
We have
\[ 
\Vert W\Vert_{E_0}\le \cF(M_s(T)), \quad \Vert W\Vert_{E_1}\le \cF(M_s(T))(1+N_r(T)).
\]
\end{lemm}
\begin{proof}
Recall from \eqref{defi:W} that $W$ is given by
\[
W=(V\circ \kappa)(\partial_x\chi\circ\kappa)+\partial_t\chi\circ \kappa.
\]
 First, by Sobolev's embedding and Lemma \ref{est:various} 4., $\| W\|_{L^\infty_tL^\infty_x}\le \cF(M_s(T))$. To estimate $\lA W\rA_{E_1}$ we compute
\[
\partial_xW=(\partial_xV\circ \kappa)(\partial_x\chi\circ\kappa)\partial_x\k+(V\circ \kappa)(\partial_x^2\chi\circ \kappa)\partial_x\kappa+(\partial_t\partial_x\chi\circ \kappa)\partial_x\kappa.
\]
Using the expression \eqref{BV} for $V$ together with the H\"older estimate for the Dirichlet-Neumann operator proved in Proposition $2.10$, \cite{NgPo1}, we obtain for a.e. $t\in [0, T]$
\bq\label{dxV}
\lA \partial_xV(t)\rA_{L^{\infty}_x}\le \cF(\lA \eta(t)\rA_{H^{s+\mez}}, \lA \psi(t)\rA_{H^s})\left(1+\lA  \psi(t)\rA_{W^{r, \infty}}\right).
\eq
On the other hand, Lemma \ref{est:various} 3. gives $
\| \partial_x\chi\|_{L^\infty_tL^\infty_x}\le \cF(M_s(T))$, hence 
\bq\label{W(ii):1}
\Vert(\partial_xV\circ \kappa)(\partial_x\chi\circ\kappa)\partial_x\k\Vert_{L^4_tL^{\infty}_x}\le \cF(M_s(T))(1+N_r(T)).
\eq
The other two terms in the expression of $\partial_xW$ are treated in the same way.
\end{proof}
\begin{coro}
For every $0<\mu<\frac{1}{4}$, there exists $\cF:\xR^+\to \xR^+$ such  that
\bq \label{strichartz-norm-u}
 \Vert  u\Vert_{L^4(I; W^{s -\mez+\mu, \infty}(\xR))} \le \cF(M_s(T)+N_r(T)).
\eq
\end{coro}
\begin{proof}
 In views of the Strichartz estimate \eqref{Strichartz'} and Lemma \ref{est:u-final}, Lemma \ref{est:W}, there holds
\bq
 \Vert  u\Vert_{L^4(I; W^{s -\mez+\mu, \infty}(\xR))} \le \cF(M_s(T)+N_r(T))\Big(\Vert f\Vert_{L^4(I; H^{s-\frac{1}{2}}(\xR))}
+ 1\Big).
\eq
On the other hand, from the estimate \eqref{est:RHS} we have
\[
\Vert f\Vert_{L^4(I; H^{s-\mez}(\xR))}\le \cF(M_s(T))(1+N_r(T)),
\]
which concludes the proof.
\end{proof}
Having established the estimate \eqref{strichartz-norm-u} for $u$, we now go back from $u$ to the original unknown $(\eta, \psi)$. To this end, we proceed in 2 steps:
\[
u=k^*\Phi \longrightarrow \Phi \longrightarrow (\eta, \psi).
\] 
Fix $\mu \in (0, \frac{1}{4})$.\\
{\bf Step 1.}  By definition \eqref{defi:u}, $\Phi\circ\kappa=u+\dot T_{\partial_x\Phi\circ\kappa}\kappa$. It is easy to see that 
\[
\lA \dot T_{\partial_x\Phi\circ\kappa}\kappa\rA_{L^\infty_t H^{s+\mez}_x}\le \cF(M_s(T))
\] 
and thus by Sobolev's embedding and the estimate \eqref{strichartz-norm-u}  it holds
\[
\Vert \Phi\circ\kappa\Vert_{L^4(I; W^{s -\mez+\mu, \infty})} \le \cF(M_s(T)+N_s(T)).
\]
 We then may estimate 
\begin{align*}
\Vert \Phi(t)\Vert_{W^{s -\mez+\mu, \infty}}&=\Vert \Phi\circ\kappa\circ\chi(t)\Vert_{W^{s -\mez+\mu, \infty}}\\
& \le \Vert \Phi(t)\circ\kappa(t)\Vert_{W^{s -\mez+\mu, \infty}_x}\cF(\Vert \chi'(t)\Vert_{W^{s -\tdm+\mu, \infty}})\\
&\le  \Vert \Phi(t)\circ\kappa(t)\Vert_{W^{s -\mez+\mu, \infty}}\cF(M_s(T)),
\end{align*}
which implies
\[
\Vert \Phi\Vert_{L^4(I; W^{s -\mez+\mu, \infty})} \le \cF(M_s(T)+N_s(T)).
\]
{\bf Step 2.}
By definition of $\Phi$ and the inequality $\|\cdot\|_{C^\sigma}\le C_\sigma\|\cdot\|_{W^{\sigma, \infty}}$ for any $\sigma >0$, the preceding estimate gives
\bq\label{Teta,U}
\Vert T_p\eta\Vert_{L^4(I; C_*^{s -\mez+\mu})} +\Vert T_q(\psi-T_B\eta)\Vert_{L^4(I; C_*^{s -\mez+\mu})}\le \cF(M_s(T)+N_s(T)).
\eq
1. Since 
\[
\sup_{t\in [0, T]}M^{-1/2}_0(p^{(-1/2)}(t))+\sup_{t\in [0, T]}M^{1/2}_1(p^{(1/2)}(t))\le \cF(M_s(T))
\]
it follows from \eqref{esti:quant1'} that 
\begin{align*}
\Vert T_{p^{(-1/2)}}\eta\Vert_{L^4(I; C_*^{s -\mez+\mu})}
\le \cF(M_s(T))\Vert \eta\Vert_{L^4(I; C_*^{s -1+\mu})}\le \cF(M_s(T)).
\end{align*}
Consequently, we have
\[
\Vert T_{p^{(1/2)}}\eta\Vert_{L^4(I; C_*^{s -\mez+\mu})}
\le \cF(M_s(T)+N_s(T)).
\]
 Since $p^{(1/2)}\in \Gamma^{1/2}_1$ is elliptic, applying \eqref{esti:quant2'} yields $
\eta=T_{1/p^{(1/2)}} T_{p^{(1/2)}}\eta+R\eta$, where $R$ is of order $-1$ and for any $\sigma \in \xR$
\[
\sup_{t\in [0, T]}\|R(t)\|_{C_*^\sigma\to C_*^{\sigma+1}}\le \cF(M_s(T)).
\] 
Thus,
\bq\label{eta:L4}
\Vert \eta\Vert_{L^4(I; C_*^{s +\mu})}\le \cF(M_s(T)+N_r(T)).
\eq
Likewise, we deduce from \eqref{Teta,U} that 
\[
\Vert \psi-T_B\eta \Vert_{L^4(I; C_*^{s -\mez+\mu})}\le \cF(M_s(T)+N_r(T)).
\]
Owing to \eqref{eta:L4} and the fact that $\lA B\rA_{L^\infty_tL^\infty_x}\le \cF(M_s(T))$, we obtain
\[
\Vert \psi\Vert_{L^4(I; C_*^{s -\mez+\mu})}\le \cF(M_s(T)+N_r(T)).
\]
In summary, we have proved that for all $(\eta, \psi)$ solution to \eqref{ww} with
\bq\label{reg:proof}
\left\{
\begin{aligned}
&(\eta, \psi)\in C^0([0, T]; H^{s+\mez}(\xR)\times H^s(\xR))\cap L^4([0, T]; W^{r+\mez,\infty}(\xR)\times W^{r,\infty}(\xR)),\\
&s>r>\tdm+\mez
\end{aligned}
\right.
\eq
there holds for any $\mu<\frac{1}{4}$,
\[
\Vert \eta\Vert_{L^4(I; C_*^{s +\mu})}+\Vert \psi\Vert_{L^4(I; C_*^{s -\mez+\mu})}\le \cF(M_s(T)+N_r(T))
\]
and thus (since $\mu<\frac{1}{4}$ is arbitrary)
\bq\label{strichartz:original}
N_{s-\mez+\mu}(T)\le \cF(M_s(T)+N_r(T)),
\eq
where $M_\sigma(T), N_\sigma(T)$ are respectively the Sobolev-norm and the Strichartz norm defined in \eqref{MN}. \eqref{strichartz:original} is the semi-classical Strichartz estimate announced in Theorem \ref{theo:Strichartz}.\\
 Of course, \eqref{strichartz:original} is meaningful only if $r<s-\mez+\mu$. Under this constrain, using an interpolation argument (see \cite{ABZ4}, page $88$, for instance) we deduce easily that
\[
N_r(T)\le \cF\left(T\big(M_s(T)+N_r(T)\big)\right).
\]
On the other hand,  in Theorem $1.1$ \cite{NgPo1} it was proved the following  energy estimate at the regularity \eqref{reg:proof} 
\[
M_s(T)\le \cF\Big(\cF(M_s(0))+T\cF(M_s(T)+N_r(T))\Big).
\]
Consequently, one gets a closed a priori estimate for the mixed norm
$M_s(T)+N_r(T)$ as in Theorem \ref{theo:apriori}:
\bq\label{est:apriori}
M_s(T)+N_r(T)\le \cF\Big(\cF(M_s(0))+T\cF(M_s(T)+N_r(T))\Big).
\eq
Finally, by virtue of the contraction estimate for two solution $(\eta_j, \psi_j)$ $j=1,2$ in the norm $M_{s-1, T}+N_{r-1, T}$ established in Theorem $5.9$, \cite{NgPo1} (whose proof makes use of Theorem \ref{strich2}) one can use the standard regularized argument (see section $6$, \cite{NgPo1}) to solve uniquely the Cauchy problem for system \eqref{ww} with initial data  $(\eta_0, \psi_0)\in H^{s+\mez}(\xR)\times H^s(\xR)$ with $s>2+\mez-\mu$ for any $\mu<\frac{1}{4}$. The proof of Theorem \ref{theo:Cauchy} is complete.
\appendix
\section{Appendix 1: Paradifferential calculus}
\begin{defi}\label{spaces}
1. (Zygmund spaces) Let 
\[
1=\sum_{p=0}^\infty\Delta_p
\]
be a Littlewood-Paley partition. For any real number~$s$, we define the Zygmund class~$C^{s}_*({\mathbf{R}}^d)$ as the 
space of tempered distributions~$u$ such that
$$
\lA u\rA_{C^{s}_*}\defn \sup_q 2^{qs}\lA \Delta_q u\rA_{L^\infty}<+\infty.
$$
2. (H\"older spaces) For~$k\in\xN$, we denote by $W^{k,\infty}({\mathbf{R}}^d)$ the usual Sobolev spaces.
For $\rho= k + \sigma$, $k\in \xN, \sigma \in (0,1)$ denote 
by~$W^{\rho,\infty}({\mathbf{R}}^d)$ 
the space of functions whose derivatives up to order~$k$ are bounded and uniformly H\"older continuous with 
exponent~$\sigma$. 
\end{defi}
Let us review notations and results about Bony's paradifferential calculus (see \cite{Bony, MePise}). Here we follow the presentation by M\'etivier in \cite{MePise} and \cite{ABZ3}.
\begin{defi}
1. (Symbols) Given~$\rho\in [0, \infty)$ and~$m\in\xR$,~$\Gamma_{\rho}^{m}({\mathbf{R}}^d)$ denotes the space of
locally bounded functions~$a(x,\xi)$
on~${\mathbf{R}}^d\times({\mathbf{R}}^d\setminus 0)$,
which are~$C^\infty$ with respect to~$\xi$ for~$\xi\neq 0$ and
such that, for all~$\alpha\in\xN^d$ and all~$\xi\neq 0$, the function
$x\mapsto \partial_\xi^\alpha a(x,\xi)$ belongs to~$W^{\rho,\infty}({\mathbf{R}}^d)$ and there exists a constant
$C_\alpha$ such that,
\begin{equation*}
\forall\la \xi\ra\ge \mez,\quad 
\lA \partial_\xi^\alpha a(\cdot,\xi)\rA_{W^{\rho,\infty}(\xR^d)}\le C_\alpha
(1+\la\xi\ra)^{m-\la\alpha\ra}.
\end{equation*}
Let $a\in \Gamma_{\rho}^{m}({\mathbf{R}}^d)$, we define for every $n\in \xN$ the semi-norm
\begin{equation}\label{defi:norms}
M_{\rho}^{m}(a; n)= 
\sup_{\la\alpha\ra\le n~}\sup_{\la\xi\ra \ge 1/2~}
\lA (1+\la\xi\ra)^{\la\alpha\ra-m}\partial_\xi^\alpha a(\cdot,\xi)\rA_{W^{\rho,\infty}({\mathbf{R}}^d)}.
\end{equation}
When $n=[d/2]+1$ we denote $M_{\rho}^{m}(a; n)=M_{\rho}^{m}(a)$.\\
2. (Classical symbols) For any $m\in \xR$ and $\rho>0$ we denote by  $\Sigma_\rho^m(\xR^d)$ the class of classical symbols $a(x, \xi)$ such that 
\[
a(x, \xi)=\sum_{0\le j\le[\rho]}a^{(m-j)}
\]
where each $a^{(m-j)}\in \Gamma^{m-j}_{\rho-j}$ is homogeneous of degree $m-j$ with respect to $\xi$.
\end{defi}
\begin{defi}\label{defi:oper}
 (Paradifferential operators) Given a symbol~$a$, we define
the paradifferential operator~$T_a$ by
\begin{equation}\label{eq.para}
\widehat{T_a u}(\xi)=(2\pi)^{-d}\int \chi(\xi-\eta,\eta)\widehat{a}(\xi-\eta,\eta)\psi(\eta)\widehat{u}(\eta)
\, d\eta,
\end{equation}
where
$\widehat{a}(\theta,\xi)=\int e^{-ix\cdot\theta}a(x,\xi)\, dx$
is the Fourier transform of~$a$ with respect to the first variable; 
$\chi$ and~$\psi$ are two fixed~$C^\infty$ functions such that:\\
$(i)$ $\psi$ is identical to $0$ near the origin and identical to $1$ away from the origin,\\
$(ii)$ there exists $0<\eps_1<\eps_2<1$ such that 
\bq\label{cutoff-chi:1}
\chi (\eta, \xi)=
\begin{cases}
1\quad &\text{if}~|\eta|\le \eps_1(1+|\xi|),\\
0\quad &\text{if}~|\eta|\ge \eps_2(1+|\xi|)
\end{cases}
\eq
and for any $(\alpha, \beta)\in \xN^2$ there exists $C_{\alpha, \beta}>0$ such that 
\bq\label{cutoff-chi:2}
\forall (\eta, \xi)\in \xR^d\times \xR^d, \la \partial_\eta^\alpha\partial_\xi^\beta \chi(\eta, \xi)\ra \le C_{\alpha, \beta}(1+|\xi|)^{-\alpha-\beta}.
\eq
\end{defi}
\begin{defi}\label{defi:order}
An operator~$T$ is said to be of  order~$\leo m\in \xR$ (or equivalently, $-m$-regularized) if, for all~$\mu\in\xR$,
it is bounded from~$H^{\mu}$ to~$H^{\mu-m}$ and from $C^\mu_*$ to $C^{\mu-m}_*$. 
\end{defi}
Symbolic calculus for paradifferential operators is summarized in the following theorem.
\begin{theo}\label{theo:sc}(Symbolic calculus, \cite{MePise})
Let~$m\in\xR$ and~$\rho\in [0, \infty)$. Denote by $\overline{\rho}$ the smallest integer that is not smaller than $\rho$ and $n_1=[d/2]+\overline\rho+1$.\\
$(i)$ If~$a \in \Gamma^m_0({\mathbf{R}}^d)$, then~$T_a$ is of order~$\leo m$. 
Moreover, for all~$\mu\in\xR$ there exists a constant~$K$ such that
\begin{gather}\label{esti:quant1}\lA T_a \rA_{H^{\mu}\rightarrow H^{\mu-m}}\le K M_{0}^{m}(a),\\ \label{esti:quant1'}
 \lA T_a \rA_{C_*^{\mu}\rightarrow C_*^{\mu-m}}\le K M_{0}^{m}(a).
\end{gather}
$(ii)$ If~$a\in \Gamma^{m}_{\rho}({\mathbf{R}}^d), b\in \Gamma^{m'}_{\rho}({\mathbf{R}}^d)$ with $\rho>0$. Then 
$T_a T_b -T_{a \sharp b}$ is of order~$\leo m+m'-\rho$ where
\[
a\sharp b:=\sum_{|\alpha|<\rho}\frac{(-i)^{\alpha}}{\alpha !}\partial_{\xi}^{\alpha}a(x, \xi)\partial_x^{\alpha}b(x, \xi).
\] 
Moreover, for all~$\mu\in\xR$ there exists a constant~$K$ such that
\begin{gather}\label{esti:quant2}
\lA T_a T_b  - T_{a \sharp b}   \rA_{H^{\mu}\rightarrow H^{\mu-m-m'+\rho}}
\le 
K M_{\rho}^{m}(a; n_1)M_{0}^{m'}(b)+K M_{0}^{m}(a)M_{\rho}^{m'}(b; n_1),\\  \label{esti:quant2'}
\lA T_a T_b  - T_{a \sharp b}   \rA_{C_*^{\mu}\rightarrow C_*^{\mu-m-m'+\rho}}
\le 
K M_{\rho}^{m}(a; n_1)M_{0}^{m'}(b)+K M_{0}^{m}(a)M_{\rho}^{m'}(b; n_1).
\end{gather}
$(iii)$ Let~$a\in \Gamma^{m}_{\rho}({\mathbf{R}}^d)$ with $\rho >0$. Denote by 
$(T_a)^*$ the adjoint operator of~$T_a$ and by~$\overline{a}$ the complex conjugate of~$a$. Then 
$(T_a)^* -T_{a^*}$ is of order~$\leo m-\rho$ where
\[
a^*=\sum_{|\alpha|<\rho}\frac{1}{i^{|\alpha|}\alpha!}\partial_{\xi}^{\alpha}\partial_x^{\alpha}\overline{a}.
\]
Moreover, for all~$\mu$ there exists a constant~$K$ such that
\begin{gather}\label{esti:quant3}
\lA (T_a)^*   - T_{\overline{a}}   \rA_{H^{\mu}\rightarrow H^{\mu-m+\rho}}\le 
K M_{\rho}^{m}(a; n_1),\\ \label{esti:quant3'}
\lA (T_a)^*   - T_{\overline{a}}   \rA_{C_*^{\mu}\rightarrow C_*^{\mu-m+\rho}}\le 
K M_{\rho}^{m}(a; n_1).
\end{gather}
\end{theo}
\begin{defi}(Paraproducts and Bony's decomposition)
Let $1=\sum_{j=0}^\infty \Delta_j$ be a dyadic partition of unity as in \eqref{partition} and $N\in \xN$ be sufficiently large such that the function $\chi$ defined in \eqref{choose:chi}:
\[
\chi(\eta, \xi)=\sum_{p= 0}^\infty \phi_{p-N}(\eta)\varphi_p(\xi)
\]
satisfies conditions \eqref{cutoff-chi:1} and \eqref{cutoff-chi:2}.\\
Given $a,~b\in \mathcal{S}'$ we define formally the paraproduct 
\bq\label{def:paraproduct}
TP_au=\sum_{p=N+1}^\infty S_{p-N}a\Delta_pu
\eq
 and the remainder 
\bq\label{R:Bony}
R(a,u)=\sum_{ j, k\ge 0, |j-k|\le N-1}\Delta_j a\Delta_k u
\eq
then we have (at least formally) the Bony's decomposition
\[
au=TP_au+TP_ua+R(a, u).
\]
\end{defi}
We shall use frequently various estimates about paraproducts (see Chapter 2, ~\cite{BCD} and \cite{ABZ3}) which are recalled here.
\begin{theo}\label{pproduct}
\begin{enumerate}
\item  Let~$\alpha,\beta\in \xR$. If~$\alpha+\beta>0$ then
\begin{align}
&\lA R(a,u) \rA _{H^{\alpha + \beta-\frac{d}{2}}({\mathbf{R}}^d)}
\leq K \lA a \rA _{H^{\alpha}({\mathbf{R}}^d)}\lA u\rA _{H^{\beta}({\mathbf{R}}^d)},\label{Bony1} \\ 
&\lA R(a,u) \rA _{H^{\alpha + \beta}({\mathbf{R}}^d)} \leq K \lA a \rA _{C^{\alpha}_*({\mathbf{R}}^d)}\lA u\rA _{H^{\beta}({\mathbf{R}}^d)}\label{Bony2},\\
&\lA R(a,u) \rA _{C_*^{\alpha + \beta}({\mathbf{R}}^d)}\leq K \lA a \rA _{C^{\alpha}_*({\mathbf{R}}^d)}\lA u\rA _{C_*^{\beta}({\mathbf{R}}^d)}.\label{Bony3}
\end{align}
\item Let~$s_0,s_1,s_2$ be such that 
$s_0\le s_2$ and~$s_0 < s_1 +s_2 -\frac{d}{2}$, 
then
\begin{equation}\label{boundpara}
\lA TP_a u\rA_{H^{s_0}}\le K \lA a\rA_{H^{s_1}}\lA u\rA_{H^{s_2}}.
\end{equation}
\item  Let~$m>0$ and~$s\in \xR$. Then
\begin{align}
&\lA TP_a u\rA_{H^{s-m}}\le K \lA a\rA_{C^{-m}_*}\lA u\rA_{H^{s}}\label{pest1},\\ 
&\lA TP_a u\rA_{C_*^{s-m}}\le K \lA a\rA_{C^{-m}_*}\lA u\rA_{C_*^{s}}\label{pest2}.
\end{align}
\end{enumerate}
\end{theo}
\begin{prop}\label{productrule}
\begin{enumerate}
\item  If~$u_j\in H^{s_j}({\mathbf{R}}^d)$ ($j=1,2$) with $s_1+s_2> 0$ then 
\begin{equation}\label{pr}
\lA u_1 u_2 \rA_{H^{s_0}}\le K \lA u_1\rA_{H^{s_1}}\lA u_2\rA_{H^{s_2}},
\end{equation}
if $s_0\le s_j$, $j=1,2$, and $s_0< s_1+s_2-d/2$.
\item If $s\ge 0$ then 
\bq\label{tame:S}
\lA u_1u_2\rA_{H^s}\le K(\lA u_1\rA_{H^s}\lA u_2\rA_{L^{\infty}}+\lA u_2\rA_{H^s}\lA u_1\rA_{L^{\infty}}).
\eq
\item If $s\ge 0$ then 
\bq\label{tame:H}
\lA u_1u_2\rA_{C_*^s}\le K(\lA u_1\rA_{C_*^s}\lA u_2\rA_{L^{\infty}}+\lA u_2\rA_{C_*^s}\lA u_1\rA_{L^{\infty}}).
\eq
\item Let $\beta>\alpha> 0$. Then
\bq\label{tame:H<0}
\lA u_1u_2\rA_{C_*^{-\alpha}}\le K\lA u_1\rA_{C_*^{\beta}}\lA u_2\rA_{C_*^{-\alpha}}.
\eq
\end{enumerate}
\end{prop}
\begin{theo}
\begin{enumerate}
\item Let~$s\ge 0$ and consider~$F\in C^\infty(\xC^N)$ such that~$F(0)=0$. 
Then there exists a non-decreasing function~$\mathcal{F}\colon\xR_+\rightarrow\xR_+$ 
such that, for any~$U\in H^s({\mathbf{R}}^d)^N\cap L^\infty(\xR^d)$,
\begin{equation}\label{F(u):H}
\lA F(U)\rA_{H^s}\le \mathcal{F}\bigl(\lA U\rA_{L^\infty}\bigr)\lA U\rA_{H^s}.
\end{equation}
\item Let~$s\ge 0$ and consider~$F\in C^\infty(\xC^N)$ such that~$F(0)=0$. 
Then there exists a non-decreasing function~$\mathcal{F}\colon\xR_+\rightarrow\xR_+$ 
such that, for any~$U\in C_*^s({\mathbf{R}}^d)^N$,
\begin{equation}\label{F(u):C}
\lA F(U)\rA_{C_*^s}\le \mathcal{F}\bigl(\lA U\rA_{L^\infty}\bigr)\lA U\rA_{C_*^s}.
\end{equation}
\end{enumerate}
\end{theo}
\begin{theo}\protect{\cite[Theorem~2.92]{BCD}}\label{paralin}(Paralinearization)
Let $r,~\rho$ be positive real numbers and $F$ be a $C^{\infty}$ function on $\xR$ such that $F(0)=0$. Assume that $\rho$ is not an integer. For any $u\in H^{\mu}(\xR^d)\cap C_*^{\rho}(\xR^d)$ we have
\[
\lA F(u)-TP_{F'(u)}u\rA_{ H^{\mu+\rho}(\xR^d)}\le C(\lA u\rA_{L^{\infty}(\xR^d)})\lA u\rA_{C_*^{\rho}(\xR^d)}\lA u\rA_{H^{\mu}(\xR^d)}.
\]
\end{theo}
\section{Appendix 2}\label{Appendix2}
\subsection{Proof of Lemma \ref{lem1:phi}}
Let $f_n\in C(\xR^d)$,  $g\in C^\infty(\xR^d)$  be two nonnegative functions satisfying
\[
f_n(t)=
\begin{cases}
1,~\text{if}~|t|\le 2^{-n}+\frac{1}{4},\\
0,~\text{if}~|t|> 2^{n+1}-\frac{1}{4}
\end{cases}
\]
and
\[
 g(t)=0,~\text{if}~|t|\ge \frac{1}{4},\quad \int_{\xR^d}g(t)\di t=1.
\]
We then define $\phi_{(n)}=f_n*g$. It is easy to see that $\phi_{(n)}\ge 0$ and satisfies condition \eqref{phi:cd1}. To verify condition \eqref{phi:cd2} we use $\partial^\alpha\phi_{(n)}=f_n*\partial^\alpha g$ to have
\[
\begin{aligned}
 x^\beta \partial^\alpha\phi_{(n)}(x)&=\int_{\xR^d}  x^\beta f_n(x-y)\partial^\alpha g(y)\di y\\
&=\sum_{\beta_1+\beta_2=\beta}\int_{\xR^d}  (x-y)^{\beta_1} f_n(x-y)y^{\beta_2}\partial^\alpha g(y)\di y,\\
&=\sum_{\beta_1+\beta_2=\beta}\left( (\cdot)^{\beta_1}f_n\right)*\left( (\cdot)^{\beta_2}\partial^\alpha g\right)(x).
\end{aligned}
\] 
Each term on the right-hand side is estimated by
\[
\lA \left( (\cdot)^{\beta_1}f_n\right)*\left( (\cdot)^{\beta_2}\partial^\alpha g\right)\rA_{L^1}\le \lA (\cdot)^{\beta_1}f_n\rA_{L^1}\lA (\cdot)^{\beta_2}\partial^\alpha g\rA_{L^1}
\]
 where $\lA (\cdot)^{\beta_2}\partial^\alpha g\rA_{L^1}$ is independent of $n$. It remains to have a uniformly bound  with respect to $n$ for $\lA (\cdot)^{\beta_1}f_n\rA_{L^1}$. To this end, one can choose the following piecewise affine functions
\[
f_n(t)=
\begin{cases}
1,~\text{if}~|t|\le 2^{-n}+\frac{1}{4},\\
0,~\text{if}~|t|> 2^{-n}+\frac{1}{2},\\
-4(|t|-2^{-n}-\mez), ~\text{if}~2^{-n}+\frac{1}{4}\le |t|\le 2^{-n}+\frac{1}{2}.
\end{cases}
\]
\subsection{Proof of Lemma \ref{boundSj}}
\hk 1. Let $1\le p\le q\le \infty$. Remark first that the estimates for $\Delta_j$ follows immediately from those of $S_j$ since $\Delta_0=S_0$ and $\Delta_j=S_j-S_{j-1},~\forall j\ge 1$. By definition \ref{SD} we have for each $n\in \xN$, $S_ju=f_{j}*u$ where $f_{j}$ is the inverse Fourier transform of $\phi_j$, where $\phi\equiv \phi_{(n)}$. With $r$ satisfying 
\[
\frac{1}{p}+\frac{1}{r}=1+\frac{1}{q}
\]
 we get by Young's inequality
\[
\lA \partial^\alpha S_j\rA_{L^p\to L^q}\le \lA \partial^\alpha f_j\rA_{L^r}.
\]
The problem then reduces to showing that  
\[
\lA \partial^\alpha f_j\rA_{L^r}\le C_\alpha 2^{j(|\alpha|+\frac{d}{p}-\frac{d}{q})}
\]
which in turn reduces to 
\[
\lA \partial^\alpha \mathfrak{F}^{-1}(\phi_{(n)})(x)\rA_{L^r}\le C_\alpha,
\]
which is true by virtue of \eqref{phi:cd2}.\\
\hk 2. The boundedness of the operators $2^{j\mu}\Delta_j,~j\ge 1$ from $W^{\mu, \infty}(\xR^d)$ to $L^\infty(\xR^d)$ is proved in Lemma $4.1.8$, \cite{MePise}. Following that proof we see that 
\[
\lA2^{j\mu} \Delta_j\rA_{W^{\mu,\infty}\to L^\infty}\le 2^{j\mu} \int_{\xR^d}|x|^\mu |g_j(x)|dx:=I,
\]
where  $g_j$ is the inverse Fourier transform of $\varphi_j=\phi_j-\phi_{j-1}$. Owing to \eqref{phi:cd2} it holds that
\[
\forall \alpha\in \xN^d,~\exists C_\alpha>0, \forall (j, n)\in \xN^*\times \xN, \int |x^\alpha g_j(x)|\di x\le C_\alpha 2^{-j|\alpha|}. 
\]
Thus, if $\mu\in \xN$ we have the result. If $\mu=\delta n+(1-\delta)(n+1)$ for some $\delta\in (0, 1),~n\in \xN$ we use H\"older's inequality to estimate
\[
I\le 2^{j\mu}\left(\int |x|^n|g_j(x)|\di x\right)^\delta \left(\int |x|^{n+1}|g_j(x)|\di x\right)^{1-\delta}\le C_\mu 2^{j\mu}2^{-jn\delta-j(n+1)(1-\delta)} =C_\mu,
\]
which concludes the proof.


\begin{thebibliography}{10}
\small
\bibitem{ABZ1}
Thomas Alazard, Nicolas Burq, and Claude Zuily.
\newblock On the water waves equations with surface tension.
\newblock {\em Duke Math. J.}, 158(3):413--499, 2011.

\bibitem{ABZ2}
Thomas Alazard, Nicolas Burq, and Claude Zuily.
\newblock Strichartz estimates for water waves.
\newblock {\em Ann. Sci. {\'E}c. Norm. Sup{\'e}r. (4)}, 44(5):855--903, 2011.

\bibitem{ABZ3}
Thomas Alazard, Nicolas Burq, and Claude Zuily.
\newblock On the Cauchy problem for gravity water waves.
\newblock {\em Invent.Math.}, 198(1): 71--163, 2014.

\bibitem{ABZ4}
Thomas Alazard, Nicolas Burq, and Claude Zuily.
\newblock Strichartz estimate and the Cauchy problem for the gravity water waves equations.
\newblock{\em  arXiv:1404.4276}, 2014.

\bibitem{AM}
Thomas Alazard and Guy M{\'e}tivier.
\newblock Paralinearization of the {D}irichlet to {N}eumann operator, and
  regularity of three-dimensional water waves.
\newblock {\em Comm. Partial Differential Equations}, 34(10-12):1632--1704,
  2009.

\bibitem{Alipara}
Serge Alinhac.
\newblock Paracomposition et op\'erateurs paradiff\'erentiels.
\newblock {\em Comm. Partial Differential Equations}, 11(1):87--121, 1986.

\bibitem{BaCh}
Hajer Bahouri and Jean-Yves Chemin.
\newblock \'{E}quations d'ondes quasilin\'eaires et estimations de
  {S}trichartz.
\newblock {\em Amer. J. Math.}, 121(6):1337--1377, 1999.

\bibitem{BCD}
Hajer Bahouri, Jean-Yves Chemin, and Rapha{\"e}l Danchin.
\newblock {\em Fourier analysis and nonlinear partial differential equations},
  volume 343 of {\em Grundlehren der Mathematischen Wissenschaften [Fundamental
  Principles of Mathematical Sciences]}.
\newblock Springer, Heidelberg, 2011.

\bibitem{Bony}
Jean-Michel Bony.
\newblock Calcul symbolique et propagation des singularit\'es pour les
  \'equations aux d\'eriv\'ees partielles non lin\'eaires.
\newblock {\em Ann. Sci. \'Ecole Norm. Sup. (4)}, 14(2):209--246, 1981.

\bibitem{BGT1}
Nicolas Burq, Patrick G\'erard, and Nikolay Tzvetkov.
\newblock Strichartz inequalities and the nonlinear {S}chr\"odinger equation on
  compact manifolds.
\newblock {\em Amer. J. Math.}, 126(3):569--605, 2004.


\bibitem{CMSW}
Robin~Ming Chen, Jeremy~L. Marzuola, Daniel Spirn, and J.~Douglas Wright.
\newblock On the regularity of the flow map for the gravity-capillary
  equations.
\newblock {\em J. Funct. Anal.}, 264(3):752--782, 2013.

\bibitem{CHS}
Hans Christianson, Vera~Mikyoung Hur, and Gigliola Staffilani.
\newblock Strichartz estimates for the water-wave problem with surface tension.
\newblock {\em Comm. Partial Differential Equations}, 35(12):2195--2252, 2010.

\bibitem{ChLi}
Demetrios Christodoulou and Hans Lindblad.
\newblock On the motion of the free surface of a liquid.
\newblock {\em Comm. Pure Appl. Math.}, 53(12):1536--1602, 2000.

\bibitem{CS}
Daniel Coutand and Steve Shkoller.
\newblock Well-posedness of the free-surface incompressible {E}uler equations
  with or without surface tension.
\newblock {\em J. Amer. Math. Soc.}, 20(3):829--930 (electronic), 2007.

\bibitem{Craig1985}
Walter Craig.
\newblock {An existence theory for water waves and the Boussinesq and
  Korteweg-deVries scaling limits}.
\newblock {\em Communications in Partial Differential Equations},
  10(8):787--1003, 1985.

\bibitem{CrSu}
Walter Craig and Catherine Sulem. 
\newblock Numerical simulation of gravity waves. 
\newblock {\em J. Comput. Phys.} 108(1):73--83, 1993.

\bibitem{GrSJ}
Alain Grigis and Johannes Sjostrand.
\newblock {\em Microlocal analysis for pseudo-differential operators.}
\newblock London Mathematical Society Lecture Note Series, 1994.

\bibitem{HuIfTa}
John Hunter, Mihaela Ifrim, Daniel Tataru.
\newblock Two dimensional water waves in holomorphic coordinates.
\newblock{\em  arXiv:1401.1252v2}, 2014.

\bibitem{NgPo1}
Quang-Huy Nguyen and Thibault de Poyferr\'{e}.
\newblock A paradifferential reduction for the gravity-capillary waves system at low regularity and applications.
\newblock{\em arXiv:1508.00326}, 2015.

\bibitem{NgPo2}
Quang-Huy Nguyen and Thibault de Poyferr\'{e}.
\newblock Strichartz estimates and local existence for the capillary water waves with non-Lipschitz initial velocity.
\newblock{\em arXiv:1507.08918}, 2015.

\bibitem{Huy}
Quang-Huy Nguyen.
\newblock Sharp Strichartz estimates for water waves systems.
\newblock{\em   arXiv:1512.02359}

\bibitem{LannesJAMS}
David Lannes.
\newblock Well-posedness of the water-waves equations.
\newblock {\em J. Amer. Math. Soc.}, 18(3):605--654 (electronic), 2005.

\bibitem{LannesLivre}
David Lannes
\newblock {\em Water waves: mathematical analysis and asymptotics.}
\newblock Mathematical Surveys and Monographs, 188. American Mathematical Society, Providence, RI, 2013. 


\bibitem{Lebeau92}
Gilles Lebeau.
\newblock Singularit\'es des solutions d'\'equations d'ondes semi-lin\'eaires.
\newblock {\em Ann. Sci. \'Ecole Norm. Sup. (4)}, 25(2):201--231, 1992.

\bibitem{LindbladAnnals}
Hans Lindblad.
\newblock Well-posedness for the motion of an incompressible liquid with free
  surface boundary.
\newblock {\em Ann. of Math. (2)}, 162(1):109--194, 2005.

\bibitem{MR}
Nader Masmoudi and Fr{\'e}d{\'e}ric Rousset.
\newblock Uniform regularity and vanishing viscosity limit for the free surface Navier-Stokes equations. 
\newblock arXiv:1202.0657.

\bibitem{MePise}
Guy M{\'e}tivier.
\newblock {\em Para-differential calculus and applications to the {C}auchy
  problem for nonlinear systems}, volume~5 of {\em Centro di Ricerca Matematica
  Ennio De Giorgi (CRM) Series}.
\newblock Edizioni della Normale, Pisa, 2008.

\bibitem{MiZh}
Mei Ming and Zhifei Zhang.
\newblock Well-posedness of the water-wave problem with surface tension.
\newblock{\em J. Math. Pures Appl.} (9) 92, no. 5, 429--455, 2009. 

\bibitem{Nalimov}
V.~I. Nalimov.
\newblock The {C}auchy-{P}oisson problem.
\newblock {\em Dinamika Splo\v sn. Sredy}, (Vyp. 18 Dinamika Zidkost. so 
Svobod. Granicami):104--210, 254, 1974.


\bibitem{SZ1}
Jalal Shatah and Chongchun Zeng.
\newblock Geometry and a priori estimates for free boundary problems of the
  {E}uler equation.
\newblock {\em Comm. Pure Appl. Math.}, 61(5):698--744, 2008.

\bibitem{SZ2}
Jalal Shatah and Chongchun Zeng.
\newblock  A priori estimates for fluid interface problems.
\newblock{\em Comm. Pure Appl. Math.} 61(6): 848–876, 2008.

\bibitem{SZ3}
Jalal Shatah and Chongchun Zeng.
 \newblock Local well-posedness for fluid interface problems. 
\newblock{\em Arch. Ration. Mech. Anal.} 199(2): 653–705, 2011.

\bibitem{StTa}
Gigliola Staffilani and Daniel Tataru.
\newblock Strichartz estimates for a {S}chr\"odinger operator with nonsmooth
  coefficients.
\newblock {\em Comm. Partial Differential Equations}, 27(7-8):1337--1372, 2002.


\bibitem{TataruNS}
Daniel Tataru.
\newblock Strichartz estimates for operators with nonsmooth coefficients and
  the nonlinear wave equation.
\newblock {\em Amer. J. Math.}, 122(2):349--376, 2000.

\bibitem{TataruNSII}
Daniel Tataru.
\newblock Strichartz estimates for second order hyperbolic operators with
  nonsmooth coefficients. {II}.
\newblock {\em Amer. J. Math.}, 123(3):385--423, 2001.

\bibitem{Taylor}
Michael~E. Taylor.
\newblock {\em Pseudo-differential operators and nonlinear {PDE}}, volume 100 of
  {\em Progress in Mathematics}.
\newblock Birkh\"auser Boston Inc., Boston, MA, 1991.

\bibitem{WuInvent}
Sijue Wu.
\newblock Well-posedness in {S}obolev spaces of the full water wave problem in
  2-{D}.
\newblock {\em Invent. Math.}, 130(1):39--72, 1997.

\bibitem{WuJAMS}
Sijue Wu.
\newblock Well-posedness in {S}obolev spaces of the full water wave problem in
  3-{D}.
\newblock {\em J. Amer. Math. Soc.}, 12(2):445--495, 1999.

\bibitem{Yosihara}
Hideaki Yosihara.
\newblock Gravity waves on the free surface of an incompressible perfect fluid
  of finite depth.
\newblock {\em Publ. Res. Inst. Math. Sci.}, 18(1):49--96, 1982.
\end{thebibliography}
\end{document}